\documentclass[11pt]{amsart}
\usepackage{a4wide,amsfonts,graphicx,
amssymb,stmaryrd,amscd,amsmath,latexsym,amsbsy,amsthm}
\usepackage[all,knot,poly]{xy}
\usepackage[all,knot,poly]{xy}
\usepackage{color}
\usepackage{hyperref}
\usepackage{ifpdf}
\ifpdf 
  \usepackage{graphicx}
  \DeclareGraphicsExtensions{.pdf,.png,.mps}
  \usepackage{pgf}
  \usepackage{tikz}
\else 
  \usepackage{graphicx}
  \DeclareGraphicsExtensions{.eps,.bmp}
  \usepackage{pgf}
  \usepackage{tikz}
  \usepackage{pstricks}
\fi
\usepackage{epic,bez123}
\usepackage{wrapfig}
\usepackage{multirow}
\newtheorem{thm}{Theorem}[section]

\newtheorem{cor}[thm]{Corollary}
\newtheorem{lem}[thm]{Lemma}
\newtheorem{prop}[thm]{Proposition}
\newtheorem{defn}[thm]{Definition}

\newtheorem{rem}[thm]{Remark}

\newcommand{\mb}{\mathbf}

\newcommand{\mc}{\mathcal}
\newcommand{\mf}{\mathfrak}

\newcommand{\Lb}{\mathbf{L}}

\newcommand{\ep}{\epsilon}

\begin{document}
\title[Spectral transfer morphisms]
{Spectral transfer morphisms for unipotent affine Hecke algebras}
\author{Eric Opdam}
\address{Korteweg de Vries Institute for Mathematics\\
University of Amsterdam\\
Science Park 107\\
1018TV Amsterdam\\
The Netherlands\\
email: e.m.opdam@uva.nl}
\date{\today}

\thanks{This research was supported  by 
ERC-advanced grant no. 268105. It is a pleasure to thank Joseph Bernstein, 
Dan Ciubotaru, David Kazhdan and Mark Reeder for useful discussions and 
comments.}
\begin{abstract}
We classify the spectral transfer morphisms (cf. \cite{Opd4}) 
between affine Hecke algebras associated to the unipotent types of the various inner 
forms of a given quasi-split, unramified absolutely simple algebraic group $G$ defined 
over a non-archimidean local field $k$. 
This characterizes Lusztig's classification \cite{Lu4}, \cite{Lu6} 
of unipotent 
characters of $G$ in terms of the Plancherel measure, up to diagram automorphisms.  
As an application of these results, the spectral correspondences associated with such 
morphisms \cite{Opd4}, and some results of Ciubotaru, Kato and Kato [CKK] we prove 
a conjecture of 
Hiraga, Ichino and 
Ikeda [HII] on the formal degrees and adjoint gamma factors for all unipotent discrete series 
characters of unramified simple groups of adjoint type defined over $k$.
\end{abstract}
\maketitle
\tableofcontents
\section{Introduction}\label{sec:intro}
Recall from \cite{Opd4} that a normalized affine Hecke algebra $\mathcal{H}$  is essentially 
determined by a complex torus $T$ and a meromorphic function $\mu$ on $T$.  
A spectral transfer morphism (see \cite{Opd4}) 
$\phi:\mathcal{H}_1\leadsto\mathcal{H}_2$ between normalized affine Hecke algebras expresses 
the fact that $\mu_1$ is equal to a residue of $\mu_2$ along a certain 
coset of a subtorus of $T_2$. This turns out to be a convenient tool to compare 
formal degrees 
of discrete series representations of different affine Hecke algebras. 

The notion is based on the special properties of the $\mu$-function of 
an affine Hecke algebra \cite{Opd1}, \cite{Opd3} which are intimitely related 
to its basic role in the derivation of the Plancherel
formula for affine Hecke algebras via residues \cite{HOH}, \cite{Opd1}, \cite{OpdSol}. 
This approach to the computation of formal degrees has its origin in the theory of 
spherical functions for $p$-adic reductive groups \cite{Ma1}, and was further inspired 
by early observations of Lusztig \cite{Lu0}, \cite{Lu10} and of Reeder \cite{Re0}, \cite{Re}
on the behaviour of formal degrees within unipotent $L$-packets. 

In the present paper we classify the spectral transfer morphisms 
(STMs in the sequel) between the unipotent affine Hecke algebras of the various inner 
forms of a given 
quasi-split absolutely simple algebraic group $G$ of adjoint type, defined and unramified over a 
non-archimidean local  field $\bf{k}$. 
In particular we will show, for any unipotent 
type $\tau=(\mathbb{P},\sigma)$ of an inner form of $G$,  \emph{existence} and \emph{uniqueness} 
(up to diagram automorphisms) 
of such STM of the Hecke algebra of 
$\tau$ to the Iwahori Hecke algebra $\mathcal{H}^{IM}(G)$ of $G$.
The STMs of this kind turn out to correspond exactly to the arithmetic-geometric
correspondences of Lusztig \cite{Lu4}, \cite{Lu6}. 

As application of this classification, using the basic properties of STMs discussed in 
\cite{Opd4}, we prove
the conjecture 
\cite[Conjecture 1.4]{HII} of Hiraga, Ichino and Ikeda expressing the formal degree of a 
discrete series representation in terms of the adjoint gamma factor of its (conjectural) 
local Langlands parameters and an explicit rational constant factor, 
for all unipotent discrete series representations of inner forms of $G$ 
(where we accept Lusztig's parameters for the unipotent discrete series representations as 
conjectural Langlands parameters). It should be mentioned that it was already known 
from Reeder's work \cite{Re0}, \cite{Re} (see also \cite{HII}) that this conjecture holds for 
the unipotent discrete series characters of split exceptional groups of adjoint type, and 
for some small rank classical groups. It should be mentioned that the stability of Lusztig's 
packets of unipotent representations was shown by Moeglin and Waldspurger 
for odd orthogonal groups \cite{MW} and by Moeglin for unitary groups \cite{Moe}.

Throughout this paper we use the normalization of Haar measures
as in \cite{DeRe}.   
Let $q=v^2$ denote the cardinality of the residue field of $\bf{k}$. 
The formal degree of a unipotent discrete series representation then factorizes 
uniquely as a product of a $q$-rational number (which we define as a fraction of products of 
$q$-numbers of the form $[n]_q:=\frac{(v^n-v^{-n})}{v-v^{-1}}$ with $n\geq 2$) and a positive 
rational number.
Our proof of conjecture \cite[Conjecture 1.4]{HII} involves the verification 
of the $q$-rational factors, which rests on the existence of Plancherel measure preserving 
correspondences for STMs as discussed in \cite{Opd4}, 
and the verification of the rational constants.  
The latter uses the knowledge of these rational constants 
from \cite{Re} for the case of equal parameter exceptional Hecke algebras, and  continuity principles 
due to \cite{CKK} and \cite{OpdSol2} (also \cite{CiuOpd2}) 
which imply roughly that we can compute these rational 
constant factors in the formal degrees of discrete series of non-simply laced affine Hecke 
algebras at any point in the parameter space of the affine Hecke algebra once we know 
these rational constants in one \emph{regular} point (in the sense of \cite{OpdSol2}) 
of the parameter space. In particular, for classical affine 
Hecke algebras of type $\textup{C}_n^{(1)}$ it was shown in \cite{CKK} that at a generic 
point in the parameter space, 
the rational constants for all generic families of discrete series characters are equal. 
The constants at special 
parameters follow then by a continuity principle in the formal degree due 
to \cite{OpdSol2}. 

An alternative approach to the conjecture \cite[Conjecture 1.4]{HII}, restricted to the case of  
formal degrees of unipotent discrete series representations, was 
formulated in \cite{CiuOpd1}. A conjectural formula for the formal degrees of unipotent 
discrete series characters is proposed in \cite{CiuOpd1}, which involves Lusztig's 
non-abelian Fourier transform matrix for families of unipotent representations \cite{L1}, \cite{L4}, \cite{L5} 
and a notion of the ``elliptic fake degree" of a unipotent discrete series character in the unramified 
minimal principal series of $G$. In this approach the formula for the rational constant factors 
of the formal degrees appears in a very natural way from the basic properties of the non-abelian 
Fourier transform.

 The notion of spectral transfer morphism is based on a certain heuristic 
idea on the behavior of $L$-packets under ordinary parabolic induction (see 
\ref{subsub:indres} for a more detailed discussion of this heuristic idea). The fact 
that this principle turns out to hold for all unipotent representations is striking.
Also striking is the fact that the isomorphism class of the Iwahori Hecke algebra 
$\mc{H}^{IM}(G)$ of $G$ is the least element in the poset of isomorphism classes 
of normalized affine Hecke algebras in the full subcategory of $\mf{C}_{es}(G)$   
whose objects are the Hecke algebras of unipotent types $(\mathbb{P},\sigma)$ of 
the inner forms of $G$, in 
the sense of \cite[Paragraph 7.1.5]{Opd4}.
Moreover, if $\mc{H}$ is such a unipotent affine Hecke algebra of an inner form of $G$, 
then the STM $\phi:\mc{H}\leadsto \mc{H}^{IM}(G)$ (which exists by the above) is 
essentially unique, and such STMs exactly match Lusztig's 
arithmetic/geometric correspondences. The proof of these statements reduces, 
as explaned in this paper,  
to the supercuspidal case \cite{FO} 
in combination with the above principle that one 
can parabolically induce 
unipotent supercuspidal STMs from Levi subalgebras to yield new STMs.  

It is 
quite clear that the definition of the notion of STM could be generalized to 
Bernstein components \cite{BD}, \cite{H1}, \cite{H2}, \cite{H3} 
in greater generality than only for the unipotent Bernstein 
components. It would be interesting to investigate the above mentioned induction principle 
in general. In view of our results, this could provide a clue how L-packets are  
partitioned by the Bernstein center beyond Lusztig's unipotent L-packets for 
simple groups of adjoint type.

In the first section of this paper we will review the theory of unipotent 
representations of $G$ with an emphasis on its harmonic analytic aspects.
The results here are all due to \cite{Mo1}, \cite{Mo2}, \cite{Lu4}, \cite{Lu6}
and \cite{DeRe}. This section serves an important purpose of reviewing 
the relevant facts on unipotent representations for this paper in the 
appropriate context of harmonic analysis, and fixing notations.
We kept the setup in this section more general than necessary for the remainder 
of the paper, since this does not complicate matters too much and this may be 
useful for later applications. 
In the second section we will describe the structure of the STMs between 
the normalized unipotent Hecke algebras of the inner forms of $G$, and 
discuss the applications of this result.
\section{Unipotent representations of quasi-simple p-adic groups}\label{sect:unip}
The category of unipotent representations of an unramified absolutely quasi-simple 
p-adic group $G$ is Morita equivalent to the category of representations of a finite 
direct sum of finitely many normalized affine Hecke algebras (called ``unipotent Hecke algebras") 
in such a way that the Morita equivalence respects the tempered spectra and 
the natural Plancherel measures on both sides.

Therefore it is an interesting problem to classify all the STMs as defined in 
\cite[Definition 5.1]{Opd4}  
between these  unipotent normalized affine Hecke algebras.
It will turn out that this task to classify these STMs essentially 
reduces to the task of finding all STMs from the rank $0$ unipotent 
affine Hecke algebras to the Iwahori-Matsumoto Hecke algebra $\mc{H}^{IM}(G')$ 
of the quasi-split $G'$ such that $G$ is an inner form of $G'$. 
In turn this reduces to solving equation \cite[equation (55)]{Opd4} where
$d^0$ denotes the formal degree of a unipotent supercuspidal representation.
The latter part of this task, the classification of the rank $0$ unipotent 
STMs, will be discussed in a second paper (joint 
with Yongqi Feng \cite{FO}). It should be remarked that the results of 
the present paper, in which the existence of certain spectral transfer 
morphisms is established, plays a role in the proof of the classification 
result in \cite{FO}.
\subsection{Unramified reductive p-adic groups}
Let $k$ be non-archimedean local field.
Fix a separable algebraic closure $\overline{k}$ of $k$, and let
$K\subset \overline{k}$ be the maximal unramified extension of $k$
in $\overline{k}$. 
Let $\mathbb{K}=\mathcal{O}/\mathcal{P}$ be the residue field of $K$, 
and let $p$ denote its characteristic. 
Let $\Gamma=\textup{Gal}(\overline{k}/k)$
denote the absolute Galois group of $k$, and let
$\mathcal{I}=\operatorname{Gal}(\overline{k}/K)\subset\Gamma$
be the inertia subgroup.
Let ${\textup{Frob}}$ be the geometric Frobenius element of
$\textup{Gal}(K/k)=\Gamma/\mathcal{I}\simeq\hat{\mathbb{Z}}$, i.e. the topological 
generator which induces the \emph{inverse} of the automorphism $x\to x^q$
of $\mathbb{K}$. Here $\mathbf{q}=p^n$ denotes the cardinality of the
residue field $\mathbf{k}:={\mathbb{K}^{{\textup{Frob}}}}$ of $k$.
We denote by $\mathbf{v}$ the positive square root of $\mathbf{q}$.

Let $\mathbf{G}$ be a connected reductive algebraic group defined over $k$, and 
split over $K$.
We denote by $G^\vee$ be the neutral component of a Langlands dual
group ${}^LG$ for $G$ (see \cite{Bo}). The construction of ${}^LG$
presupposes the choice of a maximal torus $\mathbf{S}$ and a
Borel subgroup $\mathbf{B}$ of $\mathbf{G}$ whose Levi-subgroup is $\mathbf{S}$,
and the choice of an \'epinglage for $(\mathbf{G},\mathbf{B},\mathbf{S})$,
in order to define a splitting of $\textup{Aut}(\mathbf{G})$.
Let $X^*(Z({G^\vee}))$ be the character group of the center $Z(G^\vee)$ of 
${G^\vee}$. The natural $\Gamma$-action on this space factors through the 
quotient $\textup{Gal}(K/k)$ since we are assuming that $G$ is $K$-split. 
Observe that the action of $\textup{Frob}$ on $X^*(Z({G^\vee}))$ is independent of 
the choice of a splitting of $\textup{Aut}(\mathbf{G})$.

We will always denote the group
$\mathbf{G}(K)$ of $K$-rational points of $\mathbf{G}$
by the corresponding non-boldface letter, i.e. $G=\mathbf{G}(K)$.
Kottwitz \cite[Section 7]{Ko3} has defined a $\Gamma$-equivariant 
functorial exact sequence 
\begin{equation}\label{eq:kotsurK}
1\to {G_1}\to G\stackrel{w_G}{\longrightarrow}X^*(Z({G^\vee}))\to 1
\end{equation}
In our situation there is a continuous equivariant action of the group
$\Gamma/\mathcal{I}$ on this sequence.
We denote by $F$ the action of $\textup{Frob}$ on ${G_1}$ and $G$, and
by $\theta_F$ the automorphism of $X^*(Z({G^\vee}))$ defined
by $F$.
This sequence has the property that the associated long exact sequence
in continuous nonabelian cohomology yields an exact sequence
\begin{equation}\label{eq:kotsurk}
1\to {G_1}^\textup{F}\to G(k)\to X^*(Z({G^\vee}))^{\langle\theta_F\rangle}\to 1
\end{equation}
and an isomorphism
\begin{equation}\label{eq:kotiso}
H^1(F,G)\stackrel{\sim}{\longrightarrow}X^*(Z({G^\vee}))_{\langle\theta_F\rangle}
\end{equation}

Now assume that $G$ is semisimple.
In this situation the above sequences simplify
as follows. Let $\mathbf{S}$ be a maximal $K$-split torus of $\mathbf{G}$, and
let $X:=X_*(\mathbf{S})$ be its cocharacter lattice. Let $Q:=X_{sc}=X_*(\mathbf{S}_{sc})$
be the cocharacter lattice of the inverse image of $\mathbf{S}$ in the simply
connected cover $\mathbf{G}_{sc}\to\mathbf{G}$ of $\mathbf{G}$ (hence $Q\subset X$
is the coroot lattice of $(\mathbf{G},\mathbf{S})$; we warn the reader that we call
the roots of ${G^\vee}$ ``roots'' and the roots of $(\mathbf{G},\mathbf{S})$ ``coroots''.
We apologize for this admittedly awkward convention).
Let $\Omega$ be the finite abelian group $\Omega=X/Q$. Then we may canonically
identify $X^*(Z({G^\vee}))$ with $\Omega$. Hence (\ref{eq:kotsurk})
becomes
\begin{equation}\label{eq:kotsurksimple}
1\to {G_1}^\textup{F}\to G(k)\to\Omega^{\theta_F}\to 1
\end{equation}
and (\ref{eq:kotiso}) becomes
\begin{equation}\label{eq:kotisosimple}
H^1(F,G)\stackrel{\sim}{\longrightarrow}\Omega/(1-\theta_F)\Omega
\end{equation}
We remark that $G_{der}\subset {G_1}\subset G$, and that
it can be shown that $G_{der}={G_1}$ if and only if $p$ does
not divide the order $|\Omega|$ of $\Omega$.
\emph{We will from now on always assume that $\mathbf{G}$ is absolutely quasi-simple 
and $K$-split, unless otherwise stated}.
\subsubsection{Inner $k$-rational structures of $G$}
The $k$-rational structures of $\mathbf{G}$ which
are inner forms of $G$ are parameterized by $H^1(k,\mathbf{G}_{ad})$.
By Steinberg's Vanishing Theorem it follows that
all inner $k$-forms of $G$ are $K$-split and that
$H^1(k,\mathbf{G}_{ad})=H^1(\operatorname{Gal}(K/k),G_{ad})$
(see \cite[Section 5.8]{Se2}).
\emph{We will from now on reserve the notation $G$ for a $k$-quasi-split
rational structure in this inner class}. We let $F$ be the automorphism
of $G_{ad}$ (or $G$) corresponding to the action of $\textup{Frob}$, and
$\theta=\theta_F$. We then denote the nonabelian cohomology
$H^1(\operatorname{Gal}(K/k),G_{ad})$ by $H^1(F,G_{ad})$.

For $G$ semisimple and not necessarily of adjoint type, Vogan has conjectured a 
refined Langlands parameterization 
of the irreducible tempered unipotent representations of \emph{pure inner forms}
of $G$ \cite{V}.

\emph{Pure} inner form of $G$ correspond by definition to
cocycles $z\in Z^1(F,G)$ \cite{V}, \cite{DeRe}.
Such a cocycle is determined by the image
$u:=z(\textup{Frob})\in G$. The corresponding inner $k$-form of
$G$ is defined by the functorial image $z^{ad}\in Z^1(F,G_{ad})$ of $z$.
This ``pure'' inner form is defined by the twisted Frobenius action $F_u$ on $G$ given 
by $F_u=\operatorname{Ad}(u)\circ F$, and is denoted by $G^u$. 
The cocycle $z$ determines a class in $[z]\in H^1(F,G)$. We say that two 
pure inner forms $z_1$ and $z_2$ of $G$ are equivalent iff $[z_1]=[z_2]$.
The $k$-rational isomorphism class of the inner form $G^u$ is
determined by the image $[z^{ad}]$ of $[z]$ via the natural map
$H^1(F,G)\to H^1(F,G_{ad})$. The reader be warned however, that view of 
(\ref{eq:kotisosimple}) this map is neither surjective in general (this is obvious,  
$G=SL_2$ provides an example) nor injective  
(however, if $G$ is $k$-split and semisimple then the map is injective). 
In other words, not all $k$-rational equivalence classes of inner forms of 
$G$ can be represented by a pure inner form, and if $G$ is not $k$-split and 
semisimple then an inner form of $G$ may be represented by several inequivalent 
pure inner forms.

It is in principle possible to compute with our methods the formal degrees of the elements of 
L-packets according to this refined form of the Langlands parameterization,
or even to check examples of the conjecture \cite[Conjecture 1.4]{HII} beyond the 
case of pure inner forms. For later reference we will formulate matters in this more 
general setup where possible, even though we will in present paper limit ourselves in 
the applications to the case where $G$ is of adjoint type.
\subsubsection{The affine Weyl group}
There exists a maximal $K$-split torus $\mathbf{S}$
defined over $k$ and maximally $k$-split \cite[5.1.10]{BT2}. We
fix such a maximal torus $S$ of $G$, and denote by $S_{sc}$ its
inverse image for the covering $G_{sc}\to G$. Recall that $G$ is
$k$-quasisplit, and that
$F$ defines an automorphism on the lattices $X$ and $X_{sc}=Q$
denoted by $\theta$.
The extended affine Weyl group $W$ of $(G,S)$ is defined by
\begin{equation}
W=N_G(S)/S_\mathcal{O}
\end{equation}
The group $W$ acts faithfully on the apartment $\mathcal{A}$ as an extended
affine Coxeter group.

We denote by $S_\mathcal{O}=\mathcal{O}^\times\otimes X$ the maximal
bounded subgroup of $S$. Then $X=S/S_\mathcal{O}$, and we define the
associated $F$-stable apartment $\mathcal{A}=\mathcal{A}(G,S)$
of the building of $G$ by $\mathcal{A}(G,S)=\mathbb{R}\otimes X$.
As explained in \cite[Corollary 2.4.3]{DeRe}, the isomorphism
(\ref{eq:kotisosimple}) can be made explicit by a canonical bijection
\begin{equation}\label{eq:inner}
\Omega/(1-\theta)\Omega\stackrel{\sim}{\longrightarrow} H^1(F,G)
\end{equation}
sending $[\omega]\in\Omega/{(1-\theta)\Omega}$
to the cohomology class of the cocycle $z_u$ which
maps $\textup{Frob}$ to $F_u$, where $uS_{\mathcal{O}}=x\in X$ and
$x$ is a representative of $\omega\in X/Q$.

Let $C$ be an $F$-stable alcove in $\mathcal{A}$
(such alcoves exist, see \cite{Tits}).
Let $1\to\mathbf{N}\to\mathbf{G}_{sc}\to\mathbf{G}\to 1$ be the
simply connected cover of $\mathbf{G}$, and let $\mathbf{S}_{sc}$
be the inverse image of $\mathbf{S}$.
\begin{prop}
The image of $G_{sc}\to G$ is equal to
the derived group $G_{der}$ of $G$, and
we have $G/G_{der}\stackrel{\sim}{\longrightarrow}
H^1(K,\mathbf{N})=K^\times\otimes\Omega$.
\end{prop}
\begin{proof}
Indeed, it is clear that
the image is contained in $G_{der}$ because $G_{sc}$ is its own
derived group \cite{Stein}. The other inclusion follows by applying
the long exact sequence in nonabelian cohomology to the central isogeny
$\mathbf{G}_{sc}\to\mathbf{G}$
and again appealing to Steinberg's Vanishing Theorem. It follows that
the quotient of $G$ by the image of $G_{sc}$ is the abelian group
$H^1(K,\mathbf{N})$, whence the result. On the other hand, we 
have the obvious exact sequence 
\begin{equation}\label{eq:galcohtor}
1\to\operatorname{Hom}(\Omega^*,K^\times)\to
S_{sc}\to S \to K^\times\otimes\Omega\to 1
\end{equation}
which we can compare 
to the long exact sequence in cohomology (with respect
to $\mathcal{I}$) associated to the canonical
exact sequence of diagonalizable groups
$1\to\mathbf{N}\to\mathbf{S}_{sc}\to\mathbf{S}\to 1$.
\end{proof}
We denote by $W^a_C$ the $F$-stable normal subgroup
of $W$ generated by the reflections in the walls of $C$.
This normal subgroup is independent of the choice of $C$ and
can be canonically identified with
$N_{G_{der}}(S)/S_\mathcal{O}\cap{G_{der}}
\stackrel{\sim}{\longrightarrow}W^a\subset W$,
the affine Weyl group of $(G_{sc},S_{sc})$.

Returning to Kottwitz's homomorphism we obtain the following result
(compare with \cite[5.2.11]{BT2}).
\begin{cor}\label{cor:G1}
We have $G_1=\langle S_\mathcal{O}, G_{der}\rangle$.
\end{cor}
\begin{proof}
Let $\mathbb{B}$ be the Iwahori subgroup of $G$ associated
with $C$ \cite[5.2.6]{BT2}.
By \cite[Proposition 3]{HR} we have $\mathbb{B}=\operatorname{Fix}(C)\cap {G_1}$.
In particular we have $S_\mathcal{O}\subset G_1$, so that we have
$G_{der}\subset G_1^\prime:=\langle S_\mathcal{O}, G_{der}\rangle\subset G_1$
Hence by (\ref{eq:kotsurksimple}), the equality $G_1^\prime=G_1$ is equivalent
to showing that $G/G_1^\prime=G/G_1=\Omega$.
By the previous Proposition we have
$G/G_{der}=K^\times\otimes\Omega$.
Since $S_\mathcal{O}/S_\mathcal{O}\cap{G_{der}}=\mathcal{O}^\times\otimes \Omega$
the result follows from $K^\times/\mathcal{O}^\times\simeq\mathbb{Z}$.
\end{proof}
Let $\Omega_C$ be the subgroup of $W$ which stabilizes $C$. This subgroup
may be identified with a subgroup of the group of special automorphisms
(in the sense of \cite[paragraph 1.11]{Lu4})
of the affine diagram associated with the choice of $C$.
We have a semi-direct product decomposition $W = W^a\rtimes\Omega_C$,
and thus a canonical isomorphism
$\Omega_C\stackrel{\sim}{\longrightarrow}\Omega$
for any choice of $C$.
\begin{cor}\label{cor:Wa}
We have $N_{G_1}(S)/S_\mathcal{O}
\stackrel{\sim}{\longrightarrow}W^a$, $N_{G_1}(\mathbb{B})=\mathbb{B}$ and
$N_G(\mathbb{B})/\mathbb{B}\stackrel{\sim}{\longrightarrow}\Omega$.
\end{cor}
\begin{proof}
By Corollary \ref{cor:G1} it follows that
$N_{G_1}(S)=N_{G_{der}}(S).S_\mathcal{O}$.
This implies the first assertion, since $W^a$
is the affine Weyl group of $(G_{sc},S_{sc})$ and
$G_{der}$ is the homomorphic image of $G_{sc}$.
Since an Iwahori-subgroup of $G_{sc}$ is self-normalizing, we have
similarly $N_{G_1}(\mathbb{B})=N_{G_{der}}(\mathbb{B}).S_\mathcal{O}
=(\mathbb{B}\cap G_{der}).S_\mathcal{O}=\mathbb{B}$, proving the second
assertion. For the third assertion, observe that
$\Omega_C=(N_G(\mathbb{B})\cap N_G(S))/S_\mathcal{O}$.
It is well known that $\mathbb{B}\cap N_G(S)=S_\mathcal{O}$, hence
$\Omega_C$ maps injectively into $N_G(\mathbb{B})/\mathbb{B}$. By
the second assertion this group maps injectively into $G/G_1=\Omega$.
Since $\Omega\simeq\Omega_C$ are finite the two
injective homomorphisms are in fact isomorphisms.
\end{proof}
Since $G$ is unramified there exist hyperspecial points in the apartment
$\mathcal{A}$ \cite{Tits}. A choice of a hyperspecial point $a_0\in\mathcal{A}$,
induces a semi-direct product decomposition $W=W_0\ltimes X$, where $W_0$ denotes
the isotropy subgroup of $a_0$ in $W$. The $k$-structure of $G$ defined by $F$
is quasi-split, which implies that there exists a hyperspecial point
$a_0\in\mathcal{A}(G,S)$ which is $F$-fixed.
In this case we denote by $\theta$ the automorphism of $W$
(and of $\mathcal{A}$) induced by $F$.
We fix $a_0$, an $F$-fixed hyperspecial point, and an
$F$-stable alcove $C$ having $a_0$ in its closure. Observe that
the subgroup $\Omega_C$ depends on the choice of $C$, not of the
hyperspecial point $a_0$. Recall we have a canonical isomorphism
$\Omega_C\stackrel{\sim}{\longrightarrow}\Omega=X/Q$, which we will often
use to identify these two groups.
Observe that $\theta$ stabilizes the subgroups $W^a$, $\Omega_C$,
$X$ and $W_0$ of $W$.
\subsection{Unipotent representations}
\subsubsection{Parahoric subgroups}
Recall the explicit representation of pure inner forms $G^u$
as discussed in (\ref{eq:inner}).
Fix a representative $u=\dot\omega\in N_G(S)$ with
$\omega\in\Omega_C\subset W$.
Then $F_u$ acts on the apartment $\mathcal{A}(G,S)$ by means of the
finite order automorphism $\omega\theta$.
Since $F_u$ stabilizes $C$
the Iwahori subgroup $\mathbb{B}$ is $F_u$ stable.
Recall that the group $\Omega_C$
can be canonically identified with the group $N_G(\mathbb{B})/\mathbb{B}$.
Since $\Omega$ is abelian it is clear that the subgroup
$\Omega^{F_u}_C=\Omega^{\omega\theta}_C$ of $F_u$-invariant elements is
independent of $\omega\in\Omega_C$.

Following \cite{HR} we may define
a ``standard parahoric subgroup of $G$'' as a subgroup of the form
$\textup{Fix}(F_P)\cap{G_1}$ where $F_P\subset C$ denotes a facet of $C$.
By \cite[Proposition 3]{HR} this definition coincides with the definition
in \cite{BT2}. In particular, a standard parahoric subgroup of $G$ is a connected
pro-algebraic group. A parahoric subgroup of $G$ is a subgroup conjugate to
a standard parahoric subgroup.

It is well known (by ``Lang's theorem for connected proalgebraic groups'',
see \cite[1.3]{Lu4})
that any $F_u$-stable parahoric subgroup of $G$ is
$G^{F_u}$-conjugate to a ``standard'' $F_u$-stable parahoric subgroup, i.e.
an $F_u$-stable parahoric subgroup containing $\mathbb{B}$. It follows that
the $G^{F_u}$-conjugacy classes of $F_u$-stable parahoric subgroups
are in one-to-one correspondence with the set of $\Omega^\theta$-orbits of
$\omega\theta$-stable facets in the closure of $C$. Similarly, a parahoric
subgroup $\mathbb{P}$ or a double coset of a parahoric subgroup is
$F_u$-stable iff it contains points of $G^{F_u}$.
Let $\mathbb{P}$ be an $F_u$-stable parahoric subgroup of $G$.
We call $\mathbb{P}^{F_u}$ a parahoric subgroup of $G^{F_u}$.

We record two important properties of
$F_u$-stable parahoric subgroups which follow
easily from Corollary \ref{cor:Wa}. First of all,
parahoric subgroups are self-normalizing in $G_1$, i.e.
\begin{equation}\label{eq:ratparaselfnormal}
(N_G\mathbb{P})^{F_u}\cap G_1=\mathbb{P}^{F_u}
\end{equation}
Secondly, for an $F_u$-stable
standard parahoric
subgroup $\mathbb{P}$ corresponding to a $\omega\theta$-stable 
facet $C_P$ of $C$
we have
\begin{equation}\label{eq:normalizerratpara}
(N_G\mathbb{P})^{F_u}/\mathbb{P}^{F_u}=\Omega^{\mathbb{P},\theta}
\end{equation}
where $\Omega^{\mathbb{P}}\subset \Omega_C$ is the subgroup stabilizing $C_P$,
and
$\Omega^{\mathbb{P},\theta}\subset \Omega^{\mathbb{P}}$
its fixed point group for the action of $\theta$ (or $F_u=\omega\theta$, 
which amounts to the same since $\Omega_C$ is abelian). 
We define an exact sequence
\begin{equation}\label{eq:Om}
1\to\Omega^{\mathbb{P},\theta}_1\to\Omega^{\mathbb{P},\theta}\to
\Omega^{\mathbb{P},\theta}_2\to1
\end{equation}
where $\Omega^{\mathbb{P},\theta}_1$ is the subgroup of elements
which fix the set of $F_u$-orbits of vertices of $C$ not in $C_P$ pointwise.
\subsubsection{Normalization of Haar measures}\label{par:ber}
Let $G$, $F$, and $F_u$ be as in the previous paragraph.
Then $G^{F_u}$ is a locally compact group.
For any $F_u$-stable parahoric subgroup $\mathbb{P}$ of $G$ we
denote by $\overline{\mathbb{P}^{F_u}}$ the reductive quotient of
$\mathbb{P}^{F_u}$. This is the group of $\mathbf{k}$-points of a
connected reductive group over $\mathbf{k}$.
In particular this is a finite group.

Following \cite[Section 5.1]{DeRe} we normalize the Haar measure of $G^{F_u}$
uniquely such that for all $F_u$-stable parahoric subgroups $\mathbb{P}$ of $G$
one has
\begin{equation}\label{eq:haar}
\textup{Vol}(\mathbb{P}^{F_u})=\mathbf{v}^{-a}|\overline{\mathbb{P}^{F_u}}|
\end{equation}
where $a\in\mathbb{Z}$ is equal to the dimension of $\overline{\mathbb{P}}$
over $\mathbb{K}$.
It is well known that the right hand side
is a product of powers of $\mathbf{v}$ and cyclotomic polynomials in
$\mathbf{v}$. 
\subsubsection{The anisotropic case}\label{subsub:aniso}
It is useful to discuss the case where $G$ is anisotropic explicitly.
It is well known that an anisotropic absolutely simple unramified group $G$ is isomorphic to  
$\textup{PGL}_1(\mathbb{D}):=\mathbb{D}^\times/k^\times$, where $\mathbb{D}$ is 
an unramified central division algebra over $k$ of degree $m+1$, rank $(m+1)^2$ 
(see for instance \cite{CH}).
We choose a uniformizer $\pi$ of $k$.
$\mathbb{D}$ contains an unramified extension $l$ of degree $m+1$ over $k$, 
and we may choose a uniformizer $\Pi$ of $\mathbb{D}$ which 
normalizes $l$, such that conjugation by $\Pi$ restricted to $l$ yields a  
generator for $\textup{Gal}(l:k)$, and such that $\Pi^{m+1}=\pi$. 
The group $\mathbb{P}:=G_1$ 
is the only parahoric subgoup in this situation, and obviously $G=N\mathbb{P}$.
By (\ref{eq:kotsurksimple}) we have 
$\Omega:=G(k)/G_1^{F_u}\approx \frac{\langle\Pi\rangle}{\langle\pi\rangle}\approx \textup{Gal}(l:k)$,
a cyclic group of order $m+1$.
$G$ contains a 
maximal prounipotent subgroup  $G_+$ (denoted by $V_1$ in \cite{CH}) 
and we have $G=C.G_+$, where $C$ is generated by the anisotropic torus 
$T:=l^\times/k^\times$ and $\Pi$. 
We see that the reductive quotient $\mathbb{P}/(\mathbb{P}\cap G_+)$ is 
an anisotropic torus $\overline{\mathbb{T}}$ of rank $m$ over $\mathbf{k}$, and that 
$\overline{\mathbb{T}^{F_u}}$ can be identified with the group of roots of 
unity of order prime to $p$ in $l$ modulo the subgroup of those roots of unity in $k$. 
Hence $\textup{Vol}(G(k))=v^{-m}|\Omega||\overline{\mathbb{T}^{F_u}}|=(m+1)[m+1]_q$ 
(with $[m+1]_q$ the $q$-integer associated to $m+1\in\mathbb{N}$ (see Definition \ref{defn:q-rat})).
\subsubsection{Unipotent representations and affine Hecke algebras}\label{par:uniprep}
Let $G$ be a quasi-simple linear algebraic group, defined and quasi-split over $k$
and $K$-split as above. Recall that the automorphism induced by the Frobenius $F$
on the building of $G$ was denoted by $\theta$.
Recall that the inner forms of $G$ are canonically parameterized
by the abelian group $\Omega/(1-\theta)\Omega$.
Let $u=\dot\omega\in N\mathbb{B}$ be a representative of an element
$\bar\omega\in\Omega/(1-\theta)\Omega$ and let $F_u$ denote the corresponding pure
inner twist of $F$. We denote by $G^u$ the pure inner form of $G$ defined by this twisted
$k$-structure (in particular $G^1=G$).

A representation $(E,\delta)$ of a parahoric subgroup $\mathbb{P}^{F_u}$
of $G^{F_u}$ (where $\mathbb{P}$ is an $F_u$-stable parahoric subgroup of $G$) is
called \emph{cuspidal unipotent} if it is the lift to $\mathbb{P}^{F_u}$ of
a cuspidal unipotent representation of the reductive quotient
$\overline{\mathbb{P}^{F_u}}$. An $F_u$-stable parahoric subgroup is
called cuspidal unipotent if it has cuspidal unipotent representations.

Lusztig \cite{Lu4} introduced the category $R(G^{F_u})_{uni}$ of unipotent
representations
of $G^{F_u}$. A smooth representation $(V,\pi)$ of $G^{F_u}$ is called unipotent
if $V$ is generated by a sum of cuspidal unipotent isotypical components
of restrictions of $(V,\pi)$ to various parahoric
subgroups of $G^{F_u}$.
As a generalization of Borel's theorem on Iwahori-spherical
representations, $R(G^{F_u})_{uni}$ is an abelian subcategory of the category
$R(G^{F_u})$ of smooth $G^{F_u}$ representations.
It is central to the approach in this paper
that this category is equivalent to the module category of an explicit finite
direct sum of normalized Hecke algebras in the sense of paragraph
\cite[3.1.2]{Opd4}, in a way which is compatible with harmonic analysis.
Let us therefore describe this in detail.

A cuspidal unipotent pair $(\mathbb{P},\delta)$ consists of an
$F_u$-stable parahoric subgroup $\mathbb{P}$
of $G$ and an irreducible cuspidal unipotent representation $\delta$
of $\mathbb{P}^{F_u}$. We say that $(\mathbb{P},\delta)$ is standard if
$\mathbb{P}$ is standard.
Let $R(G^{F_u})_{(\mathbb{P},\delta)}$ denote the subcategory of
$R(G^{F_u})$ consisting of the smooth representations $(V,\pi)$ such that
$V$ is generated by the isotypical component $(V|_{\mathbb{P}^{F_u}})_\delta$.
According to \cite{Lu4}, given two cuspidal unipotent pairs
$(\mathbb{P}_i,\delta_i)$ (with $i\in\{1,2\}$) the subcategories
$R(G^{F_u})_{(\mathbb{P}_i,\delta_i)}$ are either disjoint or equal, and
this last alternative occurs if and only if the pairs $(\mathbb{P}_i,\delta_i)$
are $G^{F_u}$-conjugates (and not just associates).
It follows that a smooth representation $(V,\pi)$ is unipotent iff $V$
is generated by
$\oplus (V|_{\mathbb{P}^{F_u}})_{\delta}$, where the direct sum is taken over
a complete set of representatives $(\mathbb{P},\delta)$ of the finite set
of $\Omega^\theta$-orbits of standard cuspidal unipotent pairs.

For each standard cuspidal unipotent pair $\mathfrak{s}=(\mathbb{P},\delta)$
we consider the algebra
$\mathcal{H}^{u,\mathfrak{s}}_{\mathbf{v}}$ of $\mathfrak{s}$-spherical
$\operatorname{End}(E)$-valued functions on $G^{F_u}$, equipped with
a trace $\tau(f):=\operatorname{Tr}_{V}(f(e))$ and $*$ defined by
$f^*(x):=f(x^{-1})^*$. This algebra turns out to be the specialization
at $\mathbf{v}$ of a finite direct sum of mutually
isomorphic normalized (in the sense of \cite[paragraph 3.1.2]{Opd4})
affine Hecke algebras (called unipotent affine Hecke algebras) defined
over $\mathbf{L}=\mathbb{C}[v^{\pm 1}]$,
and has been explicitly determined in all cases \cite{Mo2}, \cite{Lu4}.
The following general result from the theory of types due to \cite{BHK} 
(also see \cite{HO}) 
is fundamental to the approach in this paper: 
\begin{thm}\label{thm:typeplancherel}
The assignment
$(V,\pi)\to V^{(\mathbb{P},\delta)}:=\textup{Hom}_{\mathbb{P}}(\delta,V|_{\mathbb{P}^{F_u}})$ 
establishes an equivalence of categories
from $R(G^{F_u})_{(\mathbb{P},\delta)}$ to the category of
$\mathcal{H}^{u,\mathfrak{s}}_{\mathbf{v}}$-modules which respects the notion
of temperedness and which is Plancherel measure preserving on the level
of irreducible tempered representations. 
\end{thm}
\subsubsection{The group of unramified characters}
Recall that we have a canonically identification
of $\Omega$ with $N_{G}(\mathbb{B})/\mathbb{B}$.

By application of Lang's Theorem for proalgebraic groups
\cite[paragraph 1.8]{Lu4} one sees that the $F_u$ stable double 
$\mathbb{B}$-cosets $\Theta$ in $G$ are precisely those which 
intersect with $G^{F_u}$, in which case $\Theta\cap G^{F_u}$ is a single 
double coset of $\mathbb{B}^{F_u}$.

Let $u=\dot\omega$ be a representative of an element
$\omega\in\Omega/(1-\theta)\Omega$.
We see that the double $\mathbb{B}^{F_u}$-cosets of $G^{F_u}$ are
parameterized by the $\omega\theta$ fixed group $W^{\omega\theta}$, and
that $\Omega^{\omega\theta}=N_{G^{F_u}}\mathbb{B}^{F_u}/\mathbb{B}^{F_u}$.
Because $\Omega$ is abelian we actually have
$\Omega^{\omega\theta}=\Omega^\theta$.
We have $W^{\omega\theta}=W^{\omega\theta,a}\rtimes\Omega^\theta$. By
\cite{Lu4} this extended affine Weyl group is the underlying affine
Weyl group of the Iwahori-Hecke algebra $\mc{H}^{u,IM}:=\mc{H}^{u,(\mathbb{B},1)}$
of the group $G^{F_u}$. When $[\omega]=1$ we denote this algebra simply by
$\mc{H}^{IM}$, the \emph{generic} Iwahori-Hecke algebra.

The Pontryagin dual $(\Omega^\theta)^*$ of $\Omega^\theta$
can be viewed canonically as the group of unramified complex linear characters
$X^*_{un}(G^{F_u})$ of $G^{F_u}$. This defines a natural functorial action of
$(\Omega^\theta)^*$ on the category $R(G^{F_u})_{uni}$ (by taking tensor
products). These functors are Plancherel measure preserving, as we will see,
and play an important role.
\subsection{Unramified local Langlands parameters}\label{subsub:loclang}
The based root datum of the connected component $G^\vee$ of the Langlands dual group
${{}^LG}$ of $G^u$ is defined by $\mc{R}=(X,R_0,Y,R_0^\vee,F_0)$.
The dual Langlands group of $G^u$ is independent of $u$ and defined by
\begin{equation}
{}^LG:=G^\vee\rtimes\langle\theta\rangle
\end{equation}
where $\theta$ denotes the outer automorphism of $G^\vee$ arising from $F$.
Let $S^\vee\subset G^\vee$ be a maximal torus of
$G^\vee=G^\vee(\mathbb{C})$.
Let $Z(G^\vee)$ be the center of the neutral component $G^\vee$ of ${}^LG$.
Then $Z(G^\vee)\simeq P^\vee/Y=\Omega^*\subset S^\vee$.
We will denote by ${}^LZ$ the central subgroup ${}^LZ:=Z(G^\vee)^\theta\subset{}^LG$,
so that ${}^LZ$ is canonically equal to
$(\Omega^*)^\theta\subset S^{\vee,\theta}$.
It follows that we can canonically identify the group $\Omega/(1-\theta)\Omega$ with the
Pontryagin dual group of ${}^LZ$, which is the version of Kottwitz's
Theorem as explained in detail in \cite{DeRe}.

Let us recall the space of unramified local Langlands parameters for $G^{F_u}$
for later reference. Let $\mathcal{W}_k$ denote the Weil group of $k$ \cite{Tate},
with inertia subgroup $\mathcal{I}\subset\mathcal{W}_k$, and let $\textup{Frob}$ denote
a generator of $\mathcal{W}_k/\mathcal{I}$. An  \emph{unramified local Langlands parameter}
is a homomorphism
\begin{equation}\label{eq:qsLP}
\lambda:\langle\textup{Frob}\rangle\times\textup{SL}_2(\mathbb{C})\to {}^LG
\end{equation}
such that $\lambda(\textup{Frob}\times \operatorname{id})=s\times\theta$ with
$s\in G^\vee$ semisimple and such that $\lambda$ is algebraic on the
$\textup{SL}_2(\mathbb{C})$-factor. Given an unramified Langlands parameter 
$\lambda$ we denote by $[\lambda]$ its orbit for
the action of $G^\vee$ by conjugation. We will write $\Lambda$
for the set of orbits $[\lambda]$ of unramified Langlands parameters.

If $\lambda$ is an unramified Langlands parameter, 
let $A_\lambda:=\pi_0(C_{G^\vee}(\lambda))$
be the component group of the centralizer of $\lambda$ in $G^\vee$.
We call $\lambda$ \emph{elliptic} (or discrete) if $C_{G^\vee}(\lambda)$ is finite,
and denote by $\Lambda^e$ the space of $G^\vee$-orbits of 
unramified elliptic Langlands parameters.

Let $\lambda$ be an unramified elliptic Langlands parameter.
Observe that ${}^LZ=Z(G^\vee)^\theta\subset Z(G^\vee)\subset A_\lambda$. 
The inner forms $G^u$ of $G$ are canonically parameterized via Kottwitz's Theorem by 
the character group of ${}^LZ_{sc}$, the center of the $L$-group of $G_{ad}=G/Z(G)$. 
Given an inner form $G^u$, we choose once and for all 
a character $\zeta_u\in \textup{Irr}(Z_{sc})$ 
(with $Z_{sc}:=Z(G^\vee)_{sc})$) which restricts to the character 
$\omega_u\in \Omega_{ad}/(1-\theta)\Omega_{ad}$ of ${}^LZ_{sc}$
that is represented by $u=\dot\omega\in N_{G_{ad}}(\mathbb{B})$.
Following \cite[Section 7.2]{GR} (see also \cite{A}) 
we consider the group $A_\lambda/Z(G^\vee)\subset (G^\vee)_{ad}$, and let 
$\mc{A}_\lambda\subset (G^\vee)_{sc}$ be its full preimage in the simply connected 
cover $(G^\vee)_{sc}$ of $(G^\vee)_{ad}$. Thus $Z_{sc}\subset \mc{A}_\lambda$, 
and $\mc{A}_\lambda$ is a central extension of $A_\lambda/Z(G^\vee)$ by $Z_{sc}$.
We denote by $\operatorname{Irr}^u(\mc{A}_{\lambda})$ the set
of irreducible characters $\rho$ of $\mc{A}_\lambda$ on which $Z_{sc}$ acts by 
a multiple of $\zeta_u$.

The space of ($G^\vee$-orbits of) unramified discrete Langlands data for $G^u$ is defined by
\begin{equation}
\tilde{\Lambda}^u:=\{(\lambda,\rho)\mid[\lambda]\in\Lambda^e, \rho\in\operatorname{Irr}^u(\mc{A}_{\lambda})\}/{G^\vee}
\end{equation}
and denote its elements by $[\lambda,\rho]$.
For fixed $\lambda$ with $[\lambda]\in\Lambda^e$ we denote by $\tilde{\Lambda}^u_\lambda$
the fiber of $\tilde{\Lambda}^u$ above $[\lambda]$ 
(with respect to the projection of $\tilde{\Lambda}^u$ to the first factor).
We will often simply write ${\tilde{\Lambda}}$ if we refer to
the space of (orbits of) unramified local Langlands data of the quasi-split
group $G=G^1$.

The isomorphism classes of pure inner forms $G^u$ are parametrized canonically by 
$\omega_u\in \textup{Irr}({}^LZ_{sc})$. In the refined version of the local Langlands correspondence 
where we restrict ourselves to pure inner forms of $G$ it is therefore more natural to work 
with pairs $(\lambda,\rho)$ with $\rho\in \operatorname{Irr}^u(A_{\lambda})$, the set 
of irreducible characters of $A_\lambda$ which restrict to a multiple of $\omega_u$ on 
${}^LZ$ (hence there is no need to make choices of the extensions $\zeta_u$ in this case).

It is well known \cite[Paragraph 6.7]{Bo} that we have a canonical isomorphism
\begin{equation}
\beta:(G^\vee\times\theta)/\textup{Int}(G^\vee)\stackrel{\sim}{\longrightarrow}
\operatorname{Hom}(X^\theta,\mathbb{C}^\times)/W_0^\theta
\end{equation}
Observe that the group $(\Omega^\theta)^*$ of unramified characters
on $G^{F_u}$ is exactly the ``central subgroup''
of the complex torus $T_{\mb{v}}(\mathbb{C})
:=\operatorname{Hom}(X^\theta,\mathbb{C}^\times)$, i.e. the
subgroup of $W_0^\theta$-invariant elements.
Here we consider $T$ as the diagonalizable group scheme with character
lattice $\mathbb{Z}\times X^\theta$ over the ring
$\mathbf{L}=\mathbb{C}[v^{\pm 1}]$
and we use the notation $T_\mathbf{v}$ to denote its fiber over
$\mathbf{v}\in\mathbb{C}^\times$.

We have natural compatible
actions of $X^*_{un}(G^{F_u})=(\Omega^\theta)^*$ on the sets $\Lambda$ and
$\tilde{\Lambda}^u$ defined by $\omega[\lambda]=[\omega\lambda]$ and
$\omega[\lambda,\rho]=[\omega\lambda,\rho]$ respectively, provided that we choose the 
extensions $\zeta_u$ in a compatible way within each orbit under $X^*_{un}(G^{F_u})$ 
(for pure inner forms we do not need to worry about this).

We remark that $W_0\backslash T_{\mb{v}}(\mathbb{C})$ can be identified with
the maximal spectrum
$S_{\mathbf{v}}^{IM}$ of the center $\mathcal{Z}_{\mathbf{v}}^{IM}$ of the Iwahori Hecke
algebra $\mathcal{H}^{IM}_{\mathbf{v}}=\mathcal{H}^{(\mathbb{B},1)}(G^{F})$
of the group of points of the $k$-quasisplit group $G^F=G(k)$.
By the Kazhdan-Lusztig correspondence
\cite{KL} there exists a canonical bijection between the set of central characters
$W_0r_\mb{v}\in S_{\mathbf{v}}^{IM}$ supporting discrete series representations of
$\mathcal{H}^{IM}_{\mathbf{v}}$ and the set of $G^\vee$-orbits of unramified
elliptic local Langlands parameters.
This bijection $[\lambda]\to W_0r_{\lambda}$ is defined by
\begin{equation}\label{eq:llprespt}
W_0r_{\lambda,\mb{v}}=\beta\left(G^\vee\cdot
\lambda
\left(\textup{Frob},
\left(
\begin{matrix}
\mb{v}&0\\
0&\mb{v}^{-1}\\
\end{matrix}
\right)
\right)
\right)
\end{equation}
This map is equivariant with respect to the natural action of the group
$X^*_{un}(G^{F_u})$ of unramified characters of $G^{F_u}$.
\subsection{Unipotent affine Hecke algebras}
According to \cite[1.15, 1.16, 1.17, 1.20]{Lus2} we can decompose for
each cuspidal unipotent pair $(\mathbb{P},\delta)$ of $G^u$
the algebra
$\mathcal{H}^{u,\mathfrak{s}}$ of $\mathfrak{s}$-spherical functions
on $G^u$ explicitly as a direct sum
of mutually isomorphic extended affine Hecke algebras as follows.

Let us use the shorthand notation $N\mathbb{B}$ for $N_G(\mathbb{B})$ etc.
Recall that, since
Borel subgroups of a connected reductive group are mutually conjugate and
self normalizing, the group $\Omega^\mathbb{P}=N\mathbb{P}/\mathbb{P}$
is naturally a subgroup of the finite abelian group
$\Omega=N\mathbb{B}/\mathbb{B}$ (see (\ref{eq:normalizerratpara})).
It is known that the group $\Omega^{\mathbb{P},\theta}$ (see (\ref{eq:normalizerratpara}))
acts trivially on
the set of irreducible unipotent cuspidal representations of $\mathbb{P}^{F_u}$.
Even more is true \cite{Lu4}: for every cuspidal unipotent representation $(E,\delta)$
of $\mathbb{P}^{F_u}$ there exists an extension $(E,\tilde\delta)$ of
$(E,\delta)$ to the normalizer $N\mathbb{P}^{F_u}:=N_{G^{F_u}}(\mathbb{P}^{F_u})$ of
$\mathbb{P}^{F_u}$ in $G^{F_u}$. We denote the group
$\Omega^{\mathbb{P},\theta}$ by $\Omega^{\mf{s},\theta}$ to stress the invariance of the
cuspidal pair $\mf{s}=(\mathbb{P}^{F_u},\delta)$.
One observes
that the set of such extensions is a torsor for the group
$(\Omega^{\mf{s},\theta})^*$ of irreducible characters of
$\Omega^{\mf{s},\theta}=N\mathbb{P}^{F_u}/\mathbb{P}^{F_u}$ by tensoring.
Hence the group $X^*_{un}(G^{F_u})=(\Omega^{\theta})^*$
of unramified characters of $G^{F_u}$ acts transitively on the set of
extensions of $(E,\delta)$ to $N\mathbb{P}^{F_u}$, and
the kernel of this action is equal to the subgroup
$(\Omega^\theta/\Omega^{\mf{s},\theta})^*$
of $(\Omega^{\theta})^*$ of unramified characters of $G^{F_u}$ which restrict
to $1$ on $N\mathbb{P}^{F_u}$.
Lusztig showed that the $\mf{s}$-spherical Hecke algebra
$\mc{H}^{u,\mathfrak{s}}$ is of the form
\begin{equation}
\mc{H}^{u,\mathfrak{s}}=\mc{H}^{u,\tilde{\mf{s}},a}\rtimes\Omega^{\mathfrak{s},\theta}
\end{equation}
where $\mc{H}^{u,\tilde{\mf{s}},a}$ is an unextended affine Hecke algebra associated with a
certain affine Coxeter group $(W_{\mathfrak{s}},S_{\mathfrak{s}})$
and a parameter function $m_S^\mf{s}$, all defined in terms of the pair
$\mathfrak{s}=(\mathbb{P}^{F_u},\delta)$. In particular, they are independent of the
chosen extension $\tilde{\mf{s}}$ of $(\mathbb{P},\delta)$ to $N\mathbb{P}^{F_u}$;
for this reason we will often suppress the extension in the notation and write
$\mc{H}^{u,\mf{s},a}$ instead of $\mc{H}^{u,\tilde{\mf{s}},a}$.

In order to define a normalized affine Hecke algebra
(in our sense) from these data one needs to choose a distinguished set
$S_{\mathfrak{s},0}\subset S_{\mathfrak{s}}$.
Although this is not canonically defined, different choices are related
via admissible isomorphisms.
Let $\Omega_1^{\mathfrak{s},\theta}\subset\Omega^{\mathfrak{s},\theta}$
be the subgroup which acts trivially on $S_{\mathfrak{s}}$ (see (\ref{eq:Om})).
Then the quotient
$\Omega^{\mathfrak{s},\theta}_2=\Omega^{\mathfrak{s},\theta}/\Omega_1^{\mathfrak{s},\theta}$
acts faithfully on $(W_{\mathfrak{s}},S_{\mathfrak{s}})$
by special affine diagram automorphisms.
Lusztig \cite[1.20]{Lus2} showed that $\mc{H}^{u,\mathfrak{s}}$
is isomorphic to the tensor product of the group algebra
$\mathbb{C}[\Omega_1^{\mathfrak{s},\theta}]$ and
the crossed product
\begin{equation}\label{eq:faith}
\mc{H}^{u,\tilde{\mf{s}},e}=\mc{H}^{u,\tilde{\mf{s}},a}\rtimes\Omega_2^{\mathfrak{s},\theta}
\end{equation}
which is an extended affine Hecke algebra. Recall from \cite[Proposition 2.3]{Opd4} 
that this information is enough to recover a pair of data (independent of the chosen extension
$\tilde{\mf{s}}$ of $\mf{s}$)
$(\mc{R}^{u,\mf{s}},m^{u,\mf{s}})$ such that we have an admissible isomorphism
$\mc{H}^{u,\mf{s},e}\stackrel{\sim}{\longrightarrow}\mc{H}(\mc{R}^{u,\mf{s}},m^{u,\mf{s}})$ of normalized
affine Hecke algebras.
\subsubsection{The normalization of the algebras $\mc{H}^{u,\mf{s},e}$}\label{subsub:normunipha}
Observe that the unit element of $\mathcal{H}^{u,\mathfrak{s}}$
is the function $e^\mathfrak{s}$ on $G^{F_u}$ supported on $\mathbb{P}^{F_u}$
defined by
\begin{equation}
e^\mathfrak{s}(g)=\textup{Vol}(\mathbb{P}^{F_u})^{-1}\chi(g)\delta(g)
\end{equation}
where $\chi$ denotes the characteristic function of $\mathbb{P}^{F_u}$.

Fix an extension $\tilde{\mf{s}}$ of $\mf{s}$ as in the previous paragraph.
By (\ref{eq:normalizerratpara}) the unit
element $e^\mf{s}$ can be decomposed as a sum of mutually orthogonal
idempotents
\begin{equation}
e^\mf{s}=\sum_{\lambda\in(\Omega^{\mf{s},\theta})^*}e^{\tilde{\mf{s}}}\lambda
\end{equation}
where we view $\lambda\in(\Omega^{\mf{s},\theta})^*$ as a linear character of
$N\mathbb{P}^{F_u}$ and where
\begin{equation}
e^{\tilde{\mf{s}}}(g)=
\textup{Vol}(N\mathbb{P}^{F_u})^{-1}\chi_{N\mathbb{P}^{F_u}}(g)\tilde\delta(g)
\end{equation}
By (\ref{eq:faith}) (and the text just above it) we see that
the unit element of $\mc{H}^{u,\tilde{\mf{s}},e}$ is equal to
\begin{equation}
e^{\tilde{\mf{s}},e}=\sum_{\lambda\in(\Omega^{\mf{s},\theta}_2)^*}e^{\tilde{\mf{s}}}\lambda
\end{equation}
In particular, the group $(\Omega^{\mf{s},\theta})^*$ acts transitively on the
set of idempotents $e^{\tilde{\mf{s}},e}$ obtained by
choosing different extensions $\tilde{\mf{s}}$ of $\mf{s}$, and the kernel of this
action is the subgroup $(\Omega^{\mf{s},\theta}_2)^*\subset(\Omega^{\mf{s},\theta})^*$.

The other canonical basis elements of $\mathcal{H}^{u,\tilde{\mf{s}},e}$
are supported on other double cosets of $N\mathbb{P}^{F_u}$.
In particular, the trace $\tau$ vanishes on those other
basis elements. Hence $\tau$ is a multiple of the standard
trace of the affine Hecke algebra $\mc{H}^{u,\tilde{\mf{s}},e}$, and
the normalization factor is of the form
\begin{equation}\label{eq:NormMu}
d^{\tau,\tilde{\mf{s}},e}:=\tau(e^{\tilde{\mf{s}},e})=
|\Omega^{\mf{s},\theta}_1|^{-1}
\textup{Vol}(\mathbb{P}^{F_u})^{-1}
\textup{deg}({\delta})
\end{equation}
The rational number $d^{\tau,\mf{s},e}$ is the evaluation of a Laurent 
polynomial in the square root $\mathbf{v}$ of the cardinality
$\mathbf{q}$ of the residue field $\mathbf{k}$. When we treat $\mathbf{v}$ 
and $\mathbf{q}$ as an indeterminate we will denote these as $v$ and $q$ 
respectively.
By our normalization of the Haar measure the factor $\textup{Vol}(\mathbb{P}^{F_u})$
in the denominator is equal to, up to a power of $\mathbf{v}$, the cardinality
of the group of $\mathbf{k}$-points of the reductive group
$\overline{\mathbb{P}}$ with Frobenius action $F_u$. Therefore all factors in
$d^{\tau,\tilde{\mf{s}},e}$ are explicitly known rational function in $\mathbf{v}$
(cf. \cite[Section 2.9, Section 13.7]{C}).
The following property of $d^{\tau,\mf{s},e}$ is very convenient:
\begin{prop}\label{prop:balanced}
Let $\overline{\mathbb{T}}=
\overline{\mathbb{T}}_Z\overline{\mathbb{T}}_S$ denote a maximal $F_u$-stable,
maximally $\mathbf{k}$-split torus of $\overline{\mathbb{P}}$, with
$\overline{\mathbb{T}}_Z$ the maximal central subtorus. Let $V_Z$ (resp. $V_S$)
denote the
rational vector space spanned by the algebraic character lattice $L_Z$
(resp. $L_S$) of $\overline{\mathbb{T}}_Z$ (resp. $\overline{\mathbb{T}}_S$),
and let $F_Z$ (resp. $F_S$) be the automorphism
of $L_Z$ (resp. $L_S$) induced by $F_u$. Then we have
\begin{equation}\label{eq:deg}
d^{\tau,\mf{s},e}=\pm|\Omega^{\mf{s},\theta}_1|^{-1}
\textup{det}_{V_{Z}}(\mathbf{v}\textup{Id}_{V_{Z}}-\mathbf{v}^{-1}F_Z)^{-1}
\prod_{i=1}^l(\mathbf{v}^{d_i}-\epsilon_i\mathbf{v}^{-d_i})^{-1}
\mathbf{v}^a\textup{deg}_\mathbf{v}({\delta})
\end{equation}
where $l$ is the semisimple rank of $\overline{\mathbb{P}}$ over
$\mathbb{K}$, $d_i$ are the primitive degrees of the Weyl group invariants of
the semisimple part of $\overline{\mathbb{P}}$, the $\epsilon_i$ are
the eigenvalues of $F_S$ acting on the co-invariant ring
with respect to the Weyl group action on $V_S$
(certain roots of unity, see \cite[Section 2.9]{C}),
and where $a\in\mathbb{Z}$ is such that
$f(\mathbf{v})=\mathbf{v}^a\textup{deg}_\mathbf{v}({\delta})$ satisfies
$f(\mathbf{v}^{-1})=\pm f(\mathbf{v})$. At $v=1$, $d^{\tau,\mf{s},e}$ has a pole of order   
equal to the split rank $r_Z$ of $\overline{\mathbb{T}}_Z$, and  
satisfies $d^{\tau,\mf{s},e}(v)=(-1)^{r_Z} d^{\tau,\mf{s},e}(v^{-1})$.
\end{prop}
\begin{proof}
For $G$ containing a $k$-split torus of positive dimension then this is an easy case-by-case 
verification using \cite[Section 2.9, Section 13.7]{C}. The anisotropic case is 
easy by the results stated in \ref{subsub:aniso}.
\end{proof}
As a consequence, with our normalization of Haar measures, the normalization constant 
$d^{\tau,\mf{s},e}$ of 
a unipotent affine Hecke algebras $\mc{H}^{u,\mf{s},e}$ satisfies 
the condition of \cite[3.1.2]{Opd4} and, at $v=1$, has a pole of order equal to the rank 
of $\mc{H}^{u,\mf{s},e}$. 
Hence by Theorem \cite[Theorem 4.8]{Opd4}(iii), in our normalization 
of Haar measures 
all formal degrees of the discrete series representations of the unipotent affine Hecke 
algebras, and thus of all unipotent discrete series representations, are symmetric with regards to 
$v\to v^{-1}$, and regular and nonzero at $v=1$. This is convenient, since it implies 
that we never need to be concerned about the factors $v^N$ or of 
$(v-v^{-1})^M$ of the formal degree of a unipotent discrete series: 
With our normalizations these factors do 
not appear in $\textup{fdeg}(\pi)$.
\begin{defn}\label{defn:q-rat} Let $\mathbf{K}^\times$ be the field of rational fundtions in $v$. 
Recall the notion of a normalized affine Hecke algebra \cite[Definition 2.13]{Opd4}.
Given our normalization of the traces, we see from \cite[Theorem 4.8]{Opd4} and 
Theorem \ref{thm:typeplancherel}  
that the formal degree $\textup{fdeg}(\pi)$ of a discrete series representation $\pi$ 
of a unipotent Hecke algebra has a unique representation  
$\textup{fdeg}(\pi)=\lambda\textup{fdeg}(\pi)_q\in\mathbf{K}^\times$
where $\lambda\in\mathbb{Q}_+$,  and 
$\textup{fdeg}(\pi)_q$ is a $q$-rational number (by which we mean a fraction of products of 
$q$-integers $[n]_q:=\frac{\mb{v}^n-\mb{v}^{-n}}{\mb{v}-\mb{v}^{-1}}$ with $n\in\mathbb{N}$). 
We call $\textup{fdeg}(\pi)_q$ the $q$-rational factor of $\textup{fdeg}(\pi)$.
\end{defn}
\begin{cor}\label{cor:equiv1}
For each $\omega\in\Omega/(1-\theta)\Omega$ (with representative
$u\in N\mathbb{B}^{F_u}$ as before) and each cuspidal unipotent pair $\mf{s}$ of
$G^u$, the pair
$(\mc{H}^{u,\mf{s},e},d^{\tau,\mf{s},e})$ is a normalized affine Hecke
algebra in the sense of Definition \cite[Definition 3.1]{Opd4}. The group
$(\Omega^{\mf{s},\theta}_2)^*$ acts naturally on the algebra
$\mc{H}^{u,\mf{s},e}=\mc{H}^{u,\mf{s},a}\rtimes\Omega^{\mf{s},\theta}_2$
by means of
essentially strict automorphisms (cf. \cite[paragraphs 2.1.7 and 3.3.3]{Opd4}) (in particular, this action induces
spectral measure
preserving automorphisms on the tempered spectrum of
$(\mc{H}^{u,\mf{s},e},d^{\tau,\mf{s},e})$). The abelian group
$(\Omega^{\mf{s},\theta})^*$ acts similarly by essentially strict automorphisms on
$\mc{H}^{u,\mf{s}}\stackrel{\sim}{\longrightarrow}
\mc{H}^{u,\mf{s},a}\rtimes\Omega^{\mf{s},\theta}\approx\mc{H}^{u,\mf{s},e}\otimes\mathbb{C}[\Omega^{\mf{s},\theta}_1]$. 
This action
is transitive on the set of direct summands of the form
$(\mc{H}^{u,\tilde{\mf{s}},e},d^{\tau,\mf{s},e})$
where $\tilde{\mf{s}}$ runs over the set of extensions
of $\mf{s}$ to $N\mathbb{P}^{F_u}$. The subgroup
$(\Omega^{\mf{s},\theta}_2)^*\subset (\Omega^{\mf{s},\theta})^*$
is the kernel of the induced action on the set of these direct summands.
\end{cor}
Recall that
$N_{G^u(k)}(\mathbb{B}^{F_u})/\mathbb{B}^{F_u}\stackrel{\sim}{\longrightarrow}\Omega^\theta_C$ by
(\ref{eq:normalizerratpara}). In particular this group acts naturally on the set of
$F_u$-stable standard cuspidal parahoric subgroups of $G^u$.
This action extends naturally to an action on the set of
equivalence classes of standard cuspidal
unipotent pairs $\mf{s}=(\mathbb{P},\delta)$ by
$\omega\cdot(\mathbb{P},\delta)=({}^\omega\mathbb{P},{}^\omega\delta)$;
as was remarked before, the isotropy group of $\mf{s}=(\mathbb{P},\delta)$
is the same as that of its first component $\mathbb{P}$.
If $\omega\cdot\mf{s}_1=\mf{s}_2$ then conjugation by $\omega\in\Omega_C^\theta$
gives rise to an isomorphism
$\phi_\omega:\mc{H}^{u,\mf{s}_1}\stackrel{\sim}{\longrightarrow}
\mc{H}^{u,\mf{s}_2}$ which
maps the various normalized extended affine Hecke algebra summands of the form
$(\mc{H}^{u,\tilde{\mf{s}_1},e},d^{\tau,\mf{s}_1,e})$ in $\mc{H}^{u,\mf{s}_1}$
to corresponding
normalized extended affine Hecke algebra summands of
$\mc{H}^{u,\mf{s}_2}$ by essentially strict isomorphisms.

Given an orbit $\mathcal{O}$ of standard cuspidal unipotent pairs $\mf{s}$ of $G^u$
for action of the group $\Omega^\theta_C$ one can form the crossed product algebra
\begin{equation}
\mathcal{H}^{u,\mathcal{O}}=(\bigoplus_{\mf{s}\in\mathcal{O}}\mc{H}^{u,\mf{s}})\rtimes\Omega^\theta_C
\end{equation}
Then $\mathcal{H}^{u,\mathcal{O}}$ is Morita equivalent to the direct sum 
$\mc{H}^{u,\mf{s}}\rtimes\Omega^{\mf{s},\theta}$.
If $(V,\pi)$ is an object of $R(G^u(k))_{uni}$, let $V_{\mf{s}}$ denote the
$\mf{s}$-isotypical component of $V|_{\mathbb{P}^{F_u}}$
(where $\mf{s}=(\mathbb{P},\delta)$), and put 
$V^\mf{s}=\textup{Hom}_\mathbb{P}(\delta,V|_\mathbb{P})$. Then
\begin{equation}
V^{u,\mc{O}}=\oplus_{\omega\in\Omega^\theta/\Omega^{\mf{s},\theta}}(\pi(\omega)V^\mf{s})=
\oplus_{\mf{s}^\prime\in\mc{O}}V^{\mf{s}^\prime}=\oplus_{\mf{s}^\prime\in\mc{O}}e^{\mf{s}^\prime}V^{u,\mc{O}}
\end{equation}
is a representation of $\mc{H}^{u,\mathcal{O}}$ (see also paragraph \ref{par:uniprep}).
Here $e^{\mf{s}^\prime}$ denotes the unit element of $\mc{H}^{u,\mf{s}^\prime}$.

The Pontryagin dual $X^*_{un}(G^{F_u})=(\Omega^\theta_C)^*$ of $\Omega^\theta_C$
acts in a natural way on the algebra $\mc{H}^{u,\mc{O}}$ by automorphisms as follows.
If $\chi\in(\Omega^\theta_C)^*$ then the corresponding automorphism $\alpha_\chi$
acts as the identity on the subalgebra $\oplus \mc{H}^{u,\mf{s}}$, while
$\alpha_\chi(\omega)=\chi(\omega)\omega$.
If $\chi\in(\Omega^\theta_C/\Omega^{\mf{s},\theta})^*$ (i.e. $\chi|_{\Omega^{\mf{s},\theta}}=1$)
then $\alpha_\chi$ is the inner automorphisms obtained by conjugation with
$\sum_{\omega\in\Omega^\theta/\Omega^{\mf{s},\theta}}\chi(\omega)e^{\omega\mf{s}}$.
In particular the subgroup $(\Omega^\theta_C/\Omega^{\mf{s},\theta})^*$ of $X^*_{un}(G^{F_u})$
acts trivially on the set of irreducible representations of $\mc{H}^{u,\mc{O}}$.

The results of this paragraph can be summarized as follows:
\begin{thm}\label{thm:summary}
Let $G$ be a connected absolutely quasi-simple $K$-split,
$k$-quasisplit linear algebraic group. Consider the cartesian product
$R(G)_{uni}:=\prod_u R(G^{F_u})_{uni}$, where $R(G^{F_u})_{uni}$ denotes
the category of unipotent representations of $G^{F_u}$, and
where the product is taken over a complete set of representatives of
classes of pure inner $k$-forms $[u]\in H^1(F,G)$ of $G$.
Let $\mc{M}$ be the category of modules over the direct sum of algebras
$\mc{H}_{uni}:=\oplus_{u,\mc{O}}\mc{H}^{u,\mc{O}}$, where the direct sum is taken over the
a complete set of representatives of classes of inner $k$-forms 
$[u]\in H^1(F,G)$ of $G$ and $X^*_{un}(G)$-orbits $\mc{O}$ 
of standard cuspidal unipotent pairs $\mf{s}$ of $G^u$.
Consider the functor
$U:R(G)_{uni}\to\mc{M}$
defined by sending $V$ to $\oplus_{u,\mc{O}}V^{u,\mc{O}}$.
\begin{enumerate}
\item[(i)] The functor $U$ is an equivalence of categories.
\item[(ii)] For each orbit $\mc{O}$ of
standard cuspidal unipotent pairs of $G^u$ and each $\mf{s}\in\mc{O}$,
the irreducible spectrum of
$\mc{H}^{u,\mc{O}}$ is in canonical Morita bijection with the
irreducible spectrum of $\mc{H}^{u,\mf{s}}$. In turn this
equals the
disjoint union of the irreducible spectra of the
direct summands $\mc{H}^{u,\tilde{\mf{s}},e}$ of $\mc{H}^{u,\mf{s}}$,
where $\tilde{\mf{s}}$ runs over the collection of distinct extensions of $\mf{s}$ to
$N\mathbb{P}^{F_u}$ (this collection is a $(\Omega^{\mf{s},\theta}_1)^*$-torsor).
We define the tempered spectrum and spectral measure of $\mc{H}^{u,\mc{O}}$ via
these canonical bijections.
\item[(iii)] The bijection $[U]$ that $U$ induces on the irreducible spectrum
restricts to a homeomorphism $[U]^{temp}$
from the disjoint union of the
tempered unipotent spectra of the classes of pure inner forms $G^u$ of $G$
to the disjoint union of the tempered spectra of the various $\mc{H}^{u,\mc{O}}$.
\item[(iv)] The push forward of the union of the Plancherel measures
of the various $G^{F_u}$ under the bijection $[U^{temp}]$ is the union of
the spectral measures of the various $\mc{H}^{u,\mc{O}}$.
\item[(v)] For each $\mf{s}\in\mc{O}$ the action of $X^*_{un}(G)=(\Omega^\theta_C)^*$
on the irreducible spectrum of $\mc{H}^{u,\mc{O}}$ is trivial on the
subgroup $(\Omega^\theta_C/\Omega^{\mf{s},\theta})^*$.
The quotient $(\Omega^{\mf{s},\theta})^*$ of $(\Omega^\theta_C)^*$ acts on
the spectrum of $\mc{H}^{u,\mc{O}}$ via the canonical
Morita bijection of this set with the spectrum
of $\mc{H}^{u,\mf{s}}$ (which is naturally a $(\Omega^{\mf{s},\theta})^*$-set
by Corollary \ref{cor:equiv1}). 
\item[(vi)] 
The group $X^*_{un}(G)=(\Omega^\theta_C)^*$ acts on
$\mc{H}^{u,\mc{O}}$ via spectral automorphisms. 
In particular, this action induces a measure 
preserving action on the tempered spectrum of $\mc{H}^{u,\mc{O}}$. 
Moreover, via the bijection $[U^{temp}]$
this action corresponds with the natural action of $X^*_{un}(G)$ on $R(G)_{uni}$ 
by taking tensor products.
\end{enumerate}
\end{thm}
\section{The spectral transfer category of unipotent Hecke algebras}
\subsection{Spectral transfer morphisms}
Recall the notion of a spectral transfer morphism (STM) $\phi:\mc{H}_1\leadsto\mc{H}_2$
between two normalized affine Hecke algebras as introduced in 
\cite[Definition 5.1, Definition 5.9]{Opd4}. In this section we will classify the STMs between 
unipotent affine Hecke algebras (which will be referred to as ``unipotent STMs").
\subsubsection{Restriction of STMs}\label{subsub:res}
Let $(\mc{H},\tau)$ denote a normalized affine Hecke algebra, and let $L$ denote a 
generic residual coset $L\subset T$ for $\mc{H}$. Then there exists a unique ``parabolic subsystem"   
$R_P\subset R_0$ such that $L$ can be written in the form $L=rT^P$ with $r\in L\cap T_P$.
After moving $L$ with a suitable Weyl group element $w\in W_0$ we may assume 
that $R_P$ is standard and associated with a subset $P\subset F_0$.  To such subset 
we may associate a subalgebra $\mc{H}^P$ (``a standard Levi subalgebra") and its semisimple 
quotient algebra $\mc{H}_P$ whose associated algebraic torus is the subtorus $T_P\subset T$
(cf. \cite{Opd1}). 
In this situation $\{r\}\subset T_P$ is a residual point for $\mc{H}_P$. 
\begin{defn}\label{defn:normpar}
We will normalize 
the affine Hecke algebra $\mc{H}^P$ by the trace $\tau^{P}$ 
defined by $\tau^{P}(1)=\tau(1)$. We normalize $\mc{H}_P$ by 
the trace $\tau_P$ defined by the property 
$\tau_P(1)=(v-v^{-1})^{\textup{rk}(R_0)-\textup{rk}(R_P)}\tau(1)$ 
\end{defn}
Suppose that $\phi:(\mc{H}',\tau')\leadsto(\mc{H},\tau)$
is a strict STM which is represented by 
$\phi_T:T'\to L_n$ with $L=rT^P$ a residual coset. 
By modifying the representing map $\phi_T$ appropriately we may assume that 
that $rK_L^n=\phi_T(e)$ and such that $D\phi_T(\mf{t}')=\mf{t}^P$ for some subset 
$P\subset F_m$. 
It follows easily from Corollary \cite[Corollary 5.7]{Opd4}  and Corollary 
\cite[Corollary 5.8]{Opd4} that for any inclusion $P\subset Q\subset F_m$, 
after possibly modifying the representing morphism $\phi_T$ by a Weyl group element 
again, the inverse image  
$\phi_T^{-1}(K_L^n(L\cap T_Q)/K_L^n)\subset T'$ is a \emph{subgroup} whose identity component 
is a subtorus of $T'$ with as Lie algebra a subspace 
of $\mathfrak{t}':=\textup{Lie}(T')$ of the form $\mf{t}_{Q'}$ for some 
standard parabolic subsystem $Q'\subset F_m'$.  
Indeed, in Corollary \cite[Corollary 5.7]{Opd4}  we saw that 
$D\phi_T$ induces a bijective correspondence between parabolic subsystems 
$R_{Q'}$ of $R_m'$ and parabolic subsystems $R_Q$ of $R_m$ containing 
$R_P$. By modifying $\phi_T$ with 
an appropriate Weyl group element $w'\in W(R'_m)$ we may assume that 
$R_{Q'}=(D\phi_T)^{-1}(R_Q)$ is standard, associated to a subset $Q'\subset F_m'$. 
By Definition \ref{defn:normpar} and the definition of an STM 
it is easy to see that in this context, 
$\phi_T$ also defines an STM 
$\phi^Q:{\mc{H}'}^{Q'}\leadsto \mc{H}^Q$, and that 
the restriction $\phi_{T,Q}$ of $\phi_T$ to $T'_{Q'}\subset T'$ 
defines an STM 
$\phi_Q:{\mc{H}'}_{Q'}\leadsto \mc{H}_Q$ (Recall that $\dot{K}_L^n=N_{W_P}(L)$, 
and so $L_n$ is also the image of $\phi^Q$.
If $T_Q^P$ denotes the identity component of $T_Q\subset T^P$
then $L_Q:=rT_Q^P\subset T_Q$ is a residual coset of $T_Q$. Thus 
$L_{Q,n}=L_Q/K_L^n\cap T^P_Q$ so that $L_{Q,n}\subset L_n$.  
Hence $L_{Q,n}$ is the image of the restriction of $\phi_Q$).
\begin{defn}
We call $\phi_Q$ the restriction of $\phi$ to 
${\mc{H}'}_{Q'}$, and we say that $\phi$ is \emph{induced}  
from $\phi_Q$. In particular, $\phi$ is induced  by the 
rank $0$ STM $\phi_P:\Lb\leadsto \mc{H}_P$.
\end{defn}
\subsubsection{Induction of unipotent STMs}\label{subsub:ind}
By the above, every STM is induced from a rank $0$ 
transfer map. The converse is clearly not true: not every rank $0$ STM of 
the form $\psi: \mc{H}''=\Lb\to \mc{H}_P$ is the restriction of an STM 
$\Psi:\mc{H}'\leadsto\mc{H}$. Indeed, if $\textup{Im}(\psi)=r$ (a generic residual point of the 
subtorus $T_P\subset T$) then we should have $\textup{Im}(\Psi)=L=rT^P$.
But the spectral measure $\nu_{Pl}$ on a component 
$\mathfrak{S}_{(P,\delta)}$, where $\delta$ is a discrete series representation 
of $\mc{H}_P$ with central character $W_Pr$, is given (up to 
a rational constant depending on $\delta$) by the restriction of the 
regularisation $\mu^{(L)}|_{L^{temp}}$ of the $\mu$-function to 
$S_{(P,\delta)}=W_0\backslash W_0L^{temp}$ (cf. Theorem \cite[Theorem 4.13]{Opd4}).
This regularisation does in general not behave like a $\mu$-function 
of an affine Hecke algebra, unless for every restricted root of $R_0\backslash R_P$ 
to $L$ the appropriate cancellations occur. However, as we will see in 
\ref{subsub:indres}, if $\mc{H}=\mc{H}^{IM}(G)$ 
for a quasi-split almost simple algebraic groups $G$, and $\mc{H}''$ is the 
normalized Hecke algebra for a maximal cuspidal unipotent pair $(\mathbb{P},\sigma)$ of an  
inner form of the standard Levi subgroup of $G$ associated to $P\subset F_0$, then 
$\psi$ \emph{will} be the restriction of an STM $\Psi:\mc{H}'\leadsto\mc{H}$.
 \subsubsection{Induction and cuspidality of unipotent STMs}\label{subsub:indres}
 This brings us to an informal discussion of the heuristic ideas and surprising facts 
 behind the notion of STMs between unipotent affine 
 Hecke algebras (with their canonical normalizations as in \ref{subsub:normunipha}). 
We refer to such STMs as ``\emph{unipotent STMs}".
 Let $G^u$ be an absolutely almost simple unramified group over $k$, 
 and let $G$ denote a $k$ quasi-split group in the same inner class.  We fix a maximal $K$-split 
 torus $S\subset G$ defined over $k$. Let the automorphism induced by the action of the 
 Frobenius $F$ of $G$ on the character lattice of $S$ be denoted by $\theta$. We will 
 assume that  $u=\dot\omega\in N\mathbb{B}$ is a 
 representative of an element $\bar\omega\in\Omega/(1-\theta)\Omega$ (as in 
 paragraph \ref{par:uniprep}). We choose a minimal $F$-stable parabolic subgroup $A_0\subset G$. 
 These data give rise to the ``local index" of $G^F$, a (possibly twisted) affine Dynkin diagram which 
 contains a hyperspecial node, whose underlying finite root system is the restricted root 
 system of $G^{F}$ with respect to the $k$-split center of $A^{F}_0$ (again, we 
 apologize for denoting the restricted roots of $G(k)$ as ``coroots"). 
 We can now ``untwist" the affine 
 diagram by doubling  some of the restricted roots of $G(k)$; the resulting 
 root system is denoted by $R_0^\vee$. We have thus associated a based root datum 
 $\mc{R}:=(R_0,X,R_0^\vee,Y, F_0^\vee)$ such that the ``untwisted" local index 
of $G(k)$ equals $(R_0^\vee)^{(1)}$, and $u$ acts on this 
 affine diagram via the action of $\omega$ as a special affine diagram automorphism.
 Notice that $u$ acts naturally on the root system $R_0$ by means of an element 
 $w_u\in W_0(R_0)$.   
 The local index comes equipped with integers $m_S(a_i)$ attached to the nodes $a_i$, 
 which we transfer unaltered to the nodes of the untwisted diagram. This is 
 the arithmetic diagram $\Sigma_a(\mc{R},m)$ of \cite[Subsection 2.3]{Opd4} associated
 to $\mc{H}^{IM}$.  The associated spectral diagram $\Sigma_s(\mc{R},m)$ 
 is an untwisted affine Dynkin diagram for the affine root system 
 $\mc{R}^m=R_m^{(1)}$. Let $T$ be the complex algebraic torus 
 with character lattice $X$.
 
 The first remarkable fact is that for a cuspidal unipotent representation $\sigma_u$ of $G^u$, 
 its formal degree equals (up to a rational constant) the formal degree of an Iwahori-spherical 
 unipotent discrete series representation $\delta$ of $G$, and the central character $W_0r$ 
 of the corresponding discrete series representation $\delta_\sigma$ 
 of the Iwahori Hecke algebra $\mc{H}^{IM}:=\mc{H}^{IM}(G,\mathbb{B})$ of $G$
 is uniquely determined by this formal degree, up to the action of $X^*_{un}(G)$. 
 Here $r$ is a generic residual point.  
 (This is the rank $0$ case of Theorem \ref{thm:unique} that we already mentioned
 above). Let us agree to call a residual point $r$ of $\mc{H}^{IM}$ \emph{cuspidal} if 
 the $q$-rational factor of the residue $\mu^{IM,(\{r\})}(r)$ equals the $q$-rational factor 
 of the formal degree of a cuspidal unipotent representation $\sigma^u$ for an 
 inner form $G^u$ of $G$ as above.

 We claim that this is also true 
 if we replace $G^u$ by a proper Levi subgroup $M^u=C_{G^u}(S^u)^0$ of $G^u$ (with $S^u\subset G^u$ 
 the $k$-split part of the connected center of $M^u$) which carries a cuspidal unipotent representation 
 $\sigma^u_M$, in the following sense. 
 We may assume that $S^u\subset S$, the subtorus of $S$ 
 defined by the vanishing of the $K$-roots of $M^u$. 
 Then $S^u\subset S$ also gives rise to a $k$-Levi subgroup $M=C_G(S^u)^0$ of $G$ with 
 connected center $S^u$.  Observe that $M$ is $k$-quasisplit itself, and that $M^u$ is an 
 inner form of $M$ (since $\dot\omega\in M$). 
 
 Let $R_M^\vee\subset R_0^\vee$ denote the set of 
 (restricted) $K$-roots of $M^u$. Since $\sigma^u$ is unipotent, it factors through a cuspidal unipotent 
 representation $\sigma^u_M$ of the quotient $M^u_{ssa}:=M^u/S^u$. (This quotient consists of an almost 
 product of a semisimple group and a central anisotropic torus.)
 Then $\sigma^u_M$ first of all uniquely determines 
 an orbit of cuspidal residual points $W_Mr_M\subset T_M$ of $\mc{H}^{IM}_M$, up to the action of the  
 finite subgroup of $W_M$-invariant characters $\Omega_M^*$ of $X_M/\mathbb{Z}R_M$ of 
 $T_M$ (which contains the  
 group $K_M:=T_M\cap T^M$). This should still be true if the rank $0$ 
 case of Theorem \ref{thm:unique} holds, eventhough $M_{ssa}:=M/S^u$  is not absolutely 
 quasi-simple in general.  
 Namely, all but at most one of the absolutely quasi-simple almost factors of 
 $M_{ssa}$ are of type $A$, and these type $A$ factors admit just one (up to twisting by 
 unramified characters) residual point. 
 The residual point $r_M$ is thus the image of the representing map $\phi_T$ for a  
unique cuspidal unipotent STM 
 $\phi_M:\mb{L}\leadsto\mc{H}^{IM}(M_{ssa})$, where $\mc{H}^{IM}(M_{ssa})$
 denotes the Iwahori-Matsumoro Hecke algebra of $M_{ssa}$ with respect to the 
 Iwahori subgroup $M\cap \mathbb{B}/(S^u\cap \mathbb{B})$ of $M_{ssa}$, 
 up to the action of $\Omega_M^*$. 
 In particular $W_Mr_M$ gives rise to a maximal finite type subdiagram  
 $J_{M,r_M}\subset \Sigma_s(\mc{R}_M,m_M)$ (the spectral diagram of 
 $M_{ssa}$, defined similarly as we did for $G$ in the text above).
 Namely, $J_{M,r_M}$ is determined by choosing $r_M$ appropriately inside the orbit $W_Mr_M$, 
 then $r_M=s_Mc_M\in T_{M,u}T_{M,v}$ with $s_M$ defining a vertex 
of every component of $\Sigma_s(\mc{R}_M,m_M)$. To obtain $J_{M,r_M}$ one needs to strike out 
these nodes  from $\Sigma_s(\mc{R}_M,m_M)$.
\emph{In all cases, a subset of such type $J_M$ fits in a unique way as an \emph{excellent}  
(cf. \cite{Lu4}) subset of $\Sigma_s(\mc{R},m)$.} 
Since $T^M\subset T$ is maximal subtorus  on which the dual affine roots in $J_M$ 
are constant, we see that the pair $(r_M,T^M)$ is uniquely 
determined from just the type of $M_{ssa}$ and the $q$-rational factor of the 
unipotent degree of $\sigma^u_{M}$, up to the action of $W(R_0)$ and 
the group $\Omega_M^*$. In particular it is determined by the inertial class 
of the cuspidal pair $(M^u,\sigma^u)$.


 The cuspidal pair $(M^u,\sigma^u_M)$ is associated to a unique ``extended type"  
 $\mf{s}:=(N\mathbb{P}^{F_u},\tilde\delta)$ in the sense of \cite{Mo2} (also see paragraph \ref{par:uniprep}), 
 where $\mathbb{P}\subset G$ is an $F_u$-stable parahoric subgroup such that 
 $\mathbb{P}^{F_u}\cap M^{F_u}$ is a \emph{maximal} parahoric of $M^{F_u}$, and such 
that  the set of affine roots associated to the parahoric subgroup $\mathbb{P}$ has a basis 
given by a proper $\omega$-invariant subset $\Sigma_a(\mc{R},m)$. The Plancherel 
measure on the set of tempered representations which belong 
to the unipotent Bernstein component whose cuspidal support is the inertial 
equivalence class of the cuspidal pair $(M^u,\sigma^u)$ is given by 
the Plancherel measure of the normalized unipotent affine Hecke algebra  
$\mc{H}^{u,\tilde{\mf{s}},e}$ (cf. e.g. \cite{BHK}, \cite{HOH}, \cite{Mo1}, \cite{Mo2}, \cite{Re}).

Let us now move $M$ to its standard position, so that $R_M^\vee$ 
is replaced by a standard parabolic 
subsystem $R_Q^\vee$ of roots associated to a subset $Q\subset F_0$.
This corresponds to a standard Levi subgroup 
$G^Q=C_G(S^u)$ of $G$ which is conjugate to $M$.
 Suppose that $\widetilde{\sigma^Q}$ is 
an Iwahori spherical representation of $G^Q$ which is tempered and $L^2$ modulo the center 
of $G^Q$. The corresponding tempered representation $\pi^Q$ of $\mc{H}_Q^{IM}$ is then the form 
$\pi^Q=(\pi_{Q})_t$ for some \emph{Iwahori spherical} discrete series representation 
$\pi_Q$ of $\mc{H}_Q^{IM}$ and some $t\in T^Q_u$. Assume now that the central character 
of $\pi_Q$ is a \emph{cuspidal} residual point of $T_Q$. This means by definition that there also 
exists a cuspidal unipotent representation $\sigma^u_M$ of some Levi subgroup $M^u_{ss}$ 
of some inner form $G^u$ as above, whose formal degree 
has the same $q$-rational factor as that of $\widetilde{\sigma_Q}$ (or equivalently of $\pi_Q$). 
The second important heuristic ingredient  we now apply is the general expectation that the 
$q$-rational factor of the formal degree of the members of a discrete series unipotent L-packet should 
be the equal for all members of the packet \cite{Re}.
It then follows from the above uniqueness assertions that $\sigma_Q$ and a twist $\sigma_Q^u$ of $\sigma^u_M$  
(i.e. $\sigma_Q^u$ is obtained from $\sigma^u_M$ by a pull-back via a $k$-isomorphism between $G_Q^u$ 
and $M^u$) must belong to the same $L$-packet of $G_Q$. 

 Recall from paragraph \cite[4.2.5]{Opd4} how the Plancherel measure  $\nu_{Pl}|_{\mathfrak{S}_{(P,\sigma)}}$
 on the component $\mathfrak{S}_{(P,\sigma)}$ of the tempered spectrum of $\mc{H}^{IM}$ accociated to a discrete series  
 representation $\sigma$ of $\mc{H}^{IM}_Q$ with central character $W_Qr_Q$, is expressed in terms of the 
 residue $\mu^{IM,(L)}$ where $L=r_QT^Q$ (see paragraph \cite[4.2.5]{Opd4}).
This implies that if discrete series representations $\sigma_1,\sigma_2$ of $\mc{H}_P$ are associated to the 
the same central character $W_Pr\in W_P\backslash T_P$,  then the components $\mathfrak{S}_{(P,\sigma_i)}$
($i=1,2$) are related to each other by a Plancherel measure preserving (up to a rational constant) 
correspondence as in Theorem \cite[Theorem 6.1]{Opd4}.  

The third heuristic idea is that \emph{such a correspondence should exist for Plancherel measures on the tempered 
components determined by any two discrete series induction data $(G_Q,\widetilde{\sigma_Q})$ 
and $(G_Q^u,\sigma_Q^u)$
whenever $\widetilde{\sigma_Q}$ and $\sigma_Q^u$ belong to the same L-packet}. 
But for the latter cuspidal unipotent pair, this Plancherel measure is computed as the most continuous part 
of the tempered spectrum of the normalized unipotent affine Hecke algebra  
$\mc{H}^{u,\tilde{\mf{s}},e}$.
 On the other hand, for the first pair it was already discussed above that the Plancherel measure can be computed 
essentially as the residue measure of the $\mu$-function of $\mc{H}^{IM}$ with respect to the tempered residual 
coset $L^{temp}=r_QT^Q_u$.  
Thus these ideas suggest the existence of a \emph{unique} STM 
$\mc{H}^{u,\tilde{\mf{s}},e}\leadsto \mc{H}^{IM}$ represented by a morphism 
$\phi_T$ with image $L=r_QT^Q$ associated to the cuspidal pair $(M^u,\delta^u)$ (or equivalently, to the 
extended type $\tilde{\mf{s}}$). 

Recall from Proposition \cite[Proposition 5.6]{Opd4} that any STM $\mc{H}'\leadsto\mc{H}$ 
represented by an affine morphism $\phi_T:T'\to T$
with image $L=rT^Q$, the subtorus $T^Q\subset T$ is $W_0$-conjugacte to a subtorus 
$T^J\subset T$ which is defined as above by an excellent subset $J$ of the spectral diagram 
of $\mc{H}$. Therefore it is clear that if there exists a unipotent STM 
$\phi:\mc{H}^{u,\tilde{\mf{s}},e}\leadsto \mc{H}^{IM}$ 
as expected by the above discussion, then its image must be uniquely determined by 
the type of $M_{ssa}$ in combination with the $q$-rational factor of 
the formal degree of $\sigma^u_M$, up to the action of $X^*_{un}(G)$.
By Proposition \cite[Proposition 7.13]{Opd4} we see that $\phi$ itself is therefore  determined up 
to the action of $\textup{Aut}_\mf{C}(\mc{H}^{u,\tilde{\mf{s}},e})^{op}\times X^*_{un}(G)$.

Any unipotent STM
$\Phi:\mc{H}^{u,\tilde{\mf{s}},e}\leadsto \mc{H}^{IM}$ is induced from a 
cuspidal unipotent STM 
$\phi:\mb{L}\leadsto\mc{H}^{IM}_Q$ which is uniquely determined 
modulo the action of $K_Q/K_L^n=(T_Q\cap T^Q)/K_L^n$ (where $L_n:=L/K_L^n$ 
denotes the image of a representing morphism $\phi_T$), this is obvious.
But by our discussion above we expect that:  
\emph{Conversely, any cuspidal unipotent STM  
$\phi:(\mb{L},\tau_0)=:\mc{H}_0\leadsto\mc{H}^{IM}_Q$ for the quotient 
$G_Q=G^Q/Z_{s}(G^Q)^0$ (where $Z_{s}(G^Q)^0$ is the connected 
$k$-split  center) of a standard Levi subgroup $G^Q$ can be 
induced to yield a unique spectral transfer morphism 
$\Phi:\mc{H}^{u,\tilde{\mf{s}},e}\leadsto \mc{H}^{IM}$.
Up to the action of $\textup{Aut}_\mf{C}(\mc{H}^{u,\tilde{\mf{s}},e})^{op}\times X^*_{un}(G)$, 
$\Phi$ is 
completely determined by the type of $G_Q$ and 
the $q$-rational factor of the degree $\tau_0(1)$ 
of $\mc{H}_0$.} 

In the above arguments two important aspects of cuspidal residual points played a role. The first is 
that they can be defined by the property that the associated residue degree $\mu^{IM,(\{r\})}(r)$
has the same $q$-rational factor as that of the formal degree of a cuspidal unipotent 
representation of some inner form of $G$. The second is that a cuspidal residual 
point $r_Q$ of a semisimple quotient Hecke algebra $\mc{H}_Q$ of $\mc{H}^{IM}$ 
(where cuspidal means now that 
there exists a Levi subgroup $M^u$ of an inner form $G^u$ of $G$ which carries a cuspidal 
unipotent representation $\sigma^u$ and which is isomorphic to an inner form of $G_Q$, 
such that the $q$-rational factor of the formal degree of $\sigma^u$ is equal to the
that of the residue of $\mu^{IM}_Q$ at $r_Q$) is always the restriction of an STM 
$\mc{H}^{u,\tilde{s},e}\leadsto\mc{H}^{IM}$, for any inclusion of $(\mc{R}_Q,m_Q)$ 
as a standard 
parabolic subsystem of the based root datum $(\mc{R},m)$ (with parameter funtion) of 
$\mc{H}^{IM}(G)=\mc{H}(\mc{R},m)$. 
A priori the second property seems much more restrictive (except for the  
``final" exceptional groups $\textup{E}_8$, $\textup{F}_4$ and $\textup{G}_2$), but  
miraculously these properties lead to the same 
notion of cuspidality. The essential uniqueness part of Theorem \ref{thm:unique} 
reduces to the rank $0$ (or cuspidal) case in this way. The cuspidal case is done 
by direct inspection for the exceptional groups (most of the required results are in 
\cite{Re0}, \cite{Re}, \cite{HOH}). For the classical groupes, the cuspidal case 
is treated in \cite{FO}.

Of course the arguments above are only heuristic, but they tell us precisely 
where we should expect STMs, how these should be defined by induction from cuspidal 
ones,  and what is necessary to check in order to prove that these maps really 
are STMs (thus providing a proof of Theorem \ref{thm:unique}). 
\emph{In the remainder of this 
paper we will prove that indeed, any unipotent cuspidal pair $(\mathbb{P}^u,\sigma)$ 
of an inner form $G^u$ gives rise in this way to an essentially unique 
STM $\mc{H}^{u,\tilde{s},e}\leadsto\mc{H}^{IM}$}, thereby proving Theorem \ref{thm:unique}
in full generality.

For exceptional groups the required verifications that induction of cuspidal STMs from Levi 
subgroups always gives rise to STMs is based on the notion of a ``transfer map diagram". 
This notion is defined and discussed in paragraph \ref{subsub:dia}. 
One can also study more generally the STMs 
between two unipotent affine Hecke algebras, not just the ones with $\mc{H}^{IM}$ 
as a target. This is interesting in itself, since in several cases the ``unipotent 
spectral transfer category" is generated by very simple building blocks of this kind.
Indeed, this is how we show the existence of STMs induced from cuspidal ones in the 
classical cases.
\subsubsection{The transfer map diagram of a unipotent STM}\label{subsub:dia}
Such an expected unipotent STM 
$\Phi:\mc{H}^{u,\tilde{\mf{s}},e}\leadsto \mc{H}^{IM}$ is determined 
(up to the action of $\textup{Aut}_\mf{C}(\mc{H}^{u,\tilde{\mf{s}},e})^{op}$)
by the image  $r_Q=s_Qc_Q$ of $\phi_{T,Q}$, a cuspidal generic 
residual point for the Iwahori-Matsumoto 
Hecke algebra $\mc{H}^{IM}_Q$ of the 
quasisplit Levi subgroup $G_Q$ of $G$. We can choose  
the unitary part $s_Q=s(e_Q)\in T_{Q,u}$ such that it 
corresponds to a vertex $e_Q\in C^{Q,\vee}$, the fundamental 
alcove for dual affine Weyl group $(W_Q)^\vee_{m_Q}$ 
associated to $(\mc{R}_Q,m_Q)$. Let $v_Q$ be the set of 
corresponding nodes of the spectral diagram $\Sigma_s(\mc{R}_Q,m_Q)$, 
and put $J_Q$ for the finite type Dynkin diagram that is the complement 
of $v_Q$ of  $\Sigma_s(\mc{R}_Q,m_Q)$. Let $\mc{R}'$ denote 
the root datum underlying $\mc{H}^{u,\tilde{\mf{s}},e}$, with multiplicity 
function $m'$. 

A node $v_i$ of the complement $J_Q$ 
is weighted with the weight $w_i:=Da_i^\vee(c_Q)$ 
of the gradient $Da^\vee_i$ of the corresponding 
dual affine root $a_i^\vee$ (this value is a power of $q$). 
We may put $c_Q$ in dominant position with respect to 
the roots $Da_i^\vee$ where $i$ runs over the nodes of $J_Q$.
This is essentially the weighted Dynkin diagram of a linear 
generic residual point (in the sense of \cite{OpdSol2}, but obviously 
restricted in our context of the fixed line in the parameter space 
defined by $m^\vee_R$) for the 
finite type root system defined 
by $J_Q$ with the parameters $m_Q^\vee|_{J_Q}$ (in a 
multiplicative notation). 

As was remarked above, if the rank of $\Phi$ is positive, 
a finite type Dynkin diagram of type 
$J_Q$ fits uniquely as an excellent subdiagram $J$ of the spectral diagram 
$\Sigma_s(\mc{R},m)$
associated to $G$ (this can be checked case by case), up to the 
action of $X^*_{un}(G)$. 
Now we also assign weights to the nodes of $\texttt{K}=I\backslash J$
as follows. 
By modifying $\phi_T$ (within its equivalence class) by 
an element of the Weyl group $W'=W(R_0')$ we  may assume that via 
$D\phi_T$ the affine simple reflections of $(R_m')^{(1)}$ (relative to the base 
$(\mc{F}')^m$ of $(\mc{R}')^m$)
correspond bijectively to the elements of $\texttt{K}$. This allows us to use $\texttt{k}\in \texttt{K}$ 
also to parameterize the elements of the base of $(R_m')^{(1)}$. 
Let $\texttt{k}_0\in \texttt{K}$ be the 
vertex of the unique (dual) affine simple root which is not in $F_m'$. From (T2), 
Proposition \cite[Proposition 5.6]{Opd4}(4) and Corollary \cite[Corollary 5.8]{Opd4}  (applied to the 
case $Q'=\emptyset$) we see that this is the unique element $\texttt{k}_0\in \texttt{K}$ for which 
the corresponding vertex $D^a\phi_T(0)=\omega_{\texttt{k}_0}\in C^\vee$ has the shortest 
length. 
We interpret the gradient $Da^\vee$ of 
a (dual) affine root $a^\vee\in R_m^{(1)}$ as a character on $T$ (and 
similarly for dual affine roots of $(R_m')^{(1)}$ on $T'$). 
By construction, the character lattice of $L_n$ is mapped injectively 
to a sublattice of $X^Q_m$ and injectively to a sublattice of $X'_m$.
From (T3), \cite[equation (8)]{Opd4} and considering the numerator of the 
$\mu$-function (see Definition (\cite[Definition 3.2]{Opd4}) 
it is easy to see that $D\phi_T^*(Da^\vee_\texttt{k})$ must be 
a rational multiple $D\phi_T^*(Da^\vee_\texttt{k})=f_\texttt{k}Db^\vee_\texttt{k}$ of a root $b^\vee_\texttt{k}\in (F_m')^{(1)}$.
This sets up a bijection between the set of simple affine roots of $(F_m')^{(1)}$ and the set 
$\texttt{K}=I\backslash J$, and using this we will parameterize the elements of $(F_m')^{(1)}$ 
also by the set $\texttt{K}$. By Proposition \cite[Proposition 5.6]{Opd4}(3) this bijection defines an isomorphism 
of affine reflection groups. By Proposition \cite[Proposition 5.6]{Opd4}(4) it is then clear that 
$b^\vee_{\texttt{k}_0}$ has to be the extending affine root of the spectral  
diagram of $\mc{H}(R'_m,m')$. 
And we can say more precisely, by considering the formula of the $\mu$ function of the 
Hecke algebra and (T3), that 
$f_\texttt{k}^{-1}\in\mathbb{N}$, and that we can thus interpret the fraction $Db^\vee_\texttt{k}$ 
as the character $f_\texttt{k}^{-1}D\phi_T^*(Da^\vee_\texttt{k})$ of $T^Q$. 
Now $L$ itself is a coset of 
$T^Q$ with origin $r_Q=s_Qc_Q$, and using the above remarks it follows that 
for all $\texttt{k}\in \texttt{K}$, $Da^\vee_\texttt{k}$ lifts to a constant multiple of a character of a suitable 
covering of $T'$ (namely the fibered product of 
$L$ and $T'$ over $L_n$). We call this lift of $Da^\vee_\texttt{k}$, expressed as a radical of $b_\texttt{k}^\vee$, the \emph{weight} $w_\texttt{k}$ of $\texttt{k}\in \texttt{K}$. 
In view of (T3) and \cite[equation (8)]{Opd4}, and using Proposition \cite[Proposition 5.6]{Opd4} we see that:  
$w_\texttt{k}:=\zeta_\texttt{k} v^{c_\texttt{k}} (Db_\texttt{k}^\vee)^{f_\texttt{k}}$, where $\zeta_\texttt{k}=1$ if $\texttt{k}\not= \texttt{k}_0$, $\zeta_{\texttt{k}_0}$ is 
a $f_{\texttt{k}_0}^{-1}$-th primitive root of $1$, and $c_\texttt{k}\in \mathbb{Z}$ (which can be computed 
by evaluating $Da_\texttt{k}^\vee$ on $c_Q$). 
All this gives rise to the following notion:
\begin{defn}\label{def:transmorfdiagr} 
Given a unipotent STM 
$\Phi:\mc{H}^{u,\tilde{\mf{s}},e}\leadsto \mc{H}^{IM}$, 
the spectral diagram $\Sigma_s(\mc{R},m)$ of $G$ with 
the vertices of the excellent subdiagram $J$ marked with 
the constant weights $w_j$, and the remaining vertices of $\texttt{K}$ 
labelled with their weights $w_\texttt{k}$ as above, is called the transfer map 
diagram of $\Phi$.
\end{defn}
Observe that $\prod_{i\in I}w_i^{n_i}=1$, 
where $1=\sum_{i\in I}n_i a^\vee_i$ is the decomposition of the constant function $1$ 
as a linear combination of (dual) affine simple roots of $\mc{R}^m$ in terms of the base 
of simple roots $F_m^{(1)}$. In particular, there exists a constant $C$ such that 
for all $\texttt{k}\in \texttt{K}$, $n_\texttt{k}f_\texttt{k}=C\tilde{n}_\texttt{k}$,
where  $\sum_{\texttt{k}\in \texttt{K}}\tilde{n}_\texttt{k}b_\texttt{k}^\vee=1$ is the decompostion of the constant function $1$ 
in terms of the (dual) affine simple roots $(F_m')^{(1)}=(\mc{F}')^m$ of $(\mc{R}')^m$. Clearly the 
value of $Da^\vee_{\texttt{k}_0}$ on $s_Q=\omega_{\texttt{k}_0}$ is a primitive $n_{\texttt{k}_0}$-th root of $1$. 
Therefore 
we see that $C=1$, and $f_\texttt{k}^{-1}=n_\texttt{k}/\tilde{n}_\texttt{k}$ (this integer is called $z_\texttt{k}$ by Lusztig 
\cite[Section 2]{Lu4}); 
by Proposition \cite[Proposition 5.6]{Opd4}, we are in the setting of \cite[Section 2]{Lu4} and we 
may therefore use the results of loc. cit. paragraph 2.11 to 2.14.  For example, by carefully analyzing the 
Cartan matrices it follows that if $b_k^\vee, b_{k'}^\vee$ are connected by a single edge, 
then $f_\texttt{k}^{-1}=f_{\texttt{k}'}^{-1}$.  Moreover, 
$f_\texttt{k}^{-1}$ is a divisor of $f_{\texttt{k}_0}^{-1}$ for all $\texttt{k}\in \texttt{K}$, except possibly 
if $k, k_0\in \texttt{K}^\flat$, when we may have $f_\texttt{k}f_{\texttt{k}_0}^{-1}\in \{1,(1/2)^{\pm 1}\}$. 
We also note that, from the tables in \cite[Section 7]{Lu4} and \cite[Section 11]{L5}, for 
all $\texttt{k}\in\texttt{K}$: $n_\texttt{k}/n_{\texttt{k}_0}\in\mathbb{Z}$.
By the above it is clear that  $\Phi$ is completely determined by its 
transfer map diagram. 

The finite abelian group $K_L^n\subset T_L^{W_L}$ can be recovered from the 
transfer map diagram as the product over all $\texttt{k}\in \texttt{K}\backslash \{\texttt{k}_0\}$ of cyclic 
groups $C_\texttt{k}$ of order $z_\texttt{k}=f_\texttt{k}^{-1}$ (for those $\texttt{k}\in  \texttt{K}\backslash \{\texttt{k}_0\}$ for which 
$n_{m'}(b_\texttt{k})=1$) or of order $z_\texttt{k}/2$ (for $\texttt{k}\in  \texttt{K}\backslash \{\texttt{k}_0\}$ such that  
$n_{m'}(b_\texttt{k})=2$ and $z_\texttt{k}$ is even). For classical groups $K_L^n$ is always trivial. 
\subsection{Main theorem}
We finally have everything in place to formulate the two main theorem of this paper.
Let $\bf{G}$ be a connected absolutely quasi-simple $K$-split, $k$-quasi-split
linear algebraic group. \emph{For simplicity we will assume that $G$ is 
af adjoint type.}
Recall that $\mc{H}^{IM}$ denotes the Iwahori-Matsumoto Hecke algebra
of $G=\mb{G}(K)$, i.e. the generic affine Hecke algebra $\mc{H}^{IM}=\mc{H}^{IM}(G)$
such that $\mc{H}^{IM}_\mb{v}$ is the Iwahori-Matsumoto Hecke algebra of $G(k)=G^{F}$
with respect to the the standard cuspidal unipotent pair $\mf{s}_0:=(\mathbb{B},1)$
where $\mathbb{B}$ denotes the Iwahori subgroup. Since $\mf{s}_0$ is fixed
for the action of $N\mathbb{B}^F/\mathbb{B}^F=\Omega_C^\theta$ the orbit $\mc{O}_0$
of $\mf{s}_0$ equals $\mc{O}_0=\{\mf{s}_0\}$. We have
$\Omega^{\mf{s}_0,\theta}=\Omega^{\mf{s}_0,\theta}_2$.
Its trace $\tau^{IM}$ 
is normalized as in (\ref{eq:deg}), i.e. 
\begin{equation}
\tau^{IM}(e^{IM})=\textup{det}_{V}(\mathbf{v}\textup{Id}_{V}-\mathbf{v}^{-1}w_u\theta)^{-1}
\end{equation}
where $V=\mathbb{R}Y$, $\theta$ denotes the action on $Y$ of the outer automorphism of $G^\vee$  
corresponding to $F$, and $w_u\in W_0$ is the image of $u\in \Omega_C\subset W$ under the canonical 
projection $W\to W_0$. 
Observe that $(\mc{H}^{IM},\tau^{IM})$
is a direct summand of $\mc{H}_{uni}(G)$, namely the
unique summand of maximal rank. It corresponds to the Borel component 
of $G^{F}$, the Bernstein component corresponding to the 
cuspidal unipotent representation $1$ of a minimal 
$F$-Levi subgroup $M$ of $G$. 

From Theorem \ref{thm:summary} the group $X^*_{un}(G)$ acts by (spectral)
transfer automorphisms on $(\mc{H}_{uni}(G),\tau)$.
In particular $X^*_{un}(G)$ acts by spectral automorphisms on $\mc{H}^{IM}$ too,
(see Proposition \cite[Proposition 3.5]{Opd4}) since $\mc{H}^{IM}$ is the unique summand
of $\mc{H}_{uni}$ of maximal rank.

\subsubsection{Notational conventions for Hecke algebras}\label{subsub:not}  
Recall Definition \cite[Definition 2.11]{Opd4} and recall that the spectral diagram can be expressed 
completely in terms of $R_0$ and of the $W_0$-invariant functions $m_{\pm}(\alpha)$ on $R_0$ 
defined by \cite[Equation (4)]{Opd4}. 
In the proof of the theorem below we will denote the unipotent normalized affine Hecke 
algebra of irreducible type $\mc{H}(\mc{R}^m,m)$, with $X_m$ the weight lattice of 
the irreducible reduced root system $R_m$, as follows. If $R_m$ is simply laced 
and the parameters $m_+(\alpha_\texttt{k})$ are equal to $\mathfrak{b}$ we denote this unipotent 
Hecke algebra by $R_m[q^\mathfrak{b}]$. If $R_m$ is not simply laced and not of type $\textup{C}_d$,  
then we will denote this algebra by 
$R_m(m_+(\alpha_1),m_+(\alpha_{2}))[q^\mathfrak{b}]$ where $\alpha_1\in F_m$ is long and $\alpha_2$
is short, and $q^\mathfrak{b}$ is the base for the Hecke parameters (equivalently, we could write 
$R_m(\mathfrak{b}m_+(\alpha_1),\mathfrak{b}m_+(\alpha_{2}))[q]$). 
If both parameters are equal to $\mathfrak{b}$ 
we may also simply write $R_m[q^\mathfrak{b}]$, this will not create confusion.
For $R_m=\textup{C}_d$ we will write 
$\textup{C}_d(m_-,m_+)[q^\mathfrak{b}]$ to denote the unipotent normalized affine Hecke algebra 
with $R_m=\textup{C}_d$ with $m_+(\alpha)=\mathfrak{b}$ for $\alpha$ a type $\textup{D}_d$ root of $R_m$, 
and for a short root $\beta$ of $\textup{B}_d$ we have $m_-(\beta)=\mathfrak{b}m_-$ and 
$m_+(\beta)=\mathfrak{b}m_+$. 
If $m_-(\beta)=0$ and $m_+(\beta)=m_+(\alpha)=\mathfrak{b}$ then we may also denote this 
case $\textup{C}_d[q^\mathfrak{b}]$.
\begin{thm}\label{thm:unique} 
Let $G$ be a connected, absolutely simple, quasi-split linear 
algebraic group of adjoint type, defined and unramified over a non-archimedean local field $k$. 
Let $\mf{C}_{uni}(G)$ be the full subcategory of the spectral transfer category $\mf{C}_{es}$ 
(with essentially strict STMs as morphisms) whose set of objects is the set of normalized 
unipotent affine Hecke algebras $\mc{H}^{u,\mf{s},e}$ associated with the various inner 
forms $G^u$ of $G$ (where $u\in Z^1(F,G)$ runs over a complete set of representatives of 
the classes $[u]\in H^1(F,G)$). 
Let $\mc{H}_{uni}$ denote the direct sum of all the objects of $\mf{C}_{uni}(G)$.
Recall that there is a natural action of $X^*_{un}(G)$
on $\mc{H}_{uni}$ such that direct summands are mapped to direct summands, preserving the 
rank.  In particular $X^*_{un}(G)$ acts on the unique summand $\mc{H}^{IM}(G)$ 
of $\mc{H}_{uni}$ of largest rank.

There exists a $X^*_{un}(G)$-equivariant STM
\begin{equation}
\Phi:(\mc{H}_{uni}(G),\tau)\leadsto(\mc{H}^{IM}(G),\tau^{IM})
\end{equation}
which is essentially unique in the sense that if 
$\Phi^\prime$ is another such equivariant STM then there 
exits a spectral transfer automorphism $\sigma$ of 
$(\mc{H}_{uni}(G),\tau)$ such that $\Phi^\prime=\Phi\circ\sigma$. 

The transfer map diagrams corresponding to the restrictions 
of $\Phi$ to the various 
direct summands $\mc{H}^{u,\tilde{s},e}$ of $\mc{H}_{uni}(G)$ are 
equal to the corresponding geometric diagrams of \cite{Lu4}.
\end{thm}
\begin{cor}\label{cor:lowest}
Recall that the spectral isogeny class of an object of $\mf{C}_{uni}(G)$ is equal to its 
isomorphism class \cite[Proposition 8.3]{Opd4}, and that these classes admit a 
canonical partial ordering $\lesssim$ as defined in \cite[Definition 8.2]{Opd4}.
Then $(\mc{H}^{IM}(G),\tau^{IM})\lesssim(\mc{H},\tau)$ 
for any object $(\mc{H},\tau)$ of  $\mf{C}_{uni}(G)$. 
\end{cor}
Theorem \ref{thm:unique} is a consequence of the combined results of the 
following subsections.
\subsubsection{The case of $G=\textup{PGL}_{n+1}$.}\label{subsub:A}
In this case,  
the only cuspidal unipotent representation 
comes from the anisotropic inner form $G^u=\mathbb{D}^\times/k^\times$ (where $\mathbb{D}$ 
is an unramified division algebra over $k$ of rank $(n+1)^2$) and has a formal degree 
with $q$-rational factor given by $\textup{fdeg}:=[n+1]_q^{-1}$
(cf. \ref{subsub:aniso}). It is 
obvious that there exists a unique cuspidal STM $(\mathbf{L},\textup{fdeg}):=A_0[q^{n+1}]\leadsto A_n[q]$, 
since $\textup{A}_n[q]$ has only one orbit of residual points (up to the action of $X^*_{un}(G)$) 
and this has indeed the desired residue degree. Based on this it is easy to construct 
the general STM for the unipotent types for this $G$, and prove that these 
are unique. Suppose we have a factorization $n+1=(d+1)(m+1)$. Consider  
an inner form $G^u$ of $G$ such that $u$ has order $m+1$. A maximal 
$k$-split torus $S\approx (k^\times)^{d}$ of $G^u$ defines a Levi group 
$M^u=C_{G^u}(S)$
such that $M^u_{ssa}=M^u/S$ is of type $(\mathbb{D}^\times/k^\times)^{d+1}$
where $\mathbb{D}$ is an unramified division algebra over $k$ of rank $(m+1)^2$.
Then $J$ is of type $\textup{A}_m^{d+1}$, which fits in a unique sense (up to diagram automorphisms 
as usual) as an essential subdiagram of the spectral diagram of $\mc{H}^{IM}$.
The Hecke algebra $\mc{H}^{u,\mf{s},u}$ is of type $\textup{A}_d[q^{m+1}]$.
For the unique \emph{strict} STM we make sure that $J$ does not contain $a^\vee_0$.
The weights for the vertices of $J$ are equal to $q$, and for those of $\texttt{k}\in I\backslash J$ 
equal to $q^{-m}Db_\texttt{k}$. It is an easy check that this indeed defines a strict  
STM $\textup{A}_d[q^{m+1}]\leadsto \textup{A}_n[q]$. 
The uniqueness of such strict STM up to 
$\textup{Aut}_{\mf{C}}(\mc{H}^{u,\mf{s},e})$ is clear as before: Any strict STM 
$\phi:\textup{A}_d[q^{m+1}]\leadsto \textup{A}_n[q]$ is obtained by induction of a cuspidal one for 
$M_{ssa}$, which determines $J$ 
and the underlying geometric diagram of  $\phi$ as before. There assignment 
of weights to the vertices of $\texttt{K}$ is dictated by the basic properties of an STM 
as explained above. Hence $\phi$ must be equal to the STM constructed above, 
up to a diagram automorphism of $\textup{A}_d$. By Theorem \ref{thm:summary} the direct summands 
of $\mc{H}^{u,\mf{s}}$ form a torsor for $(\Omega_1^{\mf{s},e})^*=\Omega^*/(\Omega_2^{\mf{s},e})^*\approx C_{m+1}$, 
and hence there is a unique way to write down a 
\emph{$\Omega^*$-equivariant} STM $\mc{H}^{u,\mf{s}}\leadsto\mc{H}^{IM}$, 
up to the action of $\textup{Aut}_{es}(\mc{H}^{u,\mf{s}})$.  
This completely finishes the proof for the case of $\textup{PGL}_{n+1}$.
\subsubsection{Existence and uniqueness of rank $0$ STMs for exceptional groups.}\label{subsub:rk0exc}
This is a case check 
(with some help of Maple, to simplify the product formulas for the $q$-factor of the formal degree 
as given in \cite{OpdSol2}), almost all of which has already been done 
in the existing literature. Let $G$ be an $k$-quasisplit adjoint group 
over with $k$ which is split over $K$, of type 
${}^3\textup{D}_4, \textup{E}_6, {}^2\textup{E}_6, \textup{E}_7,\textup{E}_8, \textup{F}_4, \textup{G}_2.$
One uses the classification of the residual 
points and the product formula for the $q$-rational factor of the formal degree 
from \cite{OpdSol2} to compute, for each orbit $W_0r$ of generic residual points 
(in the sense of the present paper), the $q$-rational factor of the residue degree 
$\mu^{{IM},{(r)}}(r)$ for the Iwahori-Matsumoto Hecke algebra $\mc{H}^{IM}$ of $G$.
(Many of these results are already in the literature; 
For $\textup{E}_8$ this list was given in \cite{HOH} using essentially 
the same method. For all split exceptional groups this list can be found 
in \cite{Re}. Note that the computations
in \cite{Re} can be simplified a lot 
using the classifications  
and the \emph{product formula} from \cite{OpdSol2},  
to just ``clearing $q$-fractions", since our formal degree formula is already 
given in ``product form" (as opposed to an alternating sum  
of rational functions as in \cite{Re}). Also note that we are for this list  
only interested in the Iwahori spherical case.) 
We note that these lists reveal that these residues of 
$\mu^{IM}$ at distinct orbits $W_0r\not=W_0r'$ are distinct for all exceptional 
cases. Hence in the exceptional cases the uniqueness 
(up to diagram automorphisms) of rank $0$ spectral transfer 
maps for irreducible unipotent Hecke algebras is  guaranteed by this.

The existence of the desired cuspidal unipotent STMs 
is now an easy task; one considers the list of all cuspidal unipotent representations 
of all inner forms of $G$. This means that we need to make a list of all maximal 
$F_u$-stable parahoric subgroups of the inner forms $G^u$, consider their 
reductive quotients over $\mb{k}$, and for those quotients which 
admit a cuspidal unipotent character, compute the normalization of the associated 
Hecke algebra $\mc{H}^{u,\tilde{\mf{s}},e}:=(\mathbf{L},\tau^{\mf{s},e})$ according to 
(\ref{eq:deg}). Of course the main part of this formula is the degree of the unipotent cuspidal 
characters of the simple finite groups of Lie type, which is due to Lusztig
 \cite{Lu1}, \cite{Lu2}, \cite{Lu3} and conveniently tabulated in \cite{C}. 
Finally we need to see if the $q$-rational parts of these expression show up in our list 
of residues of the $\mu$-function. This indeed leads to  
cuspidal transfer map diagrams with the same underlying sets 
$J$ as listed by Lusztig in \cite{Lu4} and \cite{Lu6}, 
and for each of those diagrams, there exists one generic linear 
residual point for $J$ (in the form of the collection of weights assigned 
to the vertices of $J$) producing the correct 
residue of $\mu$ and thus an STM.  

Let us give the results for the two non-split quasi-split cases which were not 
yet treated in the existing literature. The unipotent Hecke algebra $\textup{G}_2(3,1)[q]$ 
(for ${}^3\textup{D}_4$), and $\textup{F}_4(2,1)[q]$ (for ${}^2\textup{E}_6$). 
The first case $\textup{G}_2(3,1)[q]$ has $4$ residual points. The spectral diagram 
\cite{Opd4} of this Hecke algebra is the untwisted version of the Kac diagram 
\cite[Subsection 4.4]{Reed}, with the equal parameters $3k$ attached to the nodes
(a similar remark applies to all simple quasi-split unramified cases).
Let us use the maximal subdiagrams of this Kac diagram to name the various orbits of residual points.
There are two orbits of residual points 
$\textup{G}_2$ and $\textup{G}_2(a_1)$ with positive central character. The corresponding groups 
$A_\lambda$ (where $\lambda$ denotes the corresponding discrete unramified Langlands parameters 
via equation (\ref{eq:llprespt})) are $1$ and $S_3$ respectively. 
There are two nonreal orbits  
$\textup{A}_2$ (with $A_\lambda=1$) and $\textup{A}_1\times \textup{A}_1$
(with $A_\lambda=C_2$). Looking at the $q$-rational factor of the 
residue of the $\mu$ function at these points, we find that the 
 cuspidal (orbits of) residual points are 
$\textup{G}_2(a_1)$ (matching the degree of ${}^3\textup{D}_4[1]$) and 
$\textup{A}_1\times \textup{A}_1$ (matching the degree of ${}^3\textup{D}_4[-1]$).
Together with the Iwahori spherical unipotent discrete series these cases make up for 
the set of $7$ unipotent discrete series ($5$ of which are Iwahori spherical, while the others 
are cuspidal).

A similar discussion for $\textup{F}_4(2,1)[q]$ shows the following. 
Again we use the Kac diagram \cite[Subsection 4.5]{Reed}, this time with the 
constant parameters $2k$ attached to each node, to indicate the various orbits 
of residual points.
We have $9$ orbits of residual points (in the notation of 
\cite{OpdSol2}). There are $4$ orbits with positive real central character, 
corresponding to $\Psi(D_i(2k,k))=D_{i}(2k,2k)$ with $D_i(x,x)$ (for $i=1,\dots,4$) as listed 
in \cite[Table 3]{OpdSol2} (or equivalently, the $D_i(x,x)$ are the weighted Dynkin diagrams 
$\textup{F}_4, \textup{F}_4(\sigma_1), \textup{F}_4(\sigma_2), \textup{F}_4(\sigma_3)$ (in this order) of the distinguished 
nilpotent orbits of $\textup{F}_4$ as denoted by \cite{C})). The corresponding groups 
$A_\lambda$ are of the form $1,\,S_2,\,S_2,\,S_3$ respectively.
In addition there is $1$ orbit corresponding to 
$\textup{A}_1\times \textup{B}_3$ with $A_\lambda=C_2$, 
$1$ orbit for $\textup{A}_2\times \textup{A}_2$ with $A_\lambda=C_3$, 
$1$ orbit for $\textup{A}_3\times \textup{A}_1$ with $A_\lambda=C_2$, 
and finally 
$2$ orbits corresponding to $\textup{C}_4$ with $A_\lambda=1$ (the regular orbit) 
and $A_\lambda=C_2$ (the subregular orbit) respectively. 
The cuspidal orbits of residual points are in this case the ones corresponding to 
$\textup{F}_4(\sigma_3)$
(matching the degree of ${}^2\textup{E}_6[1]$) 
and the one of type $\textup{A}_2\times \textup{A}_2$ 
(matching the degree of ${}^2\textup{E}_6[\theta]$ and of ${}^2\textup{E}_6[\theta^2]$).
Hence we expect in total $18$ unipotent discrete series in this case (corresponding to the 
irreducible representations of the various $A_\lambda$). Using the classification of 
\cite[Theorem 8.7]{OpdSol} we can identify $13$ Iwahori spherical cases (corresponding to the 
discrete spectrum of $\textup{F}_4(2,1)[q]$), and there are $3$ cuspidal ones. 
(The two missing ones are of intermediate type, corresponding to a rank $1$ STM. See paragraph 
\ref{par:quasi}.) 
This agrees with the tables in \cite{Lu4} and \cite{Lu6}. 
\subsubsection{Existence of STMs for the exceptional cases.}\label{subsub:exc}
Let us now consider the existence of the positive rank STMs 
in the exceptional cases. Let $S^u$ be a $k$-spit torus. As always, we assume that 
$S^u\subset S$, with $S\subset G$ a fixed maximal $k$-split torus.    
Consider $M=C_G(S^u)^0$, $M^u=C_G(S^u)^0$, and assume that 
$M^u_{ssa}=M^u/S^u$ admits a cuspidal unipotent character $\sigma^u$. 
Recall that $\mc{H}^{IM}(M_{ssa})\approx \mc{H}^{IM}_Q$
for some proper subset $Q\subset F_m$. In particular, at most 
one of the irreducible components of $Q$ will not be of type $A$, and possible 
irreducible factors of type $A$ have to be in the anisotropic kernel of $G^u$. 
By the results for type $A$ and for rank $0$ STMs for irreducible exceptional types, 
there exists a unique rank $0$ STM
$\psi:\mb{L}\leadsto\mc{H}^{IM}(M_{ssa})$ for the cuspidal unipotent representation 
$\sigma^u$. Our task will be to see that this $STM$ map can be induced 
to $\mc{H}^{IM}$.

As in \ref{subsub:ind}, consider $L=r_MT^J$ where $r_M=s_Mc_M\in T_M$ 
is the image of $\phi$, $J\subset \Sigma_s(\mc{R},m)$ is the excellent 
subset of type $J_M$ associated to the STM diagram of 
$\phi$. Here $c_M$ is in dominant position with respect to $J$, 
so that the weight of a root $a_i^\vee$ in $J$ is given $Da^\vee_i(c_M)$, 
and $s_M$ is a vertex of $C^\vee$ in $F_m^{(1)}\backslash J$.
By what was said in the previous paragraph it follows that these diagrams 
are exactly the exceptional geometric diagrams of Lusztig, with weights 
attached to the vertices of the boxed set of vertices $J$. We remark that for 
all exceptional cases the geometric diagrams with $J$ such that $|\texttt{K}|>1$ 
(i.e. of positive rank), the components of $J$ are all of type $A$
(this simplicity is in remarkable contrast with the classical cases). Therefore the 
weights $w_j$ with $j\in J$ are simply equal to $q^{m^\vee_R(a_j^\vee)}$. 
If there would indeed exist a corresponding transfer map then its transfer map
diagram should be obtained by assigning in addition weights to the vertices in 
$\texttt{K}=F_m^{(1)}\backslash J$, as described in Definition  \ref{def:transmorfdiagr}.
These weights turn out to be uniquely determined by the basic property Proposition
\cite[Proposition 5.2]{Opd4} 
of spectral transfer maps (applied to the case of residual points), and this also 
enables us to find these weights $w_i$ easily (using the known classification of 
residual points of \cite{HO} and \cite{OpdSol}).
Our tasks is then to prove that these eligible diagrams thus obtained are indeed 
transfer map diagrams.

In order to do so we need to first find $\texttt{k}_0$. This has to be the unique vertex 
$\texttt{k}\in \texttt{K}$ of the geometric diagram such that the corresponding vertex 
$\omega_\texttt{k}\in C^\vee$ has the shortest length. 
It is easy to check in all exceptional geometric diagrams 
that this condition defines a unique vertex $\texttt{k}_0\in \texttt{K}$. 
The cuspidal unipotent representation $\sigma_u$ of $M_{ssa}$ lifts to a cuspidal unipotent 
representation $\tilde{\sigma}_u$ of $M$, and the cuspidal pair $(M,\tilde{\sigma}_u)$ is 
obtained by compact induction from a cuspidal unipotent type $\mf{s}:=(\mathbb{P}_J,\delta)$.  
The affine Hecke algebra $\mc{H}^{u,\mf{s},e}$ of the cuspidal unipotent type is given, 
and let $\mc{R}'_{m'}$ be the corresponding based root datum with multiptiplicity function $m'$. 

Next we need to determine the bijection between 
the affine simple roots $b^\vee_i$ of the spectral diagram of 
$\mc{H}^{u,\mf{s},e}$ and of $\texttt{K}$. This was done by Lusztig: 
According to the main result of \cite{Lu4}, there exists such a matching such that $\texttt{k}_0$ corresponds 
to $b_0$, and such that the underlying affine Coxeter diagram of the spectral diagram of   
$\mc{H}^{u,\mf{s},e}$ matches the Coxeter relations of the reflections in the quotient 
roots $\overline{\alpha}_\texttt{k}=Da^\vee_\texttt{k}|L$ (cf. \cite[2.11(c)]{Lu4}).  (Here we identify $L$ with $T^J\subset T$, 
the maximal subtorus on which the gradients of the roots from $J$ are 
constant, by choosing $r_M\in L$ as its origin (given by the weights of $j\in J$ 
and $s_0$ corresponding to $\omega_{\texttt{k}_0}$)). Since this matching is only based on the underlying 
affine Weyl groups, and by Proposition \cite[Proposition 5.6]{Opd4}, it is clear that a possible spectral map diagram 
has to provide the same matching. 
It is easy to check case by case that such a matching 
is unique up to diagram automorphisms preserving the parameters $m^\vee_{R'}$ of the spectral 
diagram of $\mc{H}^{u,\mf{s},e}$. Thus we fix such a matching, and use this to also 
parameterize the (dual) affine simple roots of the spectral diagram of $\mc{H}^{u,\mf{s},e}$
by $\texttt{k}\in \texttt{K}$ .

Following notations as in \ref{subsub:dia}, 
we need to assign an integer $c_\texttt{k}$ to each node $\texttt{k}\in \texttt{K}$, in order to define the 
weights $w_\texttt{k}$ for all $\texttt{k}\in \texttt{K}$. We define $c_\texttt{k}$ by the formula 
\begin{equation}\label{eq:ck}
c_\texttt{k}=m^\vee_R(a_\texttt{k}^\vee)-f_\texttt{k} m^\vee_{R'}(b_\texttt{k}^\vee)   
\end{equation}
where $a_\texttt{k}^\vee$ denotes the (dual) affine root of $\mc{R}^m$ associated with 
$\texttt{k}$, and $b_\texttt{k}^\vee$ the corresponding (dual) affine root of the spectral diagram of $\mc{H}^{u,\mf{s},e}$.

The diagram thus obtained defines a map $\phi$ from $T'$ to a suitable quotient of $L\subset T$. 
For each maximal proper 
subdiagram $D$ of a spectral diagram of a semi-standard affine Hecke algebra $\mc{H}(\mc{R},m)$ 
there exists a  
generic residual point $r^D_R$ such that $Da_\texttt{k}^\vee(r^D_R)=v^{2m_R^\vee(a_\texttt{k})}$ for all $\texttt{k}\in D$, 
and for $\texttt{k}\in \texttt{K}\backslash D$, such that $b_\texttt{k}^\vee(r^D_R|{v=1})$ is a primitive root of $1$ of order 
$n_\texttt{k}$.  The above assignment means that we require the alleged spectral transfer map 
$\phi$ to satisfy the property that $\phi(r^D_{R'})=r^{D\cup J}_{R}$. We can check easily case by 
case that this map then also sends all other residual points of  $\mc{H}^{u,\mf{s},e}$
to residual points of $\mc{H}^{IM}$, and that these weights are the only possible weights 
defining a map with such properties.
\begin{rem}
Thus, the image under $\phi_Z$ of the central character of the one dimensional 
discrete series representation of $\mc{H}^{u,\mf{s},e}$ which is the deformation of the sign character 
of its underlying affine Weyl group is equal to the central character of $\mc{H}^{IM}(G)$ 
of the analogous one dimensional character. If  $\mc{H}^{u,\mf{s},e}$ is the Iwahori Hecke algebra 
of an inner twist $G^u$ of $G$ then this is true in general, since we know that the formal degree of the 
Steinberg character is unchanged by inner twists. For exceptional groups it is true for all unipotent  
STMs of positive rank, which seems related to the fact that for these STMs the subset 
$J\subset F_m^{(1)}$ consists of type $\textup{A}$ components only. In classical cases  
STMs do not have this property in general.
\end{rem}
Finally we need to check that the map $\phi$ we have thus defined indeed 
defines an STM. This amounts to applying $\phi^*$ to $(\mu^{IM})^{(L)}$, 
making the substitutions $\phi^*(\alpha_i)=w_i$ for all $i\in I\backslash \{0\}$, and 
checking that this equals the $\mu$-function $\mu^{(u,\mf{s},e)}$ of $(\mc{H}^{u,\mf{s},e},d^{\tau,\mf{s},e})$
up to a rational constant. Now this is already clear  constant factor $d^{\tau,\mf{s},e}$ because of our 
choice of the weights of the $j\in J$ and the fact that we started out from a cuspidal STM 
for $\sigma^u$ for $M^u_{ssa}$. Hence we only need to consider, for all $\texttt{k}\in \texttt{K}$, 
the cancellations in 
$\phi^*((\mu^{IM})^{(L)})$ for the factors in numerator and the denominator  
which are of the form 
$(1-\zeta v^A(Db_\texttt{k}^\vee)^F)$ (with $A,F$ rational, $F$ nonzero, and $\zeta$ a root of unity). 
This is a tedious but simple task: 
We need to compile the table of all positive roots $\alpha\in R_{m,+}$, consider 
those $\alpha$ such that $\overline{\alpha}=\alpha|_L$ is a nonzero multiple of $\overline{\alpha}_\texttt{k}$
(upon ignoring the coefficients of $\alpha$ at the $j\in J$, and 
using the relation 
$(\sum_{\texttt{k}\in \texttt{K}}n_\texttt{k}n_{\texttt{k}_0}^{-1}Da^\vee_\texttt{k})|_L=\zeta_{\texttt{k}_0}v^l$ 
(with $\zeta_{\texttt{k}_0}$ a 
primitive root of $1$ of order $n_{\texttt{k}_0}$ and $l\in\mathbb{Z}$) 
which follows from $\sum_{\texttt{k}\in \texttt{K}}\tilde{n}_\texttt{k}Db^\vee_\texttt{k}=1$ 
and the discussion in paragraph \ref{subsub:dia}). Then we compute for each of 
those roots the value $\phi^*(\alpha)$. This produces a list of integral multiples of 
$f_\texttt{k} Db^\vee_\texttt{k}$, and for each member of that list, a list of values of the form $\zeta_j v^i$
with $\zeta_j$ a root of $1$ (of order divisible by $z_{\texttt{k}_0}$), and $v^i$ an integral power 
of  $v$. From these lists we can easily see the cancellations of these type of factors in 
$\phi^*(\mu^{IM})^{(L)})$, and check that a rational function of the form 
\begin{equation}\label{eq:res}
\frac
{(1-\beta^2)^2}
{(1+v^{-2m_{-}(\beta)}\beta)
 (1+v^{2m_{-}(\beta)}\beta)
 (1-v^{-2m_{+}(\beta)}\beta)
 (1-v^{2m_{+}(\beta)}\beta)}
 \end{equation}
(with $\beta=Db^\vee_\texttt{k}$) remains, as desired.
In this way we verify that all the diagrams so obtained are 
spectral map diagrams of spectral transfer maps, in all cases.

As a (rather complicated) example, let us look at $\tilde{E}_8/\textup{A}_3\textup{A}_3\textup{A}_1$.
This diagram arises by induction from the cuspidal pair $(\textup{E}_7,\sigma^u)$, 
whose spectral map diagram is given by $\tilde{\textup{E}_7}/\textup{A}_3\textup{A}_3\textup{A}_1$
(see the geometric diagram of \cite[7.14]{Lu4}).
The spectral diagram of $\mc{H}^{u,\mf{s},e}$ is of type $\textup{C}_1(7/2,4)[q]$.
The vertex $\texttt{k}_0$ is labelled by $1$ in \cite[7.8]{Lu4}. We write  
the simple roots of $\textup{C}_1(7/2,4)[q]$ in the form $b_1^\vee=1-2\beta$, and 
$b_2^\vee=2\beta$. The weights of the roots $a_1:=\alpha_6$ and 
$a_2^\vee:=\alpha_3$ are $w_1=\sqrt{-1}v^{-6}(-\beta/2)$
and $w_2=v^{-7}\beta/2$. When $\alpha$ runs over the positive roots of $\textup{E}_8$ 
such that $\phi^*(\alpha)$ is a nonzero multiple of $\beta/2$, the following lists of 
factors in front of $\beta/2$ appear: For $\zeta:=\pm\sqrt{-1}$, the following powers of $v$: 
$v^{\pm 6}$, $2$ times $v^{\pm 4}$, 
$3$ times $v^{\pm 2}$, and $4$ times $1$, and 
for $\zeta:=\pm 1$, the following powers  of $v$:  
$v^{\pm 5}$, $2$ times $v^{\pm 3}$, $3$ times $v^{\pm 1}$.
In addition the restricted root $\beta$ appears, with 
factor $1$. One easily checks that this produces the 
$\mu$ function of $\textup{C}_1(7/2,4)[q]$ indeed. The group $K_L^n$ is isomorphic 
to $C_2$ (caused by taking the square root of $\beta$). Note that this is 
equal to the central subgroup $T_L^{W_L}$ with $T_L\subset T$ the subtorus 
whose cocharacter lattice is coroot lattice of $\textup{E}_7\subset \textup{E}_8$.
 
As an example of a somewhat different kind, let us look at the unramified nontrivial inner form 
${}^2\tilde{\textup{E}}_7$ 
of type $\textup{E}_7$ (cf. \cite[7.18]{Lu4}). This case is induced from the trival representation 
$\sigma^u$ of the anisotropic kernel $M^u_{ssa}$ of the group of type ${}^2\tilde{\textup{E}}_7$, which is 
an anisotropic reductive group of rank $3$. The spectral diagram of $\mc{H}^{u,\emptyset,e}$ is 
of type $\textup{F}_4(1,2)[q]$. 
Since $F_u$ has order $2$, it follows that 
(see \ref{subsub:aniso}, \ref{par:ber}, and (\ref{eq:NormMu})) 
the $q$-rational factor of the formal degree of $\sigma^u$ is $[2]_q^{-3}$.  
This corresponds to the residue degree of the $\mu$-function of a 
Hecke algebra of type $\textup{A}_1[q]^3$ at its unique residue point. Hence we need to take 
$J$ of type $\textup{A}_1\textup{A}_1\textup{A}_1$. Such subdiagram fits in a unique way as an excellent subset in 
the spectral diagram of type $\tilde{E}_7$, up to the diagram autmorphism of $\tilde{E}_7$. 
However, we need to choose the unique such embedding of $J$ such that the 
root $a^\vee_0$  does not belong to $J$ (i.e. $J\subset F_{m,0}$ here; it is easy to check that the other 
possibility does not lead to a \emph{strict} STM (although it does lead to an essentially 
strict but non-strict STM, obtained by composing the strict STM we are about to construct by 
the nontrivial diagram automorphism of the spectral diagram of type $\textup{E}_7$, cf \cite[Remark 6.2]{Opd4})). 
Since we know that a transfer map diagram which is induced from this cuspidal pair must 
have the property that $J$ appears as an excellent subset of the diagram of, it is clear that 
Lusztigs geometric diagram for \cite[7.18]{Lu4} indeed should be the underlying geometric 
diagram of a spectral transfer map (if it exists), and $\texttt{k}_0$ is the vertex numbered by $5$ in 
\cite[7.18]{Lu4}. The $f_\texttt{k}$ are all equal to $1$, and (in the numbering of loc. cit.) 
we have $w_i=q^{\lambda_i}Db_i^\vee$ with $\lambda_i=-1$ for $i=1,2$ and $0$ for $i=3,4,5$.
It is easy to check that this gives a spectral transfer map $\Phi$.  
All other examples are done similarly by executing this algorithm.
We remark that $z_{\texttt{k}}\leq 3$ in all cases, except possibly when $Db_\texttt{k}^\vee$ 
is a divisible root of $R'_m$, when $z_\texttt{k}=4$ may occur (as in the above example). 
We leave it to the reader to 
check the remaining exceptional cases by him/herself.
\subsubsection{The exceptional non-split quasi-split cases}\label{par:quasi}
For convenience we explicitly list the unipotent STMs for the non-split quasi-split cases ${}^3\textup{D}_4$ and 
${}^2\textup{E}_6$. Both these groups do not have nontrivial inner forms. The rank $0$ STMs were all described 
in paragraph 3.2.3. For the case ${}^3\textup{D}_4$, up to $G^F$-conjugacy the only $F$-stable cuspidal unipotent 
pairs $(\mathbb{P},\sigma)$ 
are those with $\mathbb{P}$ an $F$-stable Iwahori subgroup and $\sigma=1$, or with 
$\mathbb{P}$ maximal hyperspecial. Thus, the only nontrivial unipotent STMs are the rank $0$ ones which 
were already described in paragraph \ref{subsub:rk0exc}

For ${}^2E_6$, besides the rank $0$ cases already described in paragraph 3.2.3 we have the rank $1$  
STM which arises from the cuspidal unipotent pair $(\mathbb{P},\sigma)$ where $\overline{\mathbb{P}}$ is 
of type ${}^2A_5$ (and $\sigma$ its unique cuspidal unipotent representation). This gives rise 
to a unipotent affine Hecke algebra of type $\textup{C}_1(4,5)[q]$. The unique STM 
$\Phi:\textup{C}_1(4,5)[q]\leadsto \textup{F}_4(2,1)[q]$ maps the two central characters of the two 
discrete series of $\textup{C}_1(4,5)[q]$ in a unique way to two orbits of residual points of $\textup{F}_4(2,1)[q]$.
Namely, $q^5$ maps to $\textup{A}_1\times \textup{B}_3$, while $-q^4$ maps to $\textup{A}_3\times \textup{A}_1$.
More precisely, $\Phi$ can be represented by a morphism $\phi:T_1\to L_n$ of torsors of 
the algebraic tori. Here we consider the algebraic tori $T_i$ associated to the two relevant affine Hecke algebras 
(with $T_i$ of rank $i$, and with coordinates given by the 
simple roots $\beta_1$ for $T_1$ and $\alpha_1, \dots,\alpha_4$ for $T_4$, with $\alpha_3,\alpha_4$ the short 
simple roots).
Further $L\subset T_4$ is a rank $1$ residual coset given by the equations 
$(\alpha_1+2\alpha_3)=-q^{-5},\ \alpha_2=q^2,\ \alpha_4=q$, while $L_n$ is a quotient $L\to L_n$ of $L$, a double 
cover. The morphism $\phi$ can be chosen as follows.
Let $F_{m'}=\{\beta_0,\beta_1\}$.  We check that $\texttt{k}_0=1$ and $\texttt{k}_1=3$, and that the 
additional weights 
of the transfer diagram map of $\phi$ are given by $w_1=-v^{-4}\beta_0^{1/2}$ and $w_3=v^{-3}\beta_1^{1/4}$.
Together with the information in paragraph \ref{subsub:rk0exc} this completes the descriptions of the 
relevant STMs $\mc{H}^{u,\mf{s}}\leadsto\mc{H}^{IM}$ for the cases ${}^3\textup{D}_4$ and ${}^2\textup{E}_6$.
\begin{rem}\label{rm:uniqueLLP}
In these two cases ${}^3\textup{D}_4$ and ${}^2\textup{E}_6$ we see that a parameterization
of the unipotent discrete series representations is completely determined by the matching condition  
that the $q$-rational factor of the formal degree needs to equal the residue $\mu^{IM,(\{r\})}(r)$, together with 
the requirement that we assign the generic representation to the trivial representation of $A_\lambda$
(where $\lambda$ is the unramified Langlands parameter which corresponds to $W_0r$ according to 
(\ref{eq:llprespt})). (To be precise, in the case ${}^2\textup{E}_6$ this fixes the parameterization except for 
the  interchangeability of ${}^2\textup{E}_6[\theta]$ and ${}^2\textup{E}_6[\theta^2]$.)
\end{rem}
\subsubsection{Unipotent affine Hecke algebras of type $\textup{C}_d(m_-,m_+)$}\label{subsub:Cn}
For an absolutely simple, quasi-split classical group $G$ of adjoint type other than $\textup{PGL}_{n+1}$, 
the proof of the essential uniqueness 
of an STM $\phi:\mc{H}^{u,\mf{s},e}\leadsto \mc{H}^{IM}(G)$ for an  
affine Hecke $\mc{H}^{u,\mf{s},e}$ of any unipotent type $\mf{s}$ for 
any inner form $G^u$ follows the same pattern 
as in the exceptional case, by reducing the statement to the essential uniqueness 
for cuspidal STMs. The proof of the existence of an STM $\phi$ as above is treated 
quite differently however, for most cases by generating $\phi$ as a composition 
of a small number of basic STMs which generate the spectral transfer category 
whose objects constist of \emph{all} unipotent affine Hecke algebras of the form 
$\mc{H}^{u,\mf{s},e}$ for all groups in certain classical families (containing $G$).
It turns out that in essence there are only $2$ types of basic building blocks
generating almost all STMs between the unipotent affine Hecke algebras associated 
to the unitary, orthogonal and symplectic groups. Apart from the STMs built from these 
basic generators there is one additional, very important type of basic STMs
of the form $\phi:\mc{H}^{u,\mf{s},e}\leadsto \mc{H}^{IM}(G)$
for the orthogonal and symplectic cases which we call \emph{extraspecial}.

As mentioned above, we will now first define some basic building blocks of STMs 
between classical affine Hecke algebras which are associated to the unitary, orthogonal and 
symplectic groups. 
We define a category $\mf{C}_{class}$ whose objects are normalized affine Hecke algebras
of type $(\textup{C}_d(m_-,m_+)[q^\mathfrak{b}],\tau_{m_-,m_+})$ where $d\in\mathbb{Z}_{\geq 0}$,  
$(m_-,m_+)\in V$, the set of ordered pairs $(m_-,m_+)$ of elements 
$m_\pm\in\mathbb{Z}/4$ satisfying $m_+-m_-\in\mathbb{Z}/2$, and $\mathfrak{b}=1$ if both 
$m_+-m_-\in\mathbb{Z}$ and $m_++m_-\in\mathbb{Z}$, otherwise we put $\mathfrak{b}=2$. 
Hence the objects of $\mc{C}_{class}$ 
are in bijection with the set $V$ of triples $(d;(m_-,m_+))$ as described above.

The trace $\tau=\tau_{m_-,m_+}$ is normalized as follows. First we decompose $V$ in six 
disjoint subsets $V^X$ with $X\in\{\textup{I},\,\textup{II},\,\textup{III},\,\textup{IV},\,\textup{V},\,\textup{VI}\}$,
which are defined as follows. If $m_\pm\in \mb{Z}\pm \frac{1}{4}$ write $|m_\pm|=\kappa_\pm+\frac{(2\epsilon_\pm-1)}{4}$ with 
$\epsilon_\pm\in\{0,1\}$ 
and $\kappa_\pm\in \mathbb{Z}_{\geq 0}$. Define $\delta_\pm\in\{0,1\}$ by $\kappa_\pm\in \delta_\pm+2\mathbb{Z}$.
Then we define:
\begin{equation}\label{eq:LX}
\begin{array}{lll}
&(d;(m_-,m_+))\in V^\textup{I}&\mathrm{\ iff\ }m_\pm\in\mathbb{Z}/2\mathrm{\ and\ }m_--m_+\not\in\mathbb{Z},\\
&(d;(m_-,m_+))\in V^\textup{II}&\mathrm{\ iff\ }m_\pm\in\mathbb{Z}+\frac{1}{2}\mathrm{\ and\ }m_--m_+\in\mathbb{Z},\\
&(d;(m_-,m_+))\in V^\textup{III}&\mathrm{\ iff\ }m_\pm\in\mathbb{Z}\mathrm{\ and\ }m_--m_+\not\in 2\mathbb{Z},\\
&(d;(m_-,m_+))\in V^\textup{IV}&\mathrm{\ iff\ }m_\pm\in\mathbb{Z}\mathrm{\ and\ }m_--m_+\in 2\mathbb{Z},\\
&(d;(m_-,m_+))\in V^\textup{V}&\mathrm{\ iff\ }m_\pm\in\mathbb{Z}\pm \frac{1}{4}\mathrm{\ and\ }\delta_--\delta_+\not=0,\\
&(d;(m_-,m_+))\in V^\textup{VI}&\mathrm{\ iff\ }m_\pm\in\mathbb{Z}\pm \frac{1}{4}\mathrm{\ and\ }\delta_--\delta_+=0.\\
\end{array}
\end{equation}
Observe that the type $X$ of $(d;(m_-,m_+))$ only depends on $(m_-,m_+)$; we will 
often simply write $(m_-,m_+) \in V^\textup{X}$ instead of $(d;(m_-,m_+))\in V^\textup{X}$.
We now normalize the traces $\tau_{m_-,m_+}$ as follows. These traces are of the form  
$\tau_{m_-,m_+}=(v^\mathfrak{b}-v^{-\mathfrak{b}})^{-d}\tau_{m_-,m_+}^0$, where $\tau_{m_-,m_+}^0$ is 
independent of the rank $d$ (and $d$ is suppressed in the notation).
Explicitly we define $\tau_{m_-,m_+}$ by: 
\begin{equation}\label{eq:norms}
d^\tau_{m_-,m_+}=(v^\mathfrak{b}-v^{-\mathfrak{b}})^d\tau_{m_-,m_+}(1):=
\left\{
\begin{array}{lll}
&d^{\tau,\{{}^2A\}}_a(q)d^{\tau,\{{}^2A\}}_b(q)&\mathrm{\ if\ }(m_-,m_+)\in V^\textup{I}\\
&d^{\tau,\textup{D}}_a(q)d^{\tau,\textup{B}}_b(q)&\mathrm{\ if\ }(m_-,m_+)\in V^\textup{II}\\
&d^{\tau,\textup{B}}_a(q)d^{\tau,\textup{B}}_b(q)&\mathrm{\ if\ }(m_-,m_+)\in V^\textup{III}\\
&d^{\tau,\textup{D}}_a(q)d^{\tau,\textup{D}}_b(q)&\mathrm{\ if\ }(m_-,m_+)\in V^\textup{IV}\\
&d^{\tau,\{{}^2A\}}_a(q)d^{\tau,\textup{B}}_b(q^2)&\mathrm{\ if\ }(m_-,m_+)\in V^\textup{V}\\
&d^{\tau,\{{}^2A\}}_a(q)d^{\tau,\textup{D}}_b(q^2)&\mathrm{\ if\ }(m_-,m_+)\in V^\textup{VI}\\
\end{array}
\right.
\end{equation}
Here $d^{\tau,\{{}^2A\}}_s(q)$ is the $q$-rational part of the formal degree of the cuspidal unipotent character for the 
adjoint group $G$ of type ${}^2\textup{A}_l$ induced from ${}^2\textup{A}_l(q^2)$,  
with $l=\frac{1}{2}(s^2+s)-1$ (with $s\in\mathbb{Z}_{\geq 1}$) (see Proposition \ref{prop:balanced}; 
it is convenient to extend this to $s=0$ by setting $d^{\tau,\{{}^2A\}}_0=1$); 
similarly  $d^{\tau,\textup{B}}_s(q)$ denotes the $q$-rational part of the formal degree of the cuspidal unipotent representation of $G$ 
of type $\textup{B}_l$ induced from $\textup{B}_l(q)$ with $l=s^2+s$ (with $s\in\mathbb{Z}_{\geq 0}$) 
(this degree covers the cuspidal character of the odd orthogonal and the symplectic groups); 
$d^{\tau,\textup{D}}_s(q)$ denotes the $q$-rational part of the formal degree of the cuspidal unipotent representation of $G$ 
of type $\textup{D}_l$ induced from $\textup{D}_l(q)$ with $l=s^2$ (with $s\in\mathbb{Z}_{\geq 0}$)
(this degree covers the cuspidal character of the even split orthogonal groups ($s$ even) 
and of the even quasi-split orthogonal groups ($s$ odd)). 
(Using \cite[Section 13.7]{C} and (\ref{eq:deg}) it is easy to give explicit formulas for these formal degrees.)
where the set $\{a,b\}$ with $a,b\in\mathbb{Z}_{\geq 0}$ is determined by the following
equalities of sets:
\begin{equation}\label{eq:normsympl}
\begin{array}{lll}
&\{\frac{1}{2}+a,\frac{1}{2}+b\}=\{|m_+-m_-|,|m_++m_-|\}&\mathrm{\ if\ }(m_-,m_+)\in V^\textup{I}\\
&\{2a,1+2b\}=\{|m_+-m_-|,|m_++m_-|\}&\mathrm{\ if\ }(m_-,m_+)\in V^\textup{II}\\
&\{1+2a,1+2b\}=\{|m_+-m_-|,|m_++m_-|\}&\mathrm{\ if\ }(m_-,m_+)\in V^\textup{III}\\
&\{2a,2b\}=\{|m_+-m_-|,|m_++m_-|\}&\mathrm{\ if\ }(m_-,m_+)\in V^\textup{IV}\\
&\{\frac{1}{2}+a,1+2b\}=\{|m_+-m_-|,|m_++m_-|\}&\mathrm{\ if\ }(m_-,m_+)\in V^\textup{V}\\
&\{\frac{1}{2}+a,2b\}=\{|m_+-m_-|,|m_++m_-|\}&\mathrm{\ if\ }(m_-,m_+)\in V^\textup{VI}\\
\end{array}
\end{equation}
This determines $a$ and $b$ in case $\textup{II}$,  $\textup{V}$,  $\textup{VI}$, 
and it determines $a$ and $b$
up to order in the other cases, so that the normalization $(\ref{eq:norms})$
is always well defined.

Now we define the building blocks of the STMs between these affine Hecke algebras.
First of all, the group $D_8$ of essentially strict spectral isomorphisms as described in 
Remark \cite[Remark 7.7]{Opd4} acts on the collection of objects of $\mf{C}_{class}$.
This corresponds to the action of $D_8$ on the set $V$ by preserving $d$, and 
on a pair $(m_-,m_+)$ the action is generated by the interchanging $m_-$ and $m_+$
and by sign changes of the $m_\pm$. Observe that these operations preserve the type $X$.
Then there exist additional basic STMs in $\mf{C}_{class}$ of the types indicated below. 
(In these formulas we have used the notation 
$\epsilon(x)=x/|x|\in\{\pm 1\}$ to denote the signature of a nonzero rational number $x$.) 
In the first $5$ cases one of the parameters $m_-$ or $m_+$ is translated by a step 
of size $1$ (if the translated parameter is half integral) or $2$ (if the translated parameter is 
integral) in a direction such that its absolute value decreases. In these first $5$ cases both
parameters can be translated in this way, as long as the absolute value of this 
parameter is larger than $\frac{1}{2}$ (in the half integral case) or $1$ (in the integral case).
A formula corresponds to an STM provided that this condition on the absolute value 
of the parameter which will be translated is satisfied. 
\begin{equation*}
\begin{array}{lll}
&\textup{C}_d(m_-,m_+)[q^2]\leadsto \textup{C}_{d+|m_-|-\frac{1}{2}}(m_--\epsilon(m_-),m_+)[q^2]&\mathrm{\ if\ }(m_-,m_+)
\in V^\textup{I}, m_+\not\in\mathbb{Z}\\
&\textup{C}_d(m_-,m_+)[q^2]\leadsto \textup{C}_{d+2(|m_+|-1)}(m_-,m_+-2\epsilon(m_+))[q^2]&\mathrm{\ if\ }(m_-,m_+)
\in V^\textup{I}, m_+\in\mathbb{Z}\\
&\textup{C}_d(m_-,m_+)[q]\leadsto \textup{C}_{d+|m_-|-\frac{1}{2}}(m_--\epsilon(m_-),m_+)[q]&\mathrm{\ if\ }(m_-,m_+)\in V^\textup{II}\\
&\textup{C}_d(m_-,m_+)[q]\leadsto \textup{C}_{d+2(|m_+|-1)}(m_-,m_+-2\epsilon(m_+))[q]&\mathrm{\ if\ }(m_-,m_+)\in V^\textup{III}\\
&\textup{C}_d(m_-,m_+)[q]\leadsto \textup{C}_{d+2(|m_+|-1)}(m_-,m_+-2\epsilon(m_+))[q]&\mathrm{\ if\ }(m_-,m_+)\in V^\textup{IV}\\
&\textup{C}_d(m_-,m_+)[q^2]\leadsto \textup{C}_{2d+\frac{1}{2}a(a+1)+2b(b+1)}(\delta_-,\delta_+)[q]&\mathrm{\ if\ }(m_-,m_+)\in V^\textup{V}\\
&\textup{C}_d(m_-,m_+)[q^2]\leadsto \textup{C}_{2d+\frac{1}{2}a(a+1)+2b^2-\delta_+}(\delta_-,\delta_+)[q]&\mathrm{\ if\ }(m_-,m_+)\in V^\textup{VI}\\
\end{array}
\end{equation*}
We denote the first $5$ cases of these STMs by $\Phi^{(m_-,m_+)}_{(d,-)}$ or  
$\Phi^{(m_-,m_+)}_{(d,+)}$, where the sign $\pm$ in the subscript indicates which of the parameters 
$m_-$ or $m_+$ will be translated. Notice that if we combine the basic STMs of the first $5$ 
cases with the group 
$D_8$ of spectral isomorphisms of $\mf{C}_{class}$ then we are allowed for all objects 
$X\in \{\textup{I},\,\textup{II},\,\textup{III},\,\textup{IV}\}$ either one of $m_-$ and $m_+$
(by a step of size $1$ or $2$ depending on 
the residue modulo $\mathbb{Z}$ 
of the parameter to be translated) as long as the absolute value of this parameter can still be reduced 
by such steps. Observe that these steps preserve the type $X$.
Finally we are of course allowed to compose these basic STMs thus obtained with each other and 
with the group $D_8$ of spectral isomorphisms. The basic translation steps as above commute with each 
other and have the obvious commutation relations with the group $D_8$ of spectral isomorphisms 
(this also follows easily from the essential uniqueness of STMs discussed below, see 
Proposition \ref{prop:clasuniq}). 
Observe that while the 
parameters are strictly decreasing with these basic translation steps, the rank is strictly increasing.  

Among the objects of $\mf{C}_{class}^X$ of the types 
$X\in  \{\textup{I},\,\textup{II},\,\textup{III},\,\textup{IV}\}$, 
the minimal spectral isogeny classes of objects 
(in the sense of \cite[Definition 8.1]{Opd4})  
are of the form: 
\begin{equation*}
\begin{array}{lll}
&[\textup{C}_l(0,\frac{1}{2})[q^2]]\mathrm{\ and\ } [\textup{C}_l(1,\frac{1}{2})[q^2]]&\mathrm{\ if\ } X=\textup{I}, \\
&[\textup{C}_l(\frac{1}{2},\frac{1}{2})[q]]&\mathrm{\ if\ } X=\textup{II},\\
&[\textup{C}_l(0,1)[q]]&\mathrm{\ if\ } X=\textup{III},\\
&[\textup{C}_l(0,0)[q]]\mathrm{\ and\ } [\textup{C}_l(1,1)[q]]&\mathrm{\ if\ } X=\textup{IV}, \\
\end{array}
\end{equation*}
Note that for all objects in $\mf{C}_{class}^X$, the spectral isogeny class of an object 
is just its isomorphism class \cite[Proposition 8.3]{Opd4}.
By abuse of language we will sometimes call the objects in a minimal (least) spectral isogeny
class in this sense also ``minimal" (respectively ``least").
Note that some of these minimal objects admit a group of order $2$ of spectral 
automorphisms (the cases $X=\textup{II}$ or $\textup{IV}$). 

The cases $X\in \{\textup{V},\,\textup{VI}\}$ are of a different 
nature. There are no STMs between the different objects of these cases, as we will see 
below. But from each object of $\mf{C}_{class}^V$ there is
an essentially unique (i.e. unique up to spectral automorphisms) STM to the least object in  
$\mf{C}_{class}^{\textup{III}}$ and from each object of $\mf{C}_{class}^V$ there is  
an essentially unique STM to one of the two types of minimal objects in 
$\mf{C}_{class}^{\textup{IV}}$.  We call these STMs \emph{extraspecial}.
It is easy to give a representing morphism $\phi =\phi^{(m_-,m_+)}_{(d,\pm)}$ defining the basic STMs of this kind.
The first $4$ cases, the building blocks of elementary translations in the parameters $m_-$ and $m_+$, 
do in general not correspond to geometric diagrams as given in \cite{Lu4} and \cite{Lu6}, since 
the image of the spectral transfer map is in general not a least object. However, as we will see below, 
these building blocks are quite simple and there existence can be established easily by a direct computation.
The extraspecial cases correspond to the geometric diagrams \cite[7.51, 7.52]{Lu4} 
and to \cite[11.5]{Lu6} (in a way that will be made precise below). 
  
The formula defining such morphim $\phi$ for the first $5$ cases 
(thus a minimal translation step in one of the parameters of an object of type 
$X\in  \{\textup{I},\,\textup{II},\,\textup{III},\,\textup{IV}\}$) only depends  on 
the value modulo $\mathbb{Z}$ of the parameter to be translated. Using the 
group $D_8$ of spectral isomorphisms it is enough 
to write down the formula for a basic translation in $m_+$ where $m+>0$.

Let $(d;(m_-,m_+))\in V$.
First assume that $m_+\in \mathbb{Z}_{\geq 0}+\frac{1}{2}$. 
For $d\geq 0$ consider the torus $T_d(\mb{L}):=\mathbb{G}_m^{r}(\mb{L})$ over $\mb{L}$.
We write its character lattice as $X^*(T_d):= X_d$ (or $X_d=\mathbb{Z}^d$).
The standard basis of $X^*(T_d)$ is denoted by $(t_1,\dots,t_d)$.
We consider $X_d$ as the
root lattice of the root system of type $\textup{B}_d$. The Weyl group $W_0$ acts
by signed permutations on $X_d$. 
For $m_\pm\in\mathbb{Z}+1/2$ we define a homomorphism 
$\phi^{(m_-,m_+)}_{(d,+),T}:T_d\to T_{d+m_+-1/2}$
of algebraic tori over $\mb{L}$ by 
\begin{equation*}
\phi^{(m_-,m_+)}_{(d,+),T}(t_1,\dots,t_d):=(t_1,t_2,\dots,t_d,v^{\mathfrak{b}},v^{3\mathfrak{b}},\dots,v^{2\mathfrak{b}(m_+-1)})
\end{equation*}
Next, if $m_+\in\mathbb{Z}_{>0}$ we define a morphism
$\phi^{(m_-,m_+)}_{(d,+),T}:T_d\to T_{d+2(m_+-1)}$
of algebraic tori over $\mb{L}$
by : 
\begin{equation*}
\phi^{(m_-,m_+)}_{(d,+),T}(t_1,\dots,t_d):=(t_1,t_2,\dots,t_d,1,q^{\mathfrak{b}},q^{\mathfrak{b}},q^{2\mathfrak{b}},q^{2\mathfrak{b}},\dots,q^{\mathfrak{b}(m_+-2)},q^{\mathfrak{b}(m_+-2)},q^{\mathfrak{b}(m_+-1)})
\end{equation*}
Finally, for the extraspecial cases $X\in  \{\textup{V},\,\textup{VI}\}$ we define, for $m_\pm>0$,  
a morphism $\phi^{(m_-,m_+)}_{(d,+),T}:T_d\to T_L$ with 
$L:=2d+\frac{1}{2}a(a+1)+2b(b+1)$ (if $X=\textup{V}$) or $L:=2d+\frac{1}{2}a(a+1)+2b^2-\delta_+$
(if $X=\textup{VI}$) as follows. Observe that $L=2d+\lfloor L_-\rfloor +\lfloor L_+\rfloor$ where  
$L_\pm:=\kappa_\pm(2\kappa_\pm+2\epsilon_\pm-1)/2$.  We first define, for $m\in \mathbb{Z}\pm\frac{1}{4}$,
 residual points $r_e(m)$ recursively by putting, for $m>1$,  
 \begin{equation*}
 r_e(m)=(\sigma_e(m);r_e(m-1))
 \end{equation*}
with, for $m>1$,  
 \begin{equation*}
 \sigma_e(m)=(q^\delta,q^{\delta+1},\dots,q^{2m-\frac{3}{2}}), 
 \end{equation*}
and $r_e(\frac{1}{4})=r_e(\frac{3}{4}):=\emptyset$.
We define the representing morphism of the extraspecial STM by: 
\begin{equation}\label{eq:extraspSTM}
\phi^{(m_-,m_+)}_{(d,+),T}(t_1,\dots,t_d):=(-r_e(m_-),v^{-1}t_1,vt_1,\dots,v^{-1}t_d,vt_d,r_e(m_+))
\end{equation}
The proof of the fact that these formulas indeed define an STM is a straightforward 
computation in the cases  $X\in  \{\textup{I},\,\textup{II},\,\textup{III},\,\textup{IV}\}$.
In the extraspecial case one notices first that this map for general $d\geq 0$ 
is induced from the cuspidal map of this kind with $d=0$. It is easy to verify that this 
map is an STM, by considering the set of positive roots of the root system $R_0$ of type 
$\textup{B}_{2d+L_-+L_+}$ which restrict to a given simple root $\alpha_i$ of $\textup{B}_d$ (this process 
is similar to what we did in the exceptional cases), \emph{provided} that the inducing 
rank $0$ map is indeed a cuspidal STM.  
The latter be proved by induction on $m_\pm$, using the 
recursive definition of $r_e(m)$ and the formula (easily obtained from  
(\ref{eq:norms}) applied to the two cases $\{\textup{V},\,\textup{VI}\}$):
\begin{equation}
d^\tau_{m_-,m_+}=
\prod_{i=1}^{\lfloor |m_--m_+|\rfloor}
\frac{v^{{2(|m_--m_+|-i)}i}}{(1+q^{2(|m_--m_+|-i)})^i}
\prod_{j=1}^{\lfloor |m_-+m_+|\rfloor}
\frac{v^{{2(|m_-+m_+|-j)}j}}{(1+q^{2(|m_-+m_+|-j)})^j}
\end{equation}
\subsubsection{Existence of enough STMs for the classical cases.}\label{subsub:class}
After having established the existence of these STMs between affine Hecke algebras 
of type $\textup{C}_n^{(1)}$, it is an easy task to prove the existence 
an STM  of the form $\phi:\mc{H}^{u,\mf{s},e}\leadsto\mc{H}^{IM}(G)$ for all 
absolutely simple, quasisplit adjoint groups of classical type $G$ and unipotent 
affine Hecke algebras $\mc{H}^{u,\mf{s},e}$ of a unipotent type of an inner form of $G$
(other than $\textup{PGL}_{n+1}$), 
using covering STMs. 

For $G=\textup{PU}_{2n}$, 
we have a $2:1$ semi-standard spectral covering map (see \cite[7.1.3]{Opd4}) of the form 
$\mc{H}^{IM}(G)=\textup{B}_n(2,1)[q]\leadsto \textup{C}_{n}(0,\frac{1}{2})[q^2]$ corresponding 
to an embedding of the right hand side as an index two subalgebra of the left hand side. 
Here the right hand side is normalized as object in $\mf{C}_{class}^{\textup{I}}$. 
The representing morphism $\phi_T$ has kernel $\omega\in T^{IM}$,  the unique 
nontrivial $W_0(\textup{B}_n)$-invariant element. We can identify $\omega$ with the nontrival element 
of $X^*_{un}(G)=(\Omega_C^\theta)^*$, which equals $C_2$ in this case. It acts as a diagram 
automorphism on the geometric 
diagram via the simple affine reflection 
$\sigma=s_{1-2x_1}$ of the afine Weyl group of type $\textup{C}_n^{(1)}$ 
(in the standard coordinates for $\mf{t}$).  
For $G=\textup{PU}_{2n+1}$ we have an isomorphism 
$\mc{H}^{IM}(G)\leadsto \textup{C}_{n}(\frac{1}{2},1)[q^2]$. 
These target affine Hecke algebras are the minimal objects of $\mf{C}_{class}^\textup{I}$.
All direct summands $\mc{H}$ of $\mc{H}_{uni}(G)$ either are objects of $\mf{C}_{class}^\textup{I}$ or,   
in the case $G={}^2\textup{A}_{2n-1}$, otherwise there exists a semi-standard $2:1$ covering 
STM $\mc{H}\leadsto \mc{H}'$ 
arising from an index two embedding 
$\mc{H}'\subset \mc{H}$ of an object $\mc{H}'$ of $\mf{C}_{class}^\textup{I}$ for which one of the 
parameters $m_-$ or $m_+$ equals $0$. In the latter case, it is easy to see from the definitions 
that any composition $\phi_T$ of basic translation STMs in $\mc{C}_{class}^\textup{I}$ which yields 
an STM $\mc{H}'\leadsto \textup{C}_{n}(0,\frac{1}{2})[q^2]$ in $\mf{C}_{class}^{\textup{I}}$
factors through an STM $\mc{H}\leadsto \mc{H}^{IM}(G)$. Let $L=rT^L$ be the image of $\phi_T$.
The inverse image of $L$ under the covering map is connected if $L$ has positive rank, this 
follows from the spectral map diagram and the fact that linear residual points in a positive 
Weyl chamber are invariant for diagram automorphisms \cite{Opd3}.
This implies that we have a unique factorization of $\phi_T$ as desired in all cases. 
Remark that  $\Omega_1^{\mf{s},\theta}=1$ except if $\mf{s}$ is a supercuspidal 
unipotent type of $G^u$ of type ${}^2A_{2n-1}$ or its non-quasisplit inner form, which is 
also $\Omega_C^\theta$-invariant. In this case $\Omega_1^{\mf{s},\theta}=C_2$, and 
the supercuspidal STM $\phi_T$ as above 
has image $L$ (a residual point) which lifts to two residual points which are not conjugate under $W(\textup{B}_n)$.
Hence in this case we obtain two STMs defined by the lifts of $\phi_T$, corresponding to the 
two summands of $\mc{H}^{\mf{s},u}$. If $\Omega_1^{\mf{s},\theta}=1$ but 
$\Omega^{\mf{s},\theta}=\Omega_C^\theta=C_2$
then $X^*_{un}(G)$ acts nontrivially on the connected inverse image of $L$. This is precisely the 
case where one parameter of $\mc{H}'$ is $0$, and the rank is positive. In other cases $\mc{H}$
is itself already an object of $\mf{C}_{class}^\textup{I}$ whose STM has a unique lift to $\mc{H}^{IM}(G)$.

For $G=\textup{SO}_{2n+1}$ we have $\mc{H}^{IM}(G)=\textup{C}_{n}(\frac{1}{2},\frac{1}{2})[q]$, 
and all unipotent affine Hecke algebras are objects of $\mf{C}_{class}^{\textup{II}}$; hence this case is 
straightforward by the above.

For $G=\textup{PCSp}_{2n}$ we have a semi-standard STM $\mc{H}^{IM}(G)\leadsto \textup{C}_{n}(0,1)[q]$
arising from an embedding
of the right hand side as an index two subalgebra of the left hand side.  
Here the right hand side is normalized as object in $\mf{C}_{class}^{\textup{III}}$. 
All direct summands of $\mc{H}_{uni}(G)$ either are objects of $\mf{C}_{class}^\textup{III}$ or 
$\mf{C}_{class}^\textup{V}$, or there 
exists a semi-standard covering STM $\mc{H}\leadsto\mc{H}'$ arising from an index two embedding 
$\mc{H}'\subset \mc{H}$ of an object 
$\mc{H}'$ of $\mf{C}_{class}^\textup{III}$ for which one of the parameters $m_-$ or $m_+$ equals $0$.
Similar remarks as in the case $\textup{PU}_{2n+1}$ apply on how to obtain STMs of 
direct summands $\mc{H}$ of $\mc{H}_{uni}(G)$ to $\mc{H}^{IM}(G)$ in terms of those of 
$\mc{H}'$ to $ \textup{C}_{n}(0,1)[q]$.

For $G=\textup{P}(\textup{CO}^0_{2n})$, we have a non-semistandard STM 
$\mc{H}^{IM}(G)\leadsto \textup{C}_{n}(0,0)[q]$ which is represented by a degree $2$ 
covering of tori (essentially the ``same" covering of tori as for the case $\textup{PU}_{2n+1}$, 
but this time equipped with the action of $W(\textup{D}_n)$ instead of $W(\textup{B}_n)$) (see \cite[7.1.4]{Opd4}). 
Here the right hand side is normalized as object in $\mf{C}_{class}^{\textup{IV}}$. 
All other direct summends $\mc{H}$ of $\mc{H}_{uni}(G)$ either are objects of $\mf{C}_{class}^\textup{IV}$ with 
both $m_-$ and $m_+$ even, or of $\mf{C}_{class}^\textup{VI}$ with $\delta_-=\delta_+=0$, or there 
exists a semi-standard covering STM $\mc{H}\leadsto\mc{H}'$ arising from an index two embedding 
$\mc{H}'\subset \mc{H}$ of an object 
$\mc{H}'$ of $\mf{C}_{class}^\textup{IV}$ for which one of the parameters $m_-$ or $m_+$ equals $0$.
If both of $m_\pm\not=0$ then $\Omega^{\mf{s},\theta}=1$, and any composition of basic STMs or 
the extraspecial STM $\phi:\mc{H}\leadsto \textup{C}_{n}(0,0)[q]$ admits a unique lift to an STM 
$\phi:\mc{H}\leadsto\mc{H}^{IM}(G)$ as before. If one of $m_\pm$ equals zero then 
$\Omega^{\mf{s},\theta}=C_2$. As before, in the positive rank case we have 
$\Omega^{\mf{s},\theta}_1=1$, and $X^*_{un}(G)$ acts non-trivially by spectral 
isomorphisms, via its quotient $(\Omega^{\mf{s},\theta})^*=C_2$, on the connected 
inverse image of the residual coset $L$ which is the image of $\phi$. Finally if 
one of $m_\pm=0$ and the rank of $\phi$ is $0$ then 
$\Omega^{\mf{s},\theta}=\Omega^{\mf{s},\theta}_1=C_2$, and $L$ has two lifts under the $2:1$ covering 
which are not in the same $W(\textup{D}_n)$-orbit but which are exchanged by the action of $X^*_{un}(G)$. 
In this case, $\mc{H}^{u,\mf{s}}$ decomposes as a direct sum of two copies of $\mb{L}$, and 
we have still an essentially unique STM $\mc{H}^{u,\mf{s}}\leadsto \mc{H}^{IM}(G)$.

For $G=\textup{P}((\textup{CO}^*_{2n+2})^0)$. We have 
$\mc{H}^{IM}(G)= \textup{C}_{n}(1,1)[q]$. 
This case is similar to 
the previous case, except that this time the relevant objects from $\mf{C}_{class}^\textup{IV}$
are those with $m_-$ and $m_+$ both odd, and those of $\mf{C}_{class}^\textup{VI}$
the objects with $\delta_-=\delta_+=1$. Hence this case is easier, since the direct summands
of $\mc{H}$ of $\mc{H}_{uni}(G)$ are themselves already objects of $\mf{C}_{class}^{\textup{IV}}$ 
and of $\mf{C}_{class}^{\textup{VI}}$, and no dicussion of lifting of STMs is required.
\subsubsection{Proof of Theorem \ref{thm:unique}}
Suppose $G$ is as in Theorem \ref{thm:unique}.
In the previous paragraphs 
we have established the existence of at least one STMs $\phi_{u,\mf{s},e}:\mc{H}^{u,\mf{s},e}(G)\to\mc{H}^{IM}(G)$
for every unipotent type $\mf{s}$ of any inner form $G^u$ of $G$. Such an STM $\phi_{u,\mf{s},e}$ determines a unique 
subset $Q\subset F_0$ such that 
$\phi_{u,\mf{s},e}$ is represented by a morphism $\phi_T$ whose image is of the form $L_n:=L/K_L^n$
with $L=r_QT^Q$ and $r_Q$ a residue point of the semisimple subquotient $\mc{H}_Q^{IM}$ of $\mc{H}^{IM}(G)$.
This $r_Q$ is determined up to the action of $K_Q$.
As was explained in \ref{subsub:indres}, $\phi$ is in this situation induced from  
from a cuspidal STM $\phi_Q:(\mb{L},\tau_0):=\mc{H}_0\leadsto \mc{H}_Q^{IM}=\mc{H}(M_{ssa})$.
\emph{Suppose that we know that the essential uniqueness for the cuspidal case of Theorem \ref{thm:unique} holds.} 
Then $W_Qr_Q$ is determined by $\mc{H}_0$ up to the action of $\textup{Aut}_{es}(\mc{H}_Q^{IM})$ 
(see the argument 
in \ref{subsub:indres}), 
and since we are clearly in the standard case, this is anti-isomorphic to $\Omega^*_{X_Q}\rtimes \Omega_0^{Y_Q}$
by Proposition \cite[Proposition 3.4]{Opd4}. But we know (see \cite{Opd3}, \cite[Theorem A.14(3)]{Opd1}) that 
$W_Qr_Q$ is fixed for the action of $\Omega_0^{Y_Q}$, so that we need to consider only the orbit of 
$W_Qr_Q$ for the action of $\Omega^*_{X_Q}:=(X_Q/\mathbb{Z}R_Q)^*$. In the  case at hand  
$X_Q:=X/X\cap R_Q^\perp=P(R_0)/P(R_0)\cap R_Q^\perp=P(R_Q)$, so that 
$(\Omega_{X_Q})^*=(P(R_Q)/\mathbb{Z}R_Q)^*$. But this is exactly equal to $K_Q$, 
hence any STM which is induced from a cuspidal STM 
$\phi_Q:(\mb{L},\tau_0):=\mc{H}_0\leadsto \mc{H}_Q^{IM}$ has as its image $L_n$.
By the rigidity property  Proposition \cite[Proposition 7.13]{Opd4} we see that any two such STMs are 
equal up to the action of $\textup{Aut}_\mf{C}(\mc{H})$. But as was explained in 
\ref{subsub:indres}, the subset $Q\subset F_0$ is itself completely determined by just the  
root system of $M_{ssa}$, and this is determined by $\mf{s}$. It follows that 
any other STM $\phi':\mc{H}^{u,\mf{s},e}(G)\to\mc{H}^{IM}(G)$ can be represented by a 
$\phi_T'$ whose image is $L_n'$, with $L'$ a residual coset in the 
$X^*_{un}(G)$-orbit of $L$. 

As to the possibility to define an equivariant 
STM for the action of $X^*_{un}(G)=(\Omega^\theta_C)^*$, that is an application of Theorem 
\ref{thm:summary}. Recall that $X^*_{un}(G)$ acts on $\mc{H}^{u,\mc{O}}$ via its quotient 
$(\Omega^{\mf{s},\theta})^*$; we need to check in all cases that the subgroup 
$(\Omega^{\mf{s},\theta}_2)^*\subset (\Omega^{\mf{s},\theta})^*$ is the stabilizer of $W_0(L)$.
For the classical cases this was discussed in the previous sections, and for the exceptional 
cases this is an easy verification. It follows that the direct sum $\mc{H}^{u,\mf{s}}$
of all summands of $\mc{H}_{uni}(G)$ in the $X^*_{un}(G)$-orbit of $\mc{H}^{u,\mf{s},e}$ 
can be mapped $X^*_{un}(G)$-equivariantly by an STM to $\mc{H}^{IM}(G)$, 
and that such an equivariant STM is essentially unique up the spectral automorphism 
group of $\mc{H}_{uni}(G)$. Taking the direct sum over all orbits $X^*_{un}(G)$-orbits 
of unipotent types we obtain the desired result.

Hence Theorem \ref{thm:unique} is now reduced to the cuspidal case. 
For the exceptional cases we have already shown the essential uniqueness for 
cuspidal STMs, and for $G=\textup{PGL}_{n+1}$
this was obvious. Hence the proof of Theorem \ref{thm:unique} is completed by 
the following result, whose proof will appear in \cite{FO}:
\begin{prop}[\cite{FO}]\label{hyp} The essential uniqueness of 
Theorem \ref{thm:unique} holds true for the cuspidal part 
(or rank $0$ part) $\mc{H}_{uni, cusp}(G)$ of $\mc{H}_{uni}(G)$
for $G$ of type $\textup{PU}_{n}$, $\textup{SO}_{2n+1}$, $\textup{PCSp}_{2n}$,
$\textup{P}(\textup{CO}^0_{2n})$, and $\textup{P}((\textup{CO}^*_{2n})^0)$.
Here we denote by $\mc{H}_{uni, cusp}(G)$ 
the direct sum of all the cuspidal (or rank $0$) 
normalized generic unipotent affine Hecke algebras associated 
to $G$ and its inner forms. 
In other words, there do  not exist  other rank $0$ STMs  
than the ones constructed above, and this yields a $X^*_{un}(G)$-equivariant STM 
$\Phi_{cusp}:(\mc{H}_{uni,cusp}(G),\tau)\leadsto(\mc{H}^{IM},\tau^{IM})$
which is \emph{essentially unique} in the sense of Theorem \ref{thm:unique}.
\end{prop}
The proof of this Proposition reduces to the analogous statement 
for the spectral categories $\mf{C}_{class}^X$. For $X=\textup{I},\,\textup{II}$ this is rather easy.
When $X=\textup{III},\,\textup{IV},\,\textup{V},\,\textup{VI}$ the \emph{essential uniqueness} proof for cuspidal 
STMs is based on the \emph{existence} of the cuspidal extraspecial STMs of $\mf{C}_{class}^\textup{V}$
and $\mf{C}_{class}^\textup{VI}$. It is easy to see that every generic residual point of  $\mf{C}_{class}^X$ for 
$X=\textup{III},\,\textup{IV}$ is in the image of a unique extraspecial cuspidal STM, and this sets up a bijection 
between the set of generic residual points of the combined objects of $\mf{C}_{class}^{\textup{V},\,\textup{VI}}$ 
and those of $\mf{C}_{class}^{\textup{III},\,\textup{IV}}$. If we impose the necessary condition for cuspidality, 
namely that the formal degree (in our normalization) has no odd cyclotomic factors, then one can show 
that the corresponding generic residual point of  $\mf{C}_{class}^{\textup{V},\,\textup{VI}}$ is given 
by a pair $(\xi_-,\xi_+)$ of partitions whose Young tableaux are of rectangular shape, and 
almost a square. After applying the extraspecial STM, the solutions correspond to a pair $(u_-,u_+)$ 
of unipotent orbits of $G_s\subset G$, a semisimple subgroup of maximal rank, whose elementary divisors 
are both of the form $(1,3,5,\dots)$ or $(2,4,6,\dots)$, or are both of the form $(1,5,9,\dots)$ or $(3,7,11,\dots)$.
These solutions thus correspond to the cuspidal local systems for the endoscopic groups 
$G_s\subset G$ (cf. \cite{Lu1}, \cite{Lu1.5}).
\section{Applications}
\subsection{The classification of unipotent spectral transfer morphisms}
\subsubsection{The classical case}
\begin{prop}\label{prop:clasuniq}
Between the objects of $\mf{C}_{class}^{\textup{I}}$,  $\mf{C}_{class}^{\textup{II}}$, 
$\mf{C}_{class}^{\textup{III}\cup\textup{V}}$ and $\mf{C}_{class}^{\textup{IV}\cup\textup{VI}}$
(where $\mf{C}_{class}^{\textup{III}\cup\textup{V}}$ is shorthand for 
$\mf{C}_{class}^{\textup{III}}\cup\mf{C}_{class}^{\textup{V}}$ etc.)
all STMs are generated by the basic translation STMs we have defined in \ref{subsub:class}, the 
extraspecial STMs, and the dihedral group $D_8$ (cf. \cite[Remark 7.5]{Opd4}) 
of spectral isomorphisms. The basic translation STMs commute 
with each other, and the commutation rules of the basic translation STMs and extra special STMs 
with the $D_8$ are the obvious ones, where $D_8$ acts on the set of parameter pairs $(m_-,m_+)$ 
(i.e. $D_8$ acts as a group of endofunctors on each of these categories).
\end{prop}
\begin{proof}
For any object $\mc{H}$ in $\mf{C}_{class}^Y$ ($Y$ as in the Theorem) 
there exists an STM $\phi:\mc{H}\leadsto \mc{H}^{min}$,  
where $\mc{H}^{min}$  denotes a minimal object, and where $\phi$ is a translation STM or an extraspecial STM.
By the essential uniqueness of Theorem \ref{thm:unique} it follows that any 
STM $\psi:\mc{H}\leadsto \mc{H}^{min}$ is of the form $\psi=\beta\circ\phi\circ\alpha$ with 
$\alpha\in\textup{Aut}_{es}(\mc{H})$ and with $\beta\in \textup{Aut}_{es}(\mc{H}^{min})$.
In $\mf{C}_{class}^Y$, the group $\textup{Aut}_{es}(\mc{H})$ is trivial (if the parameters $m_-$ and 
$m_+$ are unequal) or $C_2$ (if the parameters are equal). If there exists a nontrivial $\alpha_0
\in \textup{Aut}_{es}(\mc{H})$, then $m_-=m_+$ and it follows easily from the definitions that there also 
exists a nontrivial 
$\beta_0\in \textup{Aut}_{es}(\mc{H}^{min})$, and that $\phi$ is equivariant in the sense 
$\phi\circ\alpha_0=\beta_0\circ \phi$. Hence, if $\psi$ is also a composition of 
basic translation STMs or if $\psi$ is an extra special STM we see that $\psi=\phi$. 
From the injectivity (obvious from the definitions) of the 
basic generating STMs it now follows that the basic translation STMs commute. 

We also conclude from the injectivity of the basic generating STMs that, up to spectral 
isomorphisms, 
there can exist at most one STM between any two objects of $\mf{C}_{class}^Y$. 
For  $X\in\{\textup{I}, \textup{II}, \textup{III}, \textup{IV}\}$ it follows from a consideration 
of the spectral transfer map diagrams (Definition \ref{def:transmorfdiagr}) of the (essentially 
unique, injective) STMs $\phi_1:\mc{H}_1\leadsto\mc{H}^{min}$ and  $\phi_2:\mc{H}_2\leadsto\mc{H}^{min}$
that a possible factorizing STM $\phi:\mc{H}_1\leadsto\mc{H}_2$ (uniquely determined if it exists) must be 
itself composed of basic translation STMs and spectral isomorphism itself. It is also easy to 
see in this way that there can not exist STMs between objects of $\mf{C}_{class}^{\textup{V}\cup\textup{VI}}$
and non-minimal objects of $\mf{C}_{class}^X$ with $X\in\{\textup{I}, \textup{II}, \textup{III}, \textup{IV}\}$.
 
Between objects of $\mf{C}_{class}^X$ for $X\in \{\textup{V},\textup{VI}\}$ there are 
no STMs. This again follows from the injectivity of the extra special STMs, in view of 
the fact that the images of two extraspecial  
STMs of the form $\phi_1:\mc{H}_1\leadsto \mc{H}^{min}$ and $\phi_2:\mc{H}_2\leadsto \mc{H}^{min}$
map to disjoint subsets of the spectrum of the center of $\mc{H}^{min}$, unless $\mc{H}_1$ and $\mc{H}_2$ are 
isomorphic (this follows from the ``extra special bijection" proved in \cite{FO}).  

\end{proof}
As a consequence we obtain a general description of all spectral transfer maps 
between all unipotent affine Hecke algebras in the classical cases:
\begin{cor} Let 
$\mb{G}$ be connected, absolutely simple, defined and quasisplit over $k$, 
split over $K$, and such that its restricted root system is of classical type.
There are no other STMs between the unipotent affine Hecke algebras of the 
form $\mc{H}^{u,\mf{s},e}$ which appear as summands of $\mc{H}_{uni}(G)$
than the ones obtained by lifting of STMs via spectral covering maps of 
direct summands of $\mc{H}_{uni}(G)$ to one of 
$\mf{C}_{class}^{\textup{I}}$,  $\mf{C}_{class}^{\textup{II}}$, 
$\mf{C}_{class}^{\textup{III}\cup\textup{V}}$ and $\mf{C}_{class}^{\textup{III}\cup\textup{V}}$
(lifting in a sense similar to the discussion in \ref{subsub:class}).
\end{cor}
It is not difficult to describe all STMs between the unipotent affine Hecke algebras 
for exceptional types as well, but we will not do this here.
\subsection{The partitioning of unramified square integrable L-packets according to Bernstein components}
\label{subsub:partition}
Let $\mb{G}$ be connected, absolutely simple, defined and quasisplit over $k$, 
split over $K$, and of adjoint type. Let $\mc{H}'$ be a unipotent affine Hecke algebra 
associated to a unipotent type of an inner form of $G$ (hence, a summand of $\mc{H}_{uni}(G)$). 
By our Theorem \ref{thm:unique}
we know that there exists an essentially unique $X^*_{un}{G}$-equivariant  STM 
$\phi:\mc{H}_{uni}(G)\leadsto\mc{H}^{IM}(G)$, 
and we know that such map is compatible with the arithmetic/geometric correspondence of diagrams 
of Lusztig \cite{Lu4}, \cite{Lu6}. By \cite[Theorem 6.1]{Opd4}, this STM $\phi$ gives rise to a correspondence 
between components of the tempered irreducible spectra of $\mc{H}'$ and $\mc{H}^{IM}(G)$ 
which preserves, up to rational constant factors, the Plancherel densities on these components, and 
which is 
compatible with the map $\phi_Z$ on the level of central characters of representations. In particular for 
unipotent discrete series representations, given an orbit of residual points 
$W_0r_\mathbf{L}\in W_0\backslash T(\mathbf{L})$ for 
$\mc{H}^{IM}(G)$ (these carry the discrete 
series representations, by \cite{Opd1}), we collect the irreducible discrete series characters of the 
various 
direct summands $\mc{H}'=\mc{H}^{u,\mf{s},e}$ of $\mc{H}_{uni}(G)$ whose central character 
$W_0'r'_\mathbf{L}$ satisfies $\phi_Z(W_0'r'_\mathbf{L})=W_0r_\mathbf{L}$. 
\begin{defn}\label{dfn:packet} Given an orbit $W_0r_\mathbf{L}$ of $\mathbf{L}$-residual points 
of $\mc{H}^{IM}(G)$ we form a packet $\Pi_{W_0r_\mathbf{L}}$ consisting of 
the unipotent discrete series characters of inner forms of $G$ for which the corresponding 
discrete series representation of $\mc{H}'$ (the corresponding summand of $\mc{H}_{uni}(G)$) 
has a central character $W_0'r'_\mathbf{L}$ which satisfies 
$\phi_Z(W_0'r'_\mathbf{L})=W_0r_\mathbf{L}$ (with $\phi_Z$ as above).
\end{defn}
\begin{cor}\label{cor:sameqdeg}
By Theorem \cite[Theorem 6.1]{Opd4}, the $q$-rational part of the formal degree of all the irreducible 
characters in $\Pi_{W_0r_\mathbf{L}}$ is the same. 
\end{cor}
There exists a natural bijection (cf. \cite[Corollary B.5]{Opd1}, and 
paragraph \ref{subsub:loclang}) 
$\Lambda^e\ni[\lambda]\to W_0r_{\lambda,\mathbf{L}}$ 
between orbits of discrete unramified Langlands parameters and 
orbit of residual points  $W_0r_\mathbf{L}$ for $\mc{H}^{IM}$.   
Theorem \ref{thm:unique} implies that the packets $\Pi_{W_0r_{\lambda,\mathbf{L}}}$
defined by STMs, admit a classification in terms of local systems on the $G^\vee$-orbits 
of discrete unramified Langlands parameters: 
\begin{cor}\label{cor:packet}(\cite[Theorem 5.21]{Lu4}, \cite{Lus3}, \cite{Lusz4}, \cite{Lu6})
The packet 
$\Pi_{W_0r_{\lambda,\mathbf{L}}}$ can be parameterized by 
the disjoint union of the fibres $\tilde{\Lambda}^u_\lambda$ 
(cf. paragraph \ref{subsub:loclang} for this notation), 
where $u\in N_G(\mathbb{B})$ 
corresponds to the various inner forms of $G$ via Kottwitz's Theorem (here we 
identify $N_G(\mathbb{B})$ with the character group $\Omega/{(1-\theta)\Omega)}$
of the center ${}^LZ$ of ${}^LG$ (cf. subsection \ref{subsub:loclang})). 
\end{cor}
In \cite{OpdSol}, \cite{CiuOpd2} the discrete series characters of arbitrary affine Hecke algebra 
$\mc{H}$ are parameterized differently. 
This point of view will be quite fruitful for the applications we have in mind, especially for unequal 
parameter Hecke algebras, and this is what we will discuss next.

Let $\mathbb{L}$ be a the ring of complex Laurent 
polynomials over the natural maximal algebraic torus of (possibly unequal) Hecke parameters associated 
to the underlying root datum of $\mc{H}$ (this ring was denoted by $\Lambda$ in \cite{OpdSol}). 
Explicitly, $\mathbb{L}$ is the ring of Laurent polynomials in invertible indeterminates $v_{\alpha,\pm}$ 
(with $\alpha\in R_0$) subject to the conditions $v_{\alpha,\pm}=v_{w(\alpha),\pm}$ for 
all $\alpha\in R_0$ and $w\in W_0$, and $v_{\alpha,-}=1$ iff $1-\alpha^\vee\in W\alpha^\vee$).
We give $\mathbf{L}$ the structure of a $\mathbb{L}$-algebra by putting $v_{\alpha,\pm}=v^{m_\pm(\alpha)}$.
Then have a generic affine Hecke algebra $\mc{H}_{\mathbb{L}}$ defined over $\mathbb{L}$, 
and $\mc{H}=\mathbf{L}\otimes_{\mathbb{L}}\mc{H}_{\mathbb{L}}$.

Let $\mathbb{V}$ be the space of points of the maximal spectrum of $\mathbb{L}$ such that 
for all $\underline{\mathbf{v}}=(\mathbf{v}_{\alpha,\pm})\in\mathbb{V}$ we have 
$\mathbf{v}_{\alpha,\pm}:=\underline{\mathbf{v}}(v_{\alpha_\pm})\in\mathbb{R}_+$
for all $\alpha\in R_0$. Let $\mathbf{V}$ be the space of points $\mathbf{v}\in\mathbb{R}+$ 
of the maximal spectrum of $\mathbf{L}$. Thus we have an embedding $\mathbf{V}\hookrightarrow \mathbb{V}$, 
and $\mathbf{v}_{\alpha,\pm}=\mathbf{v}^{m_\pm(\alpha)}$. 

It was shown in \cite[Theorem 3.4, Theorem 3.5]{OpdSol} that an irreducible discrete series character 
$\delta$ of $\mc{H}$ is the specialization $\delta=\widetilde{\delta}_\mathbf{L}$ at 
$\mathbf{L}$ of a \emph{generic family} of irreducible discrete series characters 
$\widetilde{\delta}$ of $\mc{H}_\mathbb{L}$ which is well defined in an open neighborhood 
of $(\mathbf{v}^{m_\pm(\alpha)})$. Each discrete series character of $\mathbb{H}$ 
can thus be locally deformed in the parameters $m_{\pm}(\alpha)$. We will write 
such deformation as $m_{\pm}^\ep(\alpha)=m_{\pm}(\alpha)+\ep_{\pm}(\alpha)$, 
where $\ep_{\pm}(\alpha)$ vary in a sufficiently small open interval $(-\ep,\ep)\subset \mathbb{R}$. 
The irreducible discrete series representations of affine 
Hecke algebras with arbitary positive parameters $(\mathbf{v}^{m_\pm^\ep(\alpha)})$, 
so in particular of affine hecke algebras  
of the form $\mc{H}'=\mc{H}^{u,\mf{s},e}$, have been classified in \cite{OpdSol2} from the point 
of view of deformations over the ring $\mathbb{L}$. 

In the case of non simply laced irreducible root systems $R_0^{u,\mf{s},e}$, the classification 
of \cite{OpdSol2} is in terms of the \emph{generic central character map} $gcc$ 
which associates to any irreducible discrete series character a $W_0:=W(R_0^{u,\mf{s},e})$-orbit 
$W_0r$ of \emph{generic residual points}. A generic residual point 
$r\in T(\mathbb{L})$ is an $\mathbb{L}$-valued point where $\mu$ has 
maximal pole order. The set of such points is finite and invariant for the action of $W_0$.

We can choose the generic residual point $r$ always of the form (see  \cite[Theorem 8.7]{OpdSol2})
$r=s(e)\exp(\xi)\in T(\mathbb{L})$, where $e$ runs over s complete set of representatives of the 
$\Gamma:=Y/Q(R_1^\vee)$-orbits of vertices of the spectral diagram $\Sigma_s(\mc{R}^m)$, 
and $s(e)$ is the corresponding vertex of the dual fundamental alcove 
$C^\vee\subset \mathbb{R}\otimes Y$. This gives rise to a semisimple subroot system 
$R_{s(e),1}\subset R_1$, with the parameter function $m_{\pm}^{\ep,e}(\alpha)$ obtained by 
restriction of the parameters $m_{\pm}^{\ep}(\alpha)$ to the sub diagram of the geometric 
diagram $\Sigma_s(\mc{R}^m)$ obtained by omitting the vertex $e$, and replacing the group 
of diagram automorphisms $\Gamma$ by the isotropy subgroup $\Gamma_e\subset \Gamma$.
Finally,  $\xi$ denotes a \emph{linear residual point} (see \cite[Section 6]{OpdSol})
for the generic graded affine Hecke algebra defined by $R_{s(e),1}$ and $m_{\pm}^{\ep,e}(\alpha)$.
Thus $\xi$ depends linearly on parameters $m_{\pm}^{\ep,e}(\alpha)$ of the graded affine Hecke algebra. 
The specialization $W_0r_0$ of the orbit $W_0r$ at $\ep_{\pm}=0$  
is a 
confluence of finitely many orbits of generic residual points $W_0r_i$, with $i\in I_{W_0r_0}$ 
(some finite set which one can explicitly determine, see \cite[Section 6]{OpdSol}) from the explicit 
classification of linear residual points). 
For each irreducible discrete series character $\delta$ with central character $W_0r_0$, its 
unique continuous deformation $\tilde{\delta}$, locally in the Hecke parameters, has a central 
character of $\tilde{\delta}$ equal to one of the orbits $W_0r_i$ of generic residual points which 
specialize at $\ep_\pm=0$ to $W_0r_0$. This defines \cite{OpdSol2} 
a unique ``generic central character" map 
$gcc$" from the 
set of irreducible discrete series at central character $W_0r_0$ to the set $I_{W_0r_0}$ 
turns out to be bijective with the single exception of the orbit of generic residual points denoted 
$f_8$ of $F_4$ (which is one of the three generic residual points which come together 
at the weighted Dynkin diagram of the minimal unipotent orbit of $F_4$). In this 
case, there are \emph{two} generic discrete series associated to $f_8$. 
The map $gcc$ also works well for the affine Hecke algebras of type $\textup{D}_n$, 
by relating this case with affine Hecke algebras of type $\textup{C}_n(0,0)[q]$. We refer 
to \cite{OpdSol2} for details. The cases of type $\textup{E}_n$ have to be treated 
in a different way (classically as in \cite{KL}, or see \cite{CiuOpd2}).

We would like to match up these two ways of parameterizing 
the discrete series characters in the packet $\Pi_\lambda$ (with $[\lambda]\in\Lambda^e$). 
This will be important for the purpose of proving Theorem \ref{thm:HII}. 
Indeed, recall that the formal degree of $\tilde{\delta}$ was shown to be continuous in terms of 
$(\ep_{\pm}(\alpha))$ \cite[Theorem 2.60, Theorem 5.12]{OpdSol2}, and that it was given 
explicitly by the product formula \cite[Theorem 5.12]{OpdSol2}. 
In addition it is known \cite{CiuOpd2} that the formal degree of a generic family of discrete series characters 
is a product of an explicitly known rational constant and an explicit rational function of the parameters 
$v_{\alpha,\pm}$. This enables us to compute 
the rational constants in the formal degree of any discrete series character $\delta$ of 
any normalized unipotent affine Hecke algebra $\mc{H}'=\mc{H}^{u,\mf{s},e}$ by a limit 
argument, using
the generic family $\tilde{\delta}$ and its formal degree. Motivated by this, let us 
consider in more detail our parameterization with this comparison in mind.
\subsection{The parameterization for classical types}\label{par:1}
For $\textup{PGL}_n$ this was discussed in paragraph \ref{subsub:A}.

For classical groups (other than type \textup{A}) everything is governed by  Hecke algebras 
of the form $\textup{C}_n(m_-,m_+)[q^\mathfrak{b}]$, via the spectral correspondences 
of certain spectral covering morphisms. These correspondences can be made explicit 
by restriction and induction operations with respect to subalgebras of equal rank, and this 
will be discussed in detail when treating the various cases of classical type. In this paragraph 
we will concentrate on the principles for Hecke algebras of type  $\textup{C}_n(m_-,m_+)[q^\mathfrak{b}]$.

The corresponding graded affine Hecke 
algebras have a root system of type $R_{s(e),0}$ of type $\textup{B}_{n_-}\times \textup{B}_{n+}$, 
where $n_-+n_+=n$, and graded Hecke algebra parameters $(m_-,m_+)$.  
A $W_{s(e),0}$-orbit of generic linear residual points is given 
by the $W_{s(e),0}=W_0(\textup{B}_{n_-})\times W_0(\textup{B}_{n_+})$-orbit of an ordered 
pair $(\xi_-,\xi_+)$, where $\xi_\pm$ is a vector of affine linear functions of $\ep_\pm$
such that $\xi_\pm(\ep_\pm)$ is the vector of contents 
of the boxes of the ``$m_\pm^{\pm\ep}$-tableau" of a partition $\pi_\pm$ of $n_\pm$ 
with the property that 
at $\ep_\pm=0$, the extremities \cite{Slooten} of the resulting $m_\pm$-tableau 
are all distinct. 

At the ``special" parameter value $m_\pm$ (integral or half-integral), the  
$W_0(\textup{B}_{n_\pm})$-orbit of the vector $\xi_\pm(0)$ is an orbit of linear   
residual points at parameter $m_\pm$. 
By a result of Slooten \cite{Slooten} (also see \cite{OpdSol2}), the set of such orbits of linear residual points 
is in bijection with the set of  ``unipotent partitions"  
$u_\pm=u_{\xi_\pm}$ of $N_\pm:=2n_\pm+(m_\pm-\frac{1}{2})(m_\pm+\frac{1}{2})$ of 
length $l\geq m_\pm-\frac{1}{2}$, 
consisting of distinct even parts (if $m_\pm$ is half integral),  
or a partition $u_\pm$ of $N_\pm:=2n_\pm+m^2_\pm$ of length $l_\pm\geq m_\pm$, 
having distinct odd parts (if $m_\pm$ is integral). 
Let us call such a pair $u=(u_-,u_+)$ of partitions a \emph{distinguished unipotent partition of 
type $m=(m_-,m_+)$}. 
This set of partitions $\pi_\pm$ (and thus 
the set of $W_0(\textup{B}_{n_\pm})$-orbits of \emph{generic} residual points $W_0(\textup{B}_{n_\pm})\xi_\pm$  
which are confluent at $\ep_\pm=0$ to the same orbit $W_0(\textup{B}_{n_\pm})\xi_\pm(0)$) was parameterized by 
Slooten in terms of the so-called $m_\pm$-symbols $\sigma_\pm$. 
These symbols are certain Lusztig-Shoji symbols with defect $D_\pm:=\lceil m_\pm \rceil$
(see \cite{Slooten}, \cite[Definition 6.9]{OpdSol2}). 
Slooten's symbols \cite[Definition 6.11]{OpdSol2} 
attached to orbits  $W_0(\textup{B}_{n_\pm})\xi_\pm(0)$ all have the same parts, but they 
are distinguished from each other by the selection of the parts which appear in the top row. 
\begin{rem}\label{rem:number}
In particular, there exists $\binom{2l+d}{l}$ such symbols, except when $u_\pm$ contains $0$ 
as a part (which may happen if $m_\pm$ is half integral), 
in which case there are $\binom{2l+d-1}{l}$ such symbols (since $0$ must appear in 
the top row in such case).
\end{rem}

Let us call these Slooten's symbols associated to $u_\pm$ at parameter ratio $m_\pm$ 
the $u_\pm$-symbols of type $m_\pm$.  
The point of view in \cite{Slooten} is that of the deformation picture sketched above: The 
symbols are ``confluence data", and each such symbol represents an orbit of generic residual points  
which evaluates to $W_0(\textup{B}_{n_\pm})\xi_\pm(0)$ 
at the parameter ratio $m_\pm$. It is convenient to formulate the results of 
this ``abstract" classification in terms of abstract packets of representations associated to central 
characters of the discrete series of the minimal objects of 
$\mf{C}_{class}^{\textup{I}}$,  $\mf{C}_{class}^{\textup{II}}$, 
$\mf{C}_{class}^{\textup{III}\cup\textup{V}}$ and $\mf{C}_{class}^{\textup{IV}\cup\textup{VI}}$:
\begin{prop}\label{prop:packet}
Let $\mc{H}'=\textup{C}_n(m_-,m_+)[q^\mathfrak{b}]$ be an affine Hecke algebra which appears 
as object of $\mf{C}_{class}^X$ 
for $X=\textup{I}, \textup{II}, \textup{III}$ or $\textup{IV}$. 
Let $u=(u_-,u_+)$ 
be an ordered pair of distinguished unipotent partitions 
corresponding to a central character $W_0r'$ of a discrete series character 
of $\mc{H}'$, with $u_\pm$ of type $m_\pm$. 
Let $\phi:\mc{H}'\leadsto\mc{H}$ be the translation STM to a minimal object 
$\mc{H}$ of $\mc{C}_{class}^X$. 
Let $\phi_Z(W_0r')=W_0r$. 
Then the ordered pair of distinguished unipotent partitions corresponding to $W_0r$ 
is equal to $u$ as well! The set of irreducible discrete series characters of $\mc{H}'$ in 
$\Pi_{W_0r'}$ is parameterized 
by ordered pairs $(\sigma_-,\sigma_+)$, where $\sigma_\pm$ is a $u_\pm$-symbol 
of type $m_\pm$.
Let $\Pi_{W_0r}^Y$ be the disjoint union of all sets of irreducible discrete series characters of the objects 
of $\mf{C}_{class}^Y$, with $Y=\textup{I}, \textup{II}, \textup{III}\cup  \textup{V}$ or $\textup{IV}\cup \textup{VI}$, 
which are assigned to $W_0r$ in this way via the translation STMs.  
The extraspecial STM's contribute $1$, $2$ or $4$ elements to $\Pi_{W_0r}^Y$ (see 
Proposition \ref{prop:centra}), 
for each discrete series 
central character $W_0r$ of $\mc{H}$.  
\end{prop}
In the context an unramified classical group of adjoint type $G$ the Hecke algebras 
of the form $\mc{H}^{u,\mf{s}}$ are direct sums of normalized extended affine 
Hecke algebras $\mc{H}^{u,\mf{s},e}$ which are spectral coverings of objects of 
$\mf{C}_{class}$. In particular, an unramified discrete Langlands parameter $\lambda$ for 
$G$ determines (via the comparison of the Kazhdan-Lusztig classification and the 
classification of discrete series representations as in \cite{OpdSol2}) an orbit of 
residual points $W_0r$ for $\mc{H}^{IM}(G)$.
In turn, via the morphism \cite[Corollary 5.5]{Opd4} associated to this spectral covering 
map, this determines a pair $(u_-,u_+)$ of distinguished unipotent partitions in the sense of
Proposition \ref{prop:packet}, for an appropriate pair of parameters $(m_-,m_+)$ of the form 
$(m_-,m_+)=(0,\frac{1}{2}),\,(1,\frac{1}{2}),\,(\frac{1}{2},\,\frac{1}{2}),\,(0,0),\,(0,1)$ or $(1,1)$  
(see paragraph \ref{subsub:Cn}).
Working this out amounts to  determining the multiplicities and types of the normalized extended 
affine Hecke algebras $\mc{H}^{u,\mf{s},e}$, and the branching rules for the algebra inclusions 
associated to the spectral covering maps to the relevant objects of $\mf{C}_{class}$.
This is not difficult, and we can check 
that in all classical cases the STM $\Phi$ of Theorem \ref{thm:unique} gives rise to 
packets $\Pi_{W_0r}$ of discrete series characters whose members are parameterized by 
pairs of Slooten's symbols or come from an extraspecial STM (see paragraphs \ref{subsub:III+IV},
\ref{subsub:II}, \ref{ex:ex2} for more details).
 
Slooten's symbols are known to correspond with Lusztig's symbols \cite{Lu1.5}   
if one uses Lusztig's arithmetic/geometric correspondences for the ``geometric" graded affine Hecke 
algebras in the following sense. Let the central character $W_0r_0$ of a discrete series character 
$\pi$ of $\textup{C}_n(m_-,m_+)[q^\mathfrak{b}]$  
be given by the pair of unipotent partitions $(u_-,u_+)$ (where $u_\pm$ has at least 
$\lfloor m_\pm\rfloor$ parts). The set of discrete series characters with central character $W_0r_0$ 
is parameterized by the set of generic central characters $W_0r$ (see \cite{OpdSol2}) which evaluate 
to $W_0r_0$, via the map $gcc$.
In turn, these generic central characters are parameterized 
by pairs $\sigma_-,\sigma_+)$ of the Slooten symbols (with defects 
$D_\pm=\lceil m_\pm\rceil$) covering $(u_-,u_+)$.
By the results of \cite{CK}, 
\cite{Lus3}, \cite[Section 4]{Kat2}, 
the top graded part with respect to Slooten's 
functions $a_{m_\pm}$ \cite{Slooten} of the corresponding graded Hecke algebra module is the irreducible 
$W(\textup{C}_{n_-})\times  W(\textup{C}_{n_+})$-module corresponding to $(\sigma_-,\sigma_+)$,  
via the generalized Springer correspondence of \cite{LuSpa}. 
Via Proposition \ref{prop:packet}, the spectral correspondences  
of the standard STMs to $\mc{H}^{IM}(G)$ together exhaust the set of pairs of Slooten symbols 
$(\sigma_-,\sigma_+)$. 

The same is known to be true for the additional contributions to packet $\Pi_{W_0r}^Y$ coming 
from the extraspecial STMs (see \cite{CK}, \cite[Section 4]{Kat2}). These remarkable 
facts should be considered as an aspect of Langlands duality. Slooten's symbols are defined 
entirely in terms of affine Hecke algebras (describing the set of orbits of generic residual points 
specializing to the central character of a discrete series representation), whereas Lusztig's symbols 
describe cuspidal local systems on an associated nilpotent orbit of ${}^LG$.  
Comparing this with Theorem \ref{thm:unique} and Proposition \ref{prop:packet}
we see that our parametrization of $\Pi_{W_0r}$ matches with Lusztig's assignment \cite{Lu4}, \cite{Lu6} 
of unramified Langlands parameters to the members of $\Pi_{W_0r}$.

We see that the defect $(D_-,D_+)$ of an unordered pair $(\sigma_-,\sigma_+)$ of
 $u$-symbols 
for a member of $\Pi_{W_0r}$ (corresponding to a pair of distinguished unipotent 
partitions $(u_-,u_+)$)
determines the parameters of the affine Hecke algebra from 
which it originates under the STM $\Phi$. This determines the Bernstein component to which 
the corresponding discrete series character belongs (up to the action of $X^*_{un}(G)$).

The final statement of Proposition \ref{prop:packet} will be proved in \cite{FO}. The 
additional contributions from the extraspecial STMs to the packets of unipotent discrete 
series of $\textup{PCSp}_{2n}$,  $\textup{P}(\textup{CO}^0_{2n})$ and $\textup{P}((\textup{CO}^*_{2n})^0)$ 
correspond to the fact that one takes the centralizers of the discrete Langlands parameter in the Spin group.
This gives rise to a nontrivial central extension by a $C_2$ of the centralizer in $\textup{SO}_{2n}$ (or 
$\textup{SO}_{2n+1}$ respectively). 
These can be be described in detail in terms of central products of groups of type $D_8$ (the dihedral 
group with $8$ elements), $Q_8$, $C_2^2$ or $C_4$ (see \cite{Lu1.5}) and among those 
groups we typically find extraspecial $2$-groups. The precise type of the groups that arise  
is complicated, but we are merely interested the the number of their irreducible 
representations and their dimensions which is less difficult,  
following the description in \cite{Lu1.5} (and also using \cite{Reed} for the 
twisted cases) one obtains: 
\begin{prop}\label{prop:centra}
Let $\lambda$ be an unramified Langlands parameter for 
discrete series  for $\textup{PCSp}_{2n}$ (with $n\geq 2$),  
 $\textup{P}(\textup{CO}^0_{2n})$ (with $n\geq 4$),  
 and $\textup{P}((\textup{CO}^*_{2n})^0)$ (with $n\geq 4$). Then $\lambda$ determines 
an ordered pair $(u_-,u_+)$ of distinguished unipotent partitions  for the parameters 
$m=(m_-,m_+)=(0,1)$ 
(if  $G=\textup{PCSp}_{2n}$), for $m=(0,0)$ (if $G=\textup{P}(\textup{CO}^0_{2n})$) or for 
$m=(1,1)$ (if $G=\textup{P}((\textup{CO}^*_{2n})^0)$).
Let $l=(l_-,l_+)$, with $l_\pm$ the number of parts of $u_\pm$.
Thus  $u_\pm$ is a partition with distinct odd parts, and $|u|:=|u_-|+|u_+|=2n$ with $l_\pm$ 
both even if $G=\textup{P}(\textup{CO}^0_{2n})$; 
$|u|=2n+1$ with $l_-$ even and $l_+$ odd if $G=\textup{PCSp}_{2n}$; 
and $|u|=2n$ with $l_\pm$ both odd if $G=\textup{P}((\textup{CO}^*_{2n})^0)$.
Let us write $2^{(2a)+b}$ for a $2$-group of size $2^{2a+b}$ which has $2^{2a+b-1}$ 
onedimensional irreducible representations, and $2^{b-1}$ irreducibles of dimension $2^{a}$.

If $u_-$ is the zero partition, then $A_\lambda$ (as defined in paragraph \ref{subsub:loclang}) 
is of type  $2^{(l_+-1)+1}$ if 
$l_+$ is odd, and of type $2^{(l_+-2)+2}$ if $l_+$ is even.
If $u_+$ and $u_-$ are both nonzero, then  
$A_\lambda$ is of type  $2^{(l_-+l_+-4)+3}$ if $l_\pm$ are both even, 
$A_\lambda$ is of type  $2^{(l_-+l_+-3)+2}$ if $l_\pm$ are unequal modulo $2$, 
and $A_\lambda$ is of type  $2^{(l_-+l_+-2)+1}$ if $l_\pm$ are both odd. 
\end{prop}
\subsection{The parameterization for split exceptional groups}\label{par:2}
For split exceptional groups the major work to match up the irreducible discrete series characters 
of affine Hecke algebra summands of $\mc{H}_{uni}(G)$ with Lusztig's parameters has been done 
by Reeder in \cite{Re} by computing the $W$-types explicitly. With this parameterization,  
the main Theorem of \cite{Re} is known to be a special case of the conjecture 
\cite[Conjecture 1.4]{HII} (as discussed loc. cit.), which takes a lot of work out of our hands.

For the types $\textup{E}_6$ and $\textup{E}_7$ we need in addition to discuss the contribution of the nontrivial inner 
forms, which we take up in the next two paragraphs.
\subsubsection{Inner forms of the split adjoint group $G$ of type $\textup{E}_6$} The inner forms of $G$ 
are parameterized by $u\in\Omega\approx C_3$. We have $X_{un}^*(G)=\Omega^*$. 
For $u=1$ we have the following $X_{un}^*(G)$ orbits of unipotent types: $\mf{s}^1_\emptyset$, $\mf{s}^1_{\textup{D}_4}$, 
$\mf{s}^1_{\textup{E}_6[\theta]}$, $\mf{s}^1_{\textup{E}_6[\theta^2]}$. For $u\not=1$ we have the following orbits of 
unipotent types: $\mf{s}^u_\emptyset$, $\mf{s}^u_{{}^3\textup{D}_4[1]}$, $\mf{s}^u_{{}^3\textup{D}_4[-1]}$.
The orbit of $\mf{s}$ is a torsor for $(\Omega_1^\mf{s})^*$ (a quotient of $X_{un}^*(G)$).  By inspection we check:
\begin{rem}\label{rem:Omageu}
In all the cases above we have $\Omega_1^\mf{s}=\langle u\rangle:=\Omega_u\subset \Omega$.
\end{rem}
We \emph{choose} an equivariant bijection $\alpha\to \mf{s}_\alpha$ between $\Omega_u^*$ 
and the orbit of $\mf{s}$.  
Then $\mc{H}_{uni}(G)$ is isomorphic to the direct sum of the extended affine Hecke algebras $\mc{H}^{u,\mf{s}_\alpha,e}$
where $u\in\Omega$, $\mf{s}$ runs over the orbits of unipotent types, and $\alpha\in\Omega_u^*$. 
By Theorem  \ref{thm:unique} there exists an essentially unique $\Omega^*$-equivariant collection of STMs 
$\Phi^{u,\alpha}_{\mf{s}}:\mc{H}^{u,\mf{s}_\alpha,e} \leadsto \mc{H}^{IM}$. 
Assume that we have chosen such a collection of STMs.

The extended affine Hecke algebras $\mc{H}^{u,\mf{s}_\alpha,e}$ of positive rank which appear as 
summand of $\mc{H}_{uni}(G)$ are: $\textup{E}_6[q]$ (for $\mf{s}^1_\emptyset$), 
$\textup{A}_2[q^4]$ (for $\mf{s}^1_{\textup{D}_4}$), $\textup{G}_2(1,3)[q]$ (for $\mf{s}^u_\emptyset$ with $u\not=1$).
It turns out that for each $u\in\Omega$, $G^u$ has $21$ unipotent discrete series representations. 

Table \ref{table:E6} displays for each $X_{un}^*(G)=\Omega^*$-orbit of discrete unramified Langlands 
parameters: A representative $\lambda$, its isotropy group $\Omega_\lambda^*$, the group $A_\lambda$, 
and for each $u\in\Omega$, the STMs $\Phi^{u,\alpha}_{\mf{s}}$ which contribute to the corresponding 
packet $\Pi_\lambda^u$ of unipotent discrete series of $G^u$. The argument of the 
STM indicates the corresponding central character of $\mc{H}^{u,\mf{s}_\alpha,e}$, expressed in terms of 
central characters of graded Hecke algebras via \cite[Theorem 8.7]{OpdSol2}, using standard 
notations referring to distinguished nilpotent orbits for equal parameter cases, and notations  
for a corresponding generic linear central character as in \cite[Section 6]{OpdSol2} otherwise.

We choose the packets $\Pi_\lambda^u:=\Pi_{W_0r_\lambda,\bf{L}}^u$  
compatibly with respect to the $X_{un}^*(G)$-action, but the precise composition of the 
$\Pi_\lambda^u$ depends on the choices
of the STMs $\Phi^{u,\alpha}_{\mf{s}}$. 
Recall from Section \ref{subsub:loclang} that their parameterization by the 
elements of $\textup{Irr}^u(A_\lambda)$ is chosen in a $X_{un}^*(G)$-invariant way. 
By this requirement it suffices to fix the parameterization of the $\Pi_\lambda^u$ 
for a set of representatives $\lambda$ of the $X_{un}^*(G)$-orbits of discrete unramified Langlands 
parameters. With the choices made above, the parameterization of the packets $\Pi_\lambda^1$ 
is determined if we also agree that the generic member of $\Pi_\lambda$ corresponds to 
the trivial representation of $A_\lambda$.
For $u\not=1$ and $\lambda=\textup{A}_1\textup{A}_5$ or $\textup{A}_2^3$,  
more information is needed to determine the exact parameterization of the sets $\Pi_\lambda^u$ 
(of size $2$ and $3$ respectively) by a local system as in \cite{Lu4}. 
Since $A_\lambda$ is abelian here, Theorem 
\ref{thm:HII} is  independent of such choices. Therefore, we ignore this issue here.
\begin{table}[h]
\begin{tabular}{|c|c|c|c|c|}
\hline
$\lambda$&$\Omega^*_\lambda$&$A_\lambda$&STMs for $\Pi_{\lambda}^u$\\\hline
$\textup{E}_6$         & 1& $C_3$ & 
\begin{tabular}{l}$u=1$: $\Phi_\emptyset^{1,1}(\textup{E}_6)$\\
$u\not=1$: $\Phi_\emptyset^{u,1}(g_1)$\end{tabular}\\
\hline
$\textup{E}_6(a_1)$& 1& $C_3$ & 
\begin{tabular}{l}$u=1$: $\Phi_\emptyset^{1,1}(\textup{E}_6(a_1))$\\
$u\not=1$: $\Phi_\emptyset^{u,1}(g_2)$\end{tabular}\\
\hline
$\textup{E}_6(a_3)$  & 1& $S_2\times C_3$ & 
\begin{tabular}{l}$u=1$: $\Phi_\emptyset^{1,1}(\textup{E}_6(a_3))$\\
$u\not=1$: $\Phi_\emptyset^{u,1}(g_3)$; 
$\Phi_{{}^3\textup{D}_4[1]}^{u,1}$\end{tabular}\\
\hline
$\textup{A}_1\textup{A}_5$& 1&  $C_2\times C_3$ & 
\begin{tabular}{l}$u=1$: $\Phi_\emptyset^{1,1}(\textup{A}_1\textup{A}_5)$; $\Phi_{\textup{D}_4}^{1,1}(\textup{A}_2)$\\ 
$u\not=1$: $\Phi_\emptyset^{u,1}(\textup{A}_1^2)$; 
$\Phi_{{}^3\textup{D}_4[-1]}^{u,1}$\end{tabular}\\
\hline
$\textup{A}_2^3$& $C_3$&  $C_3\times C_3$& 
\begin{tabular}{l}$u=1$: $\Phi_\emptyset^{1,1}$; $\Phi_{\textup{E}_6[\theta]}^{1,1}$; 
$\Phi_{\textup{E}_6[\theta^2]}^{1,1}$\\
$u\not=1$: $\Phi_\emptyset^{u,\alpha}(\textup{A}_2)\ (\alpha\in\Omega^*)$\end{tabular}\\
\hline
\end{tabular}
\caption[Table caption text]{The packets $\Pi_\lambda^u$ for type $\textup{E}_6$ and the contributing STMs}
\label{table:E6}
\end{table}
\subsubsection{The parameterization for Inner forms of the split adjoint group $G$ of type $\textup{E}_7$}
We use the same setup and notations as for the case of $\textup{E}_6$.
The inner forms of $G$ 
are parameterized by $u\in\Omega\approx C_2$. We have $X_{un}^*(G)=\Omega^*$. 
For $u=1$ we have the following $X_{un}^*(G)$ orbits of unipotent types: $\mf{s}^1_\emptyset$, $\mf{s}^1_{\textup{D}_4}$, 
$\mf{s}^1_{\textup{E}_6[\theta]}$, $\mf{s}^1_{\textup{E}_6[\theta^2]}$, 
$\mf{s}^1_{\textup{E}_7[\xi]}$, $\mf{s}^1_{\textup{E}_7[-\xi]}$. For $u=-1$ we have the following orbits of 
unipotent types: $\mf{s}^u_\emptyset$, $\mf{s}^u_{{}^2\textup{A}_5}$, $\mf{s}^u_{{}^2\textup{E}_6[1]}$, 
$\mf{s}^u_{{}^2\textup{E}_6[\theta]}$, $\mf{s}^u_{{}^2\textup{E}_6[\theta^2]}$.
The orbit of $\mf{s}$ is a torsor for $(\Omega_1^\mf{s})^*$ (a quotient of $X_{un}^*(G)$).  By inspection we check 
that the analog of Remark \ref{rem:Omageu} again holds.

The extended affine Hecke algebras $\mc{H}^{u,\mf{s}_\alpha,e}$ of positive rank which appear as 
summand of $\mc{H}_{uni}(G)$ are for $u=1$: $\textup{E}_7[q]$ (for $\mf{s}^1_\emptyset$), 
$\textup{B}_3(4,1)[q]$ (for $\mf{s}^1_{\textup{D}_4}$), $\textup{C}_1(9,9)[q]$ 
(for $\mf{s}^1_{\textup{E}_6[\theta^i]}$), and moreover for $u=-1$:   
$\textup{F}_4(1,2)[q]$ (for $\mf{s}^{-1}_\emptyset$), and 
$\textup{C}_1(9,7)[q]$ (for $\mf{s}^{-1}_{{}^2\textup{A}_5}$).
 
For each $u\in\Omega$, $G^u$ has $44$ unipotent discrete series representations. 
See Table \ref{table:E7}. In order to understand the $u=-1$ cases of 
$\lambda=\textup{A}_1\textup{D}_6, \textup{A}_1\textup{D}_6[93],
\textup{A}_1\textup{D}_6[75]$, the following remark is important:
\begin{rem}\label{rem:double}
The STMs $\Phi^{-1,\pm 1}_\emptyset:\textup{F}_4(1,2)[q]\leadsto \textup{E}_7[q]$ were constructed at the end of 
paragraph \ref{subsub:exc}. Let us write $\Phi:=\Phi^{-1,\pm 1}_\emptyset$, and let $\Psi$ denote the 
nontrivial essentially strict spectral automorphism of $\textup{E}_7[q]$. 
Then $\Phi$ has the following remarkable 
property (which is easy to check knowing the spectral map diagram):  
Let $\lambda_{[3]}, \lambda_{[111]},  \lambda_{[21]}$ be the three 
orbits of residual points of type $\textup{A}_1\times \textup{C}_3$, and let $\mu_{[4]}, \mu_{[31]}, 
\mu_{[22]}$ be the three orbits of residual points of type $\textup{B}_4$. Enumerate 
these as $\lambda_i$ and $\mu_i$ ($i=1,2,3$) in this order.
Then $\Phi_Z(\lambda_i)=(\Psi_Z\circ\Phi_Z)(\mu_i)$ for all $i$. 
\end{rem}
The precise constituents of the packets $\Pi_\lambda^u$  depend 
on the choices of the STMs $\Phi_\mf{s}^{u,\alpha}$. Again the exact parameterization of the 
packets by $\textup{Irr}^u(A_\lambda)$ is not uniquely determined for all $\lambda$ and $u$. 
If $A_\lambda$ is abelian this does not affect the statement of Theorem \ref{thm:HII}, and
we ignore this problem here (but: see  \cite{CiuOpd2}). 
But for 
$\lambda=\textup{E}_7(a_5)$ and $u=-1$ we need to be more careful. 
This packet corresponds to the generic central character $f_8$ (notation as in \cite[Section 7]{OpdSol2}) 
of $\textup{F}_4(1,2)[q]$. 
As was explained in \cite{Re}, \cite[Section 7]{OpdSol2}, 
\cite[paragraph 3.5.2]{CiuOpd2}, there are \emph{two} algebraic generic parameter families 
$\delta_8'$ and $\delta_8''$ of irreducible discrete series characters of $\textup{F}_4(m_1,m_2)[q]$ which 
stay irreducible discrete series for all $m_1, m_2>0$ (and in particular the corresponding 
$W_0(\textup{F}_4)$-types are independent of the parameters). One of these ($\delta_8'$ say) is 
$10$-dimensional, and specializes at equal parameters for $\textup{F}_4$ to the 
discrete series \cite{Re} with Langlands parameters $(\textup{F}_4(a3), [4])$. The other, $\delta_8''$ 
restricts to the discrete series with Langlands parameters $(\textup{F}_4(a3), [22])$.
Comparing with the tables in \cite{Spa}, we see that $\delta_8'$ corresponds with 
$(\textup{E}_7(a_5),-[3])$, while $\delta''_8$ corresponds to $(\textup{E}_7(a_5),-[21])$.  
On the other hand, by \cite{OpdSol2} and \cite{CiuOpd2} we conclude that  
$\textup{fdeg}(\delta''_8)=2\textup{fdeg}(\delta'_8)$, and this is also equal to 
$2\textup{fdeg}({}^2\textup{E}_6[1])$. In view of the above Langlands parameters
this is in accordance with the conjecture 
\cite[Conjecture 1.4] {HII}. \begin{table}[h]
\begin{tabular}{|c|c|c|c|c|}
\hline
$\lambda$&$\Omega^*_\lambda$&$A_\lambda$&STMs for $\Pi_{\lambda}^u$\\\hline
$\textup{E}_7$         & 1& $C_2$ & 
\begin{tabular}{l}$u=1$: $\Phi_\emptyset^{1,1}(\textup{E}_7)$\\
$u=-1$: $\Phi_\emptyset^{u,1}(f_1)$\end{tabular}\\
\hline
$\textup{E}_7(a_1)$& 1& $C_2$ & 
\begin{tabular}{l}$u=1$: $\Phi_\emptyset^{1,1}(\textup{E}_7(a_1))$\\
$u=-1$: $\Phi_\emptyset^{u,1}(f_2)$\end{tabular}\\
\hline
$\textup{E}_7(a_2)$& 1& $C_2$ & 
\begin{tabular}{l}$u=1$: $\Phi_\emptyset^{1,1}(\textup{E}_7(a_2))$\\
$u=-1$: $\Phi_\emptyset^{u,1}(f_3)$\end{tabular}\\
\hline
$\textup{E}_7(a_3)$  & 1& $S_2\times C_2$ & 
\begin{tabular}{l}$u=1$: $\Phi_\emptyset^{1,1}(\textup{E}_7(a_3))$\\
$u=-1$: $\Phi_\emptyset^{u,1}(f_4)$\end{tabular}\\
\hline
$\textup{E}_7(a_4)$  & 1& $S_2\times C_2$ & 
\begin{tabular}{l}$u=1$: $\Phi_\emptyset^{1,1}(\textup{E}_7(a_4))$\\
$u=-1$: $\Phi_\emptyset^{u,1}(f_6)$\end{tabular}\\
\hline
$\textup{E}_7(a_5)$  & 1& $S_3\times C_2$ & 
\begin{tabular}{l}$u=1$: $\Phi_\emptyset^{1,1}(\textup{E}_7(a_5))$\\
$u=-1$: $\Phi_\emptyset^{u,1}(f_8)$; 
$\Phi_{{}^2\textup{E}_6[1]}^{u,1}$\end{tabular}\\
\hline
$\textup{A}_1\textup{D}_6$& 1&  $C_2\times C_2$ & 
\begin{tabular}{l}$u=1$: $\Phi_\emptyset^{1,1}(\textup{A}_1\textup{D}_6)$; $\Phi_{\textup{D}_4}^{1,1}(\textup{B}_3)$\\ 
$u=-1$: $\Phi_\emptyset^{u,1}(\textup{A}_1\textup{C}_3)$; $\Phi_\emptyset^{u,-1}(\textup{B}_4)$\end{tabular}\\
\hline
$\textup{A}_1\textup{D}_6[93]$& 1&  $C_2\times C_2$ & 
\begin{tabular}{l}$u=1$: $\Phi_\emptyset^{1,1}(\textup{A}_1\textup{D}_6[93])$; $\Phi_{\textup{D}_4}^{1,1}(\textup{B}_3[111])$\\ 
$u=-1$: $\Phi_\emptyset^{u,1}(\textup{A}_1\textup{C}_3[111])$; $\Phi_\emptyset^{u,-1}(\textup{B}_4[31])$\end{tabular}\\
\hline
$\textup{A}_1\textup{D}_6[75]$& 1&  $C_2\times C_2$ & 
\begin{tabular}{l}$u=1$: $\Phi_\emptyset^{1,1}(\textup{A}_1\textup{D}_6[75])$; $\Phi_{\textup{D}_4}^{1,1}(\textup{B}_3[21])$\\ 
$u=-1$: $\Phi_\emptyset^{u,1}(\textup{A}_1\textup{C}_3[21])$; $\Phi_\emptyset^{u,-1}(\textup{B}_4[22])$\end{tabular}\\
\hline
$\textup{A}_2\textup{A}_5$& 1&  $C_3\times C_2$ & 
\begin{tabular}{l}$u=1$: $\Phi_\emptyset^{1,1}(\textup{A}_2\textup{A}_5)$; $\Phi_{\textup{E}_6[\theta^i]}^{1,1}$\\ 
$u=-1$: $\Phi_\emptyset^{u,1}(\textup{A}_2\textup{A}_2)$; 
$\Phi_{{}^2\textup{E}_6[\theta^i]}^{u,1}$\end{tabular}\\
\hline
$\textup{A}_3^2\textup{A}_1$ & $C_2$&  $C_4\times C_2$ & 
\begin{tabular}{l}$u=1$: $\Phi_\emptyset^{1,1}(\textup{A}_3^2\textup{A}_1)$; $\Phi_{\textup{D}_4}^{1,1}(\textup{A}_1^3)$; 
$\Phi^{1,1}_{\textup{E}_7[\pm\xi]}$\\ 
$u=-1$: $\Phi_\emptyset^{u,\pm 1}(\textup{A}_3\textup{A}_1)$; 
$\Phi_{{}^2\textup{A}_5}^{u,\pm1}(\textup{A}_1)$\end{tabular}\\
\hline
$\textup{A}_7$& $C_2$&  $C_4$& 
\begin{tabular}{l}$u=1$: $\Phi_\emptyset^{1,1}(\textup{A}_3^2\textup{A}_1)$; $\Phi_{\textup{D}_4}^{1,1}(\textup{A}_3)$\\
$u=-1$: $\Phi_{{}^2\textup{A}_5}^{u,\pm1}(\textup{A}_1')$\end{tabular}\\
\hline
\end{tabular}
\caption[Table caption text]{The packets $\Pi_\lambda^u$ for type $\textup{E}_7$ and the contributing STMs}
\label{table:E7}
\end{table}
\subsection{The parameterization for non-split quasi-split exceptional groups}\label{par:3}
The parameterization and the STMs for the remaining twisted exceptional cases were discussed in 
\ref{par:quasi}. By Corollary \ref{cor:sameqdeg}, Corollary \ref{cor:packet} and Remark \ref{rm:uniqueLLP} 
it then follows that Lusztig's parameterization of $\Pi_{W_0r_\mathbf{L}}$ is uniquely determined by 
this, so this gives rise to a canonical matching of Lusztig's parameterization and our parameterization.
\subsection{The formal degree of unipotent discrete series representations}\label{subsub:formdeg}
The application in this section is independent of the uniqueness result based on \cite{FO}.

A general conjecture has been put forward by \cite{HII} expressing the formal degree 
of a discrete series character in terms of the adjoint gamma factor (also see \cite{GR}). 
Recall that our standing assumption is the $\mb{G}$ is a connected, absolutely simple algebraic 
group of adjoint type, defined and quasisplit over $k$, and split over $K$.
 
In order to formulate the conjecture in our setting, we first should note that 
the Haar measures in \cite{HII} are equal to those we have used (following \cite{DeRe}) times 
$v^{-\textup{dim}(\mb{G})}$. Hence the formal degrees in \cite{HII} are $v^{\textup{dim}(\mb{G})}$
times the formal degree in our setting. Let $G^u$ be an inner form of $G$.
Given a discrete unramified local Langlands parameter 
$\lambda$  for $G$, we defined $\mc{A}_\lambda$ (see \ref{subsub:loclang}).
Suppose that for an irreducible representation $\rho\in \textup{Irr}(\mc{A}_\lambda^u)$ we 
have a corresponding unipotent (or unramified) discrete series representation
$\pi_{(\lambda,\rho)}$ of $G^{F_u}$, satisfying the expected character identities 
as asserted in the local Langlands conjecture.

Then \cite[Conjecture 1.4] {HII} (also see \cite[Conjecture 7.1]{GR}) is equivalent to
(with our normalization of Haar measures): 
\begin{equation}\label{eq:HII}
\textup{fdeg}(\pi_{(\lambda,\rho)})=\pm\frac{\textup{dim}(\rho)}{|A_\lambda|}v^{-\textup{dim}(\mb{G})}\gamma(\lambda)
\end{equation}
where $\gamma$ denotes the adjoint gamma factor of the discrete local Langlands 
parameter $\lambda$. 
Following \cite[Lemma 3.4]{HII}, it is easy to show that  (using the notations of \ref{subsub:loclang})
\begin{equation}\label{eq:gammamu}
\gamma(\lambda)=\pm v^{\textup{dim}(\mb{G})}(\mu^{IM})^{(\{r\})}
\end{equation}
where we should remind the reader that the normalization of the $\mu$-function $\mu^{IM}$ 
of $\mc{H}^{IM}(G)$ is 
given by the trace $\tau^{IM}$ such that $\tau^{IM}(1)=\textup{Vol}(\mathbb{B}^{F})^{-1}$.
It was verified in \cite{HII}
that Reeder's results \cite{Re} for Iwahori spherical discrete series representations of 
adjoint, split exceptional groups over a nonarchimedean field are compatible with the conjecture.
We are now able to extend this result to arbitrary adjoint absolutely simple groups over a nonarchimedean 
local field which split over an unramified field extension. 
\begin{thm}\label{thm:HII}
Conjecture \cite[Conjecture 1.4] {HII} 
(equivalent to equation (\ref{eq:HII}))
holds for all unipotent discrete series representations of an unramified 
connected absolutely simple group $G^u$ of adjoint type defined over a nonarchimedean local field $k$, 
where we use Lusztig's parametrization of unipotent discrete series representations as Langlands 
parameters. 
\end{thm}
\begin{proof} 
We need to consider the classical groups,  the nontrivial inner forms 
of split exceptional groups, and the non-split quasi-split exceptional groups.
The way in which we assign unramified discrete Langlands parameters to the members
of the packets $\Pi_{W_0r_{\lambda,\mathbf{L}}}$ of discrete series characters 
of Definition \ref{dfn:packet} for these cases was explained in the Subsections 
\ref{par:1}, \ref{par:2} and \ref{par:3}.

We know that $\pi_{(\lambda,\rho)}$ corresponds via Lusztig's arithmetic-geometric 
correspondences to an irreducible discrete series representation $\delta_{\lambda,\rho}$ of 
an extended affine Hecke algebra of type $\mc{H}^{u,\mf{s},e}$ for some cuspidal type $\mf{s}$
of $G^u$. By our main Theorem \ref{thm:unique} there exists an STM 
$\phi:\mc{H}^{u,\mf{s},e}\leadsto \mc{H}^{IM}(G)$ such that $\phi_Z(\textup{cc}(\delta_{\lambda,\rho}))=W_0r$, 
and we have:
\begin{equation}
\textup{fdeg}(\pi_{(\lambda,\rho)})=\textup{fdeg}_{\mc{H}^{u,\mf{s},e}}(\delta_{(\lambda,\rho)})=
c_{(\lambda,\rho)}(\mu^{IM})^{(\{r\})}
\end{equation}
for some rational constant $c_{(\lambda,\rho)}\in\mathbb{Q}$. 
Combining (\ref{eq:HII}) and (\ref{eq:gammamu})
we see that what is necessary to verify in order to prove the conjecture in these cases is 
that 
\begin{equation}
c_{(\lambda,\rho)}=\pm\frac{\textup{dim}(\rho)}{|A_\lambda|}
\end{equation}
In \cite[Section 5.1]{GR} it was shown that 
\begin{equation}
\gamma(\lambda)=|C_\lambda^F|q^{N_{\lambda}}\gamma(\lambda)_q
\end{equation}
with $\gamma(\lambda)_q$ a q-rational number, $N_\lambda\in\mathbb{N}$ (which is in fact 
always $0$ with our definition of q-rational numbers, but this is not important here), 
and $C_\lambda^F\subset A_\lambda$
a normal subgroup such that 
\begin{equation}\label{eq:quot}
A_\lambda/C_\lambda^F\approx (\pi_0(M_\lambda))^F,
\end{equation}
is the group of $F$-fixed points in the component group 
of the 
centralizer $M_\lambda$ of $\lambda|_{\textup{SL}_2(\mathbb{C})}$ in $G^\vee$. 
(The group $C_\lambda^F$ is the group of $F$-fixed points 
in the identity component $M_\lambda^0$ of $M_\lambda$ (a torus).)
With this notation we are 
reduced to proving that 
\begin{equation}\label{eq:HIIconv}
\textup{fdeg}_{\mc{H}^{u,\mf{s},e}}(\delta_{(\lambda,\rho)})\sim
\pm\frac{\textup{dim}(\rho)}{|(\pi_0(M_\lambda))^F|}q^{N'_{\lambda}}
\end{equation}
(where $\sim$ refers to asymptotic behavior if $q$ tends to $0$)
for some $N'_\lambda\in\mathbb{N}$.
Let us write $\lambda_{ad}$ for the composition of $\lambda$ with the 
canonical homomorphism of $G^\vee$ to $G^\vee_{ad}$.
In the twisted cases it is helpful to note that $A_\lambda/{}^LZ$ 
is the centralizer of 
$\lambda_{ad}|_{\textup{SL}_2(\mathbb{C})}$
in $C_{G^\vee_{ad}}(\lambda_{ad}(F))$, and realizing that $\lambda_{ad}(F)$
is a semisimple element of $G^\vee_{ad}\rtimes\langle\theta\rangle$
of the form $(s,\theta)$, where $s$ is a vertex of the alcove 
of the restricted root system $R_0^\theta$ consisting of roots of $R_0$ 
restricted to $\mf{t}^\theta$, extended to an affine reflection group by the lattice 
of translations obtained from projecting the coweight lattice $P(R_0^\vee)$ 
onto $\mf{t}^\theta$ (see \cite{Reed}). The semisimple centralizers 
$C_{G^\vee_{ad}}(\lambda_{ad}(F))$ are described by Reeder in
 \cite{Reed}.
 
 This amounts to a long list of case by case verifications. 
 The case of $\textup{PGL}_{n+1}$ is easy. 
 For $u$ of order $(m+1)|(n+1)$ we have $\Omega^{\mf{s},e}_1=\langle u\rangle\approx C_{m+1}$.
 Hence $\mc{H}^{IM}(G^u)$ is isomorphic to a direct sum of $m+1$ copies of $A_{d}[q^{m+1}]$ 
 (with $n+1=(d+1)(m+1)$), 
 normalized by $\tau(1)=(m+1)^{-1}[m+1]_q^{-1}$ (cf. \ref{subsub:aniso} and Proposition \ref{prop:balanced}). 
 This yields $n+1$ unipotent discrete series 
 characters, each with formal degree $(n+1)^{-1}[n+1]_q^{-1}$. In total we thus obtain 
 $n+1$ packets of unipotent 
 discrete series characters, each with $n+1$ members (one element for each inner form).
 
 The case of $G=\textup{PU}_{2n}$ or $G=\textup{PU}_{2n+1}$ is easy too, since all unipotent 
affine Hecke algebras are in a generic parameter situation here, in the sense of \cite{CKK}. 
It is shown in \cite{CKK} that the rational constant in the 
 formal degree is then independent of the particular discrete series we consider (of a 
 given Hecke algebra of such kind). Looking at the Steinberg character \cite[Equation (6.26)]{Opd1}
  we easily check therefore that the 
 rational constants for all unipotent discrete series are equal to $|{}^LZ|^{-1}$ 
 (so $\frac{1}{2}$ if $n$ is odd, and $1$ otherwise).
 
 In the remaining classical cases one also uses the results of \cite{CKK}, where it was shown that 
 the rational constant factor of the formal degree of a generic discrete series representation of a generic 
 multi parameter type $C_n$ affine Hecke algebra specialized at non-special parameters is equal 
 for all generic discrete series.  
 All rational constants for discrete series representation of affine Hecke algebras with special parameters 
 can be subsequently be computed by this result by a limit procedure, since \cite{OpdSol2} shows that any discrete 
 serie representation is the limit of a generic continuous family of discrete series representations in a small 
 open set in the parameter space, and 
 that the formal degree 
 is locally continuous in the parameters. We will use Propositions \ref{prop:packet} and 
 \ref{prop:centra} to build the packets from the various unipotent Hecke algebras, and compute  
 the expected rational constants according to \cite[Proposition 1.4]{HII}. On the spectral side 
 one again relies on the results of \cite{CKK} and \cite{OpdSol2} to compute the rational constants.
  
 For exceptional cases the results of \cite{Re} prove the statement for the 
 unipotent discrete series of all split adjoint groups $G$. 
 For the non-split cases, and 
 the nontrivial inner forms of $\textup{E}_6$ and $\textup{E}_7$ more work needs to be done, 
 but this follows the same scheme as discussed above, with the help of 
 \cite{Reed}, the tables 
 in Subsections \ref{par:2} and \ref{par:3}, and the results of \cite{OpdSol2}, \cite{CiuOpd2}.
That is, we need to compute the rational constants for the formal degrees 
of discrete series representations of the multi parameter affine Hecke algebras arising from 
these non-split cases. The classical Hecke algebras 
are treated as before, so that leaves the exceptional unequal parameter 
Hecke algebras which appear in this way. We find that we need to compute the formal degrees of 
$\textup{G}_2(3,1)$ (for type ${}^3\textup{D}_4$), $\textup{G}_2(1,3)$ (for type ${}^3\widetilde{\textup{E}}_6$), 
$\textup{F}_4(2,1)$ (for type ${}^2\textup{E}_6$), 
and $\textup{F}_4(1,2)$ (for type ${}^2\widetilde{\textup{E}}_7$).  
The first main observation in this kind of computations is the fact \cite{OpdSol2} that 
any discrete series character $\delta$ defines a \emph{generic central character} 
$gcc(\delta)=W_0r$ (an orbit of generic residual points) and extends uniquely to a continuous 
family of discrete series characters on a connected component $C$ of the open subset of 
space of positive parameters of the Hecke algebra on which $W_0r$ is still residual.
Moreover $\textup{fdeg}(\delta)$ depends continuously on the parameters in such a 
continuous family of discrete series.  But there is a deeper fact which is very useful.
The formal degree of a generic 
family of discrete series representations (in the sense of \cite{OpdSol2}) 
depends algebraically on the parameters, and this expression only depends 
on the elliptic class of the limit $\mathbf{q}\to 1$ of the discrete series representation 
(a representation of $W$). This result follows essentially from \cite{COT} and the Euler-Poincar\'e 
formula in \cite{OpdSol2}, using the argument of \cite[Proposition 5.6]{CiuOpd1}
in the unequal parameter setting. This implies (see \cite{CiuOpd2} for details) that 
the formal degree of a generic families associated with the same generic central character $W_0r$ 
but defined on different connected components $C$ and $C'$ of the open subset of the 
positive parameter space  where $W_0r$ is residual, is given by \emph{the same} 
algebraic expression (provided the 
families define the same elliptic representation of $W$), except possibly for a sign change.
(This result generalizes the result of \cite{CKK} to arbitrary Hecke algebras).
This algebraic expression for the formal degree is a product formula (see \cite{OpdSol2})
of terms $(1\pm M)^{\pm 1}$ where $M$ is a monomial in the parameters, multiplied 
by a rational constant $d$ (which only depends on an elliptic representation of $W$), a monomial 
in the parameters, and a sign. 
The upshot is that in order to compute $\textup{fdeg}(\delta)$ it is sufficient to 
compute $\textup{fdeg}(\delta')$ for any discrete series $\delta'$ with $gcc(\delta')=W_0r$
at any positive parameter $q'$ where $W_0r(q')$ is residual, provided $\delta$ and $\delta'$ define the 
same elliptic representation of $W$.
Using the results of \cite{Re}, we can find a $\delta'$ and $q'$ where the constants are known,  
for every generic family.
Hence the generic rational constants $d$ can be determined, and from this we can determine 
the formal degree at any singular parameter line in the parameter space by continuity.

See Example \ref{ex:ex2} for more details in the case ${}^3\textup{D}_4$.  
For $\textup{F}_4(2,1)$ we have a similar situation, here we need to cross the singular lines 
$\frac{k_1}{k_2}=\frac{6}{5}, \frac{4}{3}, 
\frac{3}{2}, \frac{5}{3}$ in the parameter space. For this we need to know the confluence relations 
of the generic 
discrete series at these singular lines. This can be deduced 
from \cite[Table 3]{OpdSol2}. The considerations are similar as in Example \ref{ex:ex2}.
Similarly for ${}^3\widetilde{\textup{E}}_6$ and ${}^2\widetilde{\textup{E}}_7$.
\end{proof}
\subsubsection{Unipotent representations of inner forms of $\textup{PCSp}_{2n}$,   
$\textup{P}(\textup{CO}^0_{2n})$, $\textup{P}((\textup{CO}^*_{2n+2})^0)$}\label{subsub:III+IV}
In these cases, a unipotent affine Hecke algebra is always isomorphic to a direct sum of finitely many 
copies of a normalized affine Hecke algebra which is related to an object of 
$\mf{C}_{class}^{\textup{III}\cup\textup{V}}$ or $\mf{C}_{class}^{\textup{IV}\cup\textup{VI}}$ through 
a (finite) sequence of spectral covering maps. 

Let us first compute the rational factors appearing in the formal degrees of discrete series representations  
of a normalized affine Hecke algebra $\mc{H}$ of type $\textup{C}_d(m_-,m_+)[q]$ 
with $m_\pm\in\mathbb{Z}$, normalized by $\tau(1)=1$. Using the 
group $\mathbb{D}_8$ of spectral isomorphisms (see \cite[Remark 7.7]{Opd4})  we may, 
without loss of generality, assume that $0\leq m_-\leq m_+$. 

As described in Subsection \ref{subsub:partition}, the discrete series of $\mc{H}$  
are parameterized by ordered pairs $(\sigma_-,\sigma_+)$ of symbols associated to an ordered pair 
$(u_-,u_+)$ of distinguished unipotent partitions for the pair of parameters $m=(m_-,m_+)$
(so $u_\pm$ is a partition of $m_\pm^2+2d_\pm$). 
By Slooten's ``joining procedure" \cite[Theorem 5.27]{Slooten} (see also the explanation in 
Subsection \ref{subsub:partition}), the set of symbols  
$\sigma_\pm$ corresponds bijectively to the set of partitions $\pi_\pm$ of $d_\pm$ 
whose $m_\pm$-tableaux have distinct extremities in the sense of \cite{Slooten} and 
such that the corresponding orbit of linear residual points corresponds to $u_\pm$.
Then vector of contents of this $m^\ep_\pm:=m_\pm+\ep_\pm$-tableau of $\pi_\pm$ defines, for all 
$\ep_\pm$ sufficiently small, 
linear residual point $\xi_\pm(m_\pm+\epsilon_\pm)$ whose 
$W_{n_\pm}$-orbit generically supports a unique  
discrete series representation,  which we will denote by $\delta_{(\pi_-,\pi_+)}(\epsilon_-,\epsilon_+)$.
\begin{thm}\label{thm:ratSPSO}
Let $m=(m_-,m_+)\in\mathbb{Z}^2$ such that $0\leq m_-\leq m_+$.
Consider $\pi_{(u_-,u_+),(\sigma_-,\sigma_+)}:=\delta_{(\pi_-,\pi_+)}(0,0)$ as a discrete series of the 
normalized affine Hecke algebra $(\mc{H}, \tau)$ of type $\textup{C}_d(m_-,m_+)[q]$, 
normalized by $\tau(1)=1$. Let $(u_-,u_+)$ be the pair of unipotent partitions of type 
$m=(m_-,m_+)$ associated with pair $(T_{m_-}(\pi_-),T_{m_+}(\pi_+))$ of $m$-tableaux, 
and let $(\xi_-(m_-+\epsilon_-),\xi_+(m_++\epsilon_+))$ be the corresponding 
pair of linear residual points. 
Let $\textup{fdeg}_{\mathbb{Q}}(\pi_{(u_-,u_+),(\sigma_-,\sigma_+)})$ denote the rational factor of 
$\textup{fdeg}(\pi_{(u_-,u_+),(\sigma_-,\sigma_+)})$.
Let $u_-\cup u_+$ be the partition which one obtains by concatenating $u_-$ and $u_+$ 
and rearranging the parts as a partition (our convention will be to arrange the parts 
in a nondecreasing order). 
Let $\#(u)$ denote the number of distinct parts of a partition $u$. Then we have:
\begin{equation}
\textup{fdeg}_{\mathbb{Q}}(\pi_{(u_-,u_+),(\sigma_-,\sigma_+)})=2^{-\#(u_-\cup u_+)+m_+}
\end{equation}
\end{thm} 
\begin{proof}
Let the central character of 
$\delta_{(\pi_-,\pi_+)}(\epsilon_-,\epsilon_+)$ be denoted by 
$W_0r$, where $r:=r_{(\pi_-,\pi_+)}(\ep_-,\ep_+)
=(-r_-(\ep_-),r_+(\ep_+))$ with 
$r_\pm(\ep_\pm):=\exp(\xi_\pm(m_\pm+\epsilon_\pm))$.
We have   
$\textup{fdeg}(\delta_{(\pi_-,\pi_+)}(\ep_-,\ep_+))=
cm_{W_0r}$ by \cite[Theorem 4.6]{OpdSol2}, 
with $c\in\mathbb{Q}^\times$, and with 
the rational function $m_{W_0r}$, 
defined by \cite[(39)]{OpdSol2}. 
The constant $|c|$ is known 
\cite[Theorem C]{CKK} and turns out to be equal to $1$, independent of the 
parameters and of $(\pi_-,\pi_-)$ (there is a harmless but unfortunate mistake in 
\cite[Definition 4.3]{CKK} (the factor $\frac{1}{2}$ on the right hand side should 
not be there, see the update on arXiv) which resulted in the erroneous 
extra factor $\frac{1}{2}$ in  
\cite[Theorem C]{CKK}). 
We have the basic regularity result 
\cite[Corollary 4.4]{OpdSol2}. Hence $\textup{fdeg}(\pi_{(u_-,u_+),(\sigma_-,\sigma_+)})$ equals the 
limit for $(\ep_-,\ep_+)\to(0,0)$ of $m_{W_0r}$.

For an arbitrary root datum $\mc{R}$ with parameter function 
$m_\pm^\ep(\alpha)=m_\pm(\alpha)+\ep_\pm(\alpha)$, and a generic 
residual point $r$ which specializes to a residual point at $\ep_\pm=0$, 
we can rewrite $m_{W_0r}$ in the following form 
(cf. \cite[(13)]{Opd4}))
(here $N=N(\ep)$ is an affine linear function of the deformation 
parameters $\ep$):
\begin{align*}
&m_{W_0r}=\\
&v^N
\prod_{\alpha\in R_{0,+}}\frac{(1+\alpha(r))^2(1-\alpha(r))^2}
{(1+q^{m_-^{\ep}(\alpha)}\alpha(r))
(1+q^{-m_-^{\ep}(\alpha)}\alpha(r))
(1-q^{m_+^{\ep}(\alpha)}\alpha(r))
(1-q^{-m_+^{\ep}(\alpha)}\alpha(r))}\\
\end{align*}
where a factor of the numerator or of the denominator 
has to be omitted if it is identically equal to $0$ as a function 
of $\ep$ in a neighborhood of $0$. 

In our present case, 
$R_{0,+}=\{e_i\pm e_j\mid 1\leq i< j\leq d\}\cup \{e_i\mid 1\leq i\leq d\}$.

For a positive root $\alpha$ of type $\textup{D}$, we have $m_-^{\ep}(\alpha)=0$ and 
$m_+^{\ep}(\alpha)=1$; for positive root $\beta$ of type $\textup{A}_1^d$, we have 
$m_-^{\ep}(\beta)=m_-+\ep_-$ and  $m_+^{\ep}(\beta)=m_++\ep_+$. 
In the limit $\ep=(\ep_-,\ep_+)\to 0$ some of the factors which are generically nonzero 
tend to $0$, but the number of those factors in the numerator and the denominator 
is equal by \cite[Corollary 4.4]{OpdSol2} (or \cite{Opd3}). This potentially produces rational 
factors in the limit, but actually all such factors (for type $\textup{D}$ roots 
as well as for type $\textup{A}_1^d$ roots) are of the form $(1-q^{\pm 2\ep_-})$,  
$(1-q^{\pm 2\ep_+})$, or $(1-q^{\pm (\ep_--\ep_+)})$. 
For each of these three types, the total number of these factors 
in the numerator and in the denominator has to be equal, 
by the above regularity result. Hence altogether these factors yield at most a sign
in the limit, 
and that does not 
contribute to $\textup{fdeg}_{\mathbb{Q}}(\pi_{(u_-,u_+),(\sigma_-,\sigma_+)})$. In addition 
we have factors $(1+q^{l(\ep)})$, with $l(\ep)$ linear in $\ep$,  
in the denominator and in the numerator. Each such factor yields a factor $2$, 
irrespective of the precise form of $l(\ep)$. Let the total number of factors $2$ thus 
obtained be denoted by $M$.
In order to count $M$, let us 
write $h_{u_\pm}^{m_\pm}(x)$ for the number coordinates of $\xi_\pm(m_\pm)$
which are equal to $x$ (for $x\in\mathbb{Z}_{\geq 0}$) (cf. \cite{HO}, or 
\cite[Proposition 6.6]{OpdSol2}). We also define 
$H_{u_\pm}^{m_\pm}(x)=h_{u_\pm}^{m_\pm}(x)$ for $x>0$, and $H_{u_\pm}^{m_\pm}(0)=2h_{u_\pm}^{m_\pm}(0)$.
Finally, if $h$ is a function on $\mathbb{Z}$ we define $\delta(h)(x):=h(x)-h(x+1)$. 
It is straightforward to deduce from the above formula for $m_{W_0r}$ that:
\begin{align*}
M&:=\sum_{x\geq 0} \delta(H^{m_-}_{u_-})(x)\delta(H^{m_+}_{u_+})(x)-H^{m_-}_{u_-}(m_+)-H^{m_+}_{u_+}(m_-)\\
&=\sum_{x\geq 1} \delta(H^{m_-}_{u_-})(x)\delta(H^{m_+}_{u_+})(x)+
\delta(H^{m_-}_{u_-})(0)\delta(H^{m_+}_{u_+})(0)-H^{m_-}_{u_-}(m_+)-H^{m_+}_{u_+}(m_-)\\
&=\sum_{x\geq 1} \delta(H^{m_-}_{u_-})(x)\delta(H^{m_+}_{u_+})(x)+
(J^{m_-}_{u_-}(0)+\delta_{m_-,0}-1)(J^{m_+}_{u_+}(0)+\delta_{m_+,0}-1)\\
&\qquad-H^{m_-}_{u_-}(m_+)-H^{m_+}_{u_+}(m_-)\\
\end{align*}
where $J^m_u(0)=1$ if $1$ is a part of $u$ (equivalently,  if $0$ is a jump of $\xi$), and $J^m_u(0)=0$ else
(this value depends only on $u$ (is independent of $m$)).
Recall that (\cite{Slooten}, or \cite[Proposition 6.6]{OpdSol2}) the number of jumps of the vector of contents 
$\xi(m)$ of $T_m(\pi)$ equals $\#(u)$, and that this is also equal to $m+H^m_u(0)$. In the second equality 
above we used that $\delta(H^{m}_{u})(0)=2h^m_u(0)-h^m_u(1)=J^m_u(0)+\delta_{m,0}-1$.

Now let $\delta_\pm\in\{0,1\}$ be such that $\delta_\pm\equiv m_\pm\textup{mod}(2)$. 
There exist partitions $\pi_\pm'$ such that set of jumps of the vector $\xi'_\pm$ 
of contents of the $\delta_\pm$-tableau $T_{\delta_\pm}(\pi_\pm')$ of $\pi_\pm'$ equals 
the set of jumps of $\xi_\pm$ (cf. \cite[Proposition 6.6]{OpdSol2}). 
By Proposition \ref{prop:packet}, the central character $W_0'r'$ 
of $\textup{C}_{n}(\delta_-,\delta_+)[q]$ (with $2n=|u_-|+|u_+|-\delta_--\delta_+$) which corresponds 
to $W_0r$ under the translation STM $\textup{C}_{d}(m_-,m_+)[q]\leadsto \textup{C}_{n}(\delta_-,\delta_+)[q]$, 
is of the form $r'=(-\exp(\xi'_-),\exp(\xi'_+))$. Let $h^{\delta_\pm}_{u_\pm}(x)$ denote the
multiplicity of $x$ in the vector $\xi_\pm'$, and let $H^{\delta_\pm}_{u_\pm}(x)$ be define 
similar to $H_{u_\pm}^{m_\pm}(x)$. 
We define $\Delta_\pm^{m_\pm}:=H^{\delta_\pm}_{u_\pm}(x)-H_{u_\pm}^{m_\pm}(x)$. 
Then it follows 
from the definition of the jump vector at $m_\pm$ and at $\delta_\pm$ that for $x\geq 1$, 
$\Delta_{\pm}^{m_\pm}(x)=\max(0,m-x)$. Thus for $x\geq 1$, we have
$\delta(\Delta_{\pm}^{m_\pm})(x)=\chi_{[1,m_\pm-1]}(x)$, where $\chi_{[1,m_\pm-1]}$ 
denotes the indicator function of the interval $[1,m_\pm-1]$. Let $\#(u_-\cap u_+)$ denote 
the number of parts that $u_-$ and $u_+$ have in common.
Then we get:
\begin{align*}
M&:=\sum_{x\geq 1} \delta(H^{\delta_-}_{u_-})(x)\delta(H^{\delta_+}_{u_+})(x)
-H^{\delta_-}_{u_-}(1)+H^{\delta_-}_{u_-}(m_+)-H^{m_-}_{u_-}(m_+)\\
&-H^{\delta_+}_{u_+}(1)+H^{\delta_+}_{u_+}(m_-)-H^{m_+}_{u_+}(m_-)+
(J^{\delta_-}_{u_-}(0)+\delta_{m_-,0}-1)(J^{\delta_+}_{u_+}(0)+\delta_{m_+,0}-1)\\
\end{align*}
\begin{align*}
&\qquad -\delta_{m_+,0}(J^{\delta_-}_{u_-}(0)-\delta_-)-
\delta_{m_-,0}(J^{\delta_+}_{u_+}(0)-\delta_+)+\max(0,m_--1) \\
&=\#(u_-\cap u_+)-H^{\delta_-}_{u_-}(1)-H^{\delta_+}_{u_+}(1)
+\Delta_-^{m_-}(m_+)+\Delta_-^{m_+}(m_-)-J^{\delta_-}_{u_-}(0)-J^{\delta_+}_{u_+}(0)\\
&\qquad+\delta_{m_-,0}\delta_{m_+,0}
+\delta_{m_-,0}(\delta_+-1)+\delta_{m_+,0}(\delta_--1)+1+\max(0,m_--1)\\
&=\#(u_-\cap u_+)-H^{\delta_-}_{u_-}(0)-H^{\delta_+}_{u_+}(0)-\delta_--\delta_+
+\Delta_-^{m_+}(m_-)\\
&\qquad+\delta_{m_-,0}\delta_{m_+,0}
+\delta_{m_-,0}(\delta_+-1)+\delta_{m_+,0}(\delta_--1)+1+\max(0,m_--1)\\
&=-\#(u_-\cup u_+)
+\delta_{m_-,0}\Delta_+^{m_+}(0)+(1-\delta_{m_-,0}))\Delta_+^{m_+}(m_-)\\
&\qquad+\delta_{m_-,0}\delta_{m_+,0}
+\delta_{m_-,0}(\delta_+-1)+\delta_{m_+,0}(\delta_--1)+1+\max(0,m_--1)\\
&=-\#(u_-\cup u_+)
+\delta_{m_-,0}(m_+-\delta_+)+(1-\delta_{m_-,0}))(m_+-m_-)\\
&\qquad+\delta_{m_-,0}\delta_{m_+,0}
+\delta_{m_-,0}(\delta_+-1)+\delta_{m_+,0}(\delta_--1)+1+\max(0,m_--1)\\
&=-\#(u_-\cup u_+)+m_++\delta_{m_-,0}\delta_{m_+,0}+\delta_{m_+,0}(\delta_--1)\\
&=-\#(u_-\cup u_+)+m_+\\
\end{align*}
finishing the proof. In the above computation we used at several steps that $0\leq m_-\leq m_+$, 
and that $H^m_u(0)=\#(u)-m$.
\end{proof}
A similar but easier computation shows a similar result for Hecke algebras of unipotent representations of 
$\textup{SO}_{2n+1}$ (cf. \ref{subsub:II}):
\begin{thm}\label{thm:ratSOodd}
Let $m=(m_-,m_+)\in(\frac{1}{2}+\mathbb{Z})^2$ such that $0< m_-\leq m_+$.
Consider $\pi_{(u_-,u_+),(\sigma_-,\sigma_+)}:=\delta_{(\pi_-,\pi_+)}(0,0)$ as a discrete series of the 
normalized affine Hecke algebra $(\mc{H}, \tau)$ of type $\textup{C}_d(m_-,m_+)[q]$, 
normalized by $\tau(1)=1$. Let $(u_-,u_+)$ be the pair of distinguished unipotent 
partitions of type 
$m=(m_-,m_+)$ associated with pair $(T_{m_-}(\pi_-),T_{m_+}(\pi_+))$ of $m$-tableaux
(i.e. $u_\pm$ is a partition of $2n_\pm$ with distinct, even parts of length at least 
$m_\pm-\frac{1}{2}$, such that $n_-+n_+=n$). Then 
\begin{equation}
\textup{fdeg}_{\mathbb{Q}}(\pi_{(u_-,u_+),(\sigma_-,\sigma_+)})=2^{m_+-\frac{1}{2}-\#(u_-\cup u_+)}
\end{equation}
\end{thm} 
The proof of the next result (of \cite{FO}) is similar in spirit as the above results.
\begin{thm}[\cite{FO}]\label{thm:ratfactextra} Let $d=d_-+d_+\in\mathbb{Z}_{\geq 0}$, 
and let $\pi_\pm\vdash d_\pm$.
Let $0\leq m_-\leq m_+$ with $m_\pm\in \pm\frac{1}{4}+\mathbb{Z}$. Write 
$m_\pm=\kappa_\pm+\frac{1}{4}(2\epsilon_\pm-1)$ with $\kappa_\pm\in\mathbb{Z}_{\geq 0}$
and $\epsilon_\pm\in \{0,\,1\}$. Write $\delta_\pm \in\{0,\,1\}$ be defined by 
$\delta_\pm\equiv \kappa_\pm\ (\textup{mod}\ 2)$.
Consider $\pi_{(\pi_-,\pi_+),extra}:=\delta_{(\pi_-,\pi_+)}(0,0)$ as a discrete series of the 
normalized affine Hecke algebra $(\mc{H}, \tau)$ of type $\textup{C}_d(m_-,m_+)[q^2]$, 
normalized by $\tau(1)=1$. Let $(u_-,u_+)$ be the pair of unipotent partitions of type 
$(\delta_-,\delta_+)$ associated with pair $(T_{m_-}(\pi_-),T_{m_+}(\pi_+))$ of $m$-tableaux via the 
extraspecial STM (cf. (\ref{eq:extraspSTM}), and \cite{FO}) $\mc{H}\leadsto \textup{C}_n(\delta_-,\delta_+)[q]$. 
Then we have:
\begin{equation}
\textup{fdeg}_{\mathbb{Q}}(\pi_{(\pi_-,\pi_+),extra})=
\begin{cases}
2^{\#(u_-\cap u_+)-h_-(\frac{1}{4})-h_+(\frac{1}{4})} & \text{if } \epsilon_-\not=\epsilon_+ \\
2^{\#(u_-\cap u_+)-h_-(\frac{1}{4})-h_+(\frac{1}{4})-\kappa_-} & \text{if } \epsilon_-=\epsilon_+ \\
\end{cases}
\end{equation}
\end{thm}
Let us now look at the proof of Theorem \ref{thm:HII} for these cases:
\begin{lem}
Theorem \ref{thm:HII} holds for $G=\textup{PCSp}_{2n}$ (with $n\geq 2$),   
$\textup{P}(\textup{CO}^0_{2n})$ (with $n\geq 4$)  
or $\textup{P}((\textup{CO}^*_{2n})^0)$ (with $n\geq 4$).
\end{lem}
\begin{proof}
Assume that we have fixed a Borel subgroup $B\subset G$, a maximal torus 
$T\subset B$ and a pinning for the reductive groups $G$ considered below.
 
For $G=\textup{PCSp}_{2n}$ we have $\Omega=\{\epsilon,\eta\}\approx C_2$, hence we need to 
consider two inner 
forms $G^\epsilon$ and $G^\eta$. We first deal with the split form $G^\epsilon$.
We have $\mc{H}^{IM}(G^\epsilon)$ of type $\textup{B}_n(1,1)[q]$ (also 
denoted by $\mc{H}(\mc{R}^\textup{B}_{ad},m^B)$ in \cite[7.1.4]{Opd4}). 
The conjugacy classes of parahoric subgroups of $G^\epsilon$ which carry a (unique) cuspidal unipotent representation 
correspond 
to unordered pairs $(a,b)$ with $a, b\in\mathbb{Z}_{\geq0}$ such that $d:=n-a^2-b^2-a-b\geq 0$.
The corresponding type $\mf{s}_{d,a,b}$ corresponds to a subdiagram of type 
$\textup{B}_{a^2+a}\sqcup \textup{B}_{b^2+b}$ of the affine diagram $\textup{C}_n^{(1)}$ of a set of affine simple roots of 
$G^\epsilon(k)$. Consider the corresponding 
associated normalized (extended) affine Hecke algebra $\mc{H}^{\epsilon,\tilde{\mf{s}},e}$. Put 
$m_-:=|a-b|$, and $m_+:=1+a+b$. Then 
\begin{equation*}
\mc{H}^{\epsilon,\tilde{\mf{s}},e}\simeq
\begin{cases}
   \textup{C}_d(m_-,m_+)[q]& \text{if } a\not= b \\
   \textup{B}_d(1,m_+)[q]      & \text{ else} 
  \end{cases}
\end{equation*}
and  
\begin{equation*}
\Omega^{\mf{s}}_1\simeq\begin{cases}
   1& \text{ if } a\not= b\text{ or } d>0 \\
   C_2     & \text{ else} 
  \end{cases}
\end{equation*}
Thus by Proposition 
\ref{prop:balanced} and \cite[Section 13.7]{C} the rational factor $\tau^{\epsilon,\mf{s},e}(1)_{\mathbb{Q}}$ of 
the trace $\tau^{\epsilon,\mf{s},e}$ of $\mc{H}^{\epsilon,\tilde{\mf{s}},e}$
is such that (since 
$\mc{H}^{\epsilon,\mf{s}}\simeq\mc{H}^{\epsilon,\tilde{\mf{s}},e}\otimes\mathbb{C}[\Omega^{\mf{s}}_1]$, cf. Corollary \ref{cor:equiv1}):
\begin{equation*}
\tau^{\epsilon,\mf{s},e}(1)_{\mathbb{Q}}=
\begin{cases}
   2^{-a-b}& \text{ if } a\not= b\text{ or } d>0 \\
   2^{-1-a-b}   & \text{ else} 
  \end{cases}
\end{equation*}
As was discussed in paragraph \ref{subsub:class}
(also see \cite[7.1.4]{Opd4}), there exists an STM $\mc{H}^{\epsilon,\tilde{\mf{s}},e}\leadsto \mc{H}^{d,a,b}$
corresponding to a strict algebra inclusion $\mc{H}^{d,a,b}\subset \mc{H}^{\epsilon,\tilde{\mf{s}},e}$, 
where $\mc{H}^{d,a,b}=C_d(m_-,m_+)[q]$ is an object of $\mf{C}_{class}^{\textup{III}}$. 
This inclusion satisfies 
\begin{equation*}
\begin{cases}
   \mc{H}^{d,a,b}= \mc{H}^{\epsilon,\tilde{\mf{s}},e}& \text{if } a\not= b\text{ or } d=0 \\
   \mc{H}^{d,a,b}\subset \mc{H}^{\epsilon,\tilde{\mf{s}},e}& \text{has index two, else} 
  \end{cases}
\end{equation*}
We define 
the trace $\tau^{d,a,b}$ of $\mc{H}^{d,a,b}$ by restriction of the trace $\tau^{\epsilon,\mf{s},e}(1)_{\mathbb{Q}}$ 
of $\mc{H}^{\epsilon,\tilde{\mf{s}},e}$, so we have:
\begin{equation*}
\tau^{d,a,b}(1)_{\mathbb{Q}}=
\begin{cases}
   2^{1-m_+}& \text{ if } a\not= b\text{ or } d>0 \\
   2^{-m_+}   & \text{ else }
  \end{cases}
\end{equation*}
Now we want to compute the rational factor of the formal degree of a unipotent discrete series 
representation $\pi$ in a block corresponding to the type $\mf{s}:=\mf{s}_{d,a,b}$. According to Lusztig's 
parameterization \cite{Lu4} we attach to $\pi$ an unramified Langlands parameter 
$\lambda$, and an irreducible representation $\alpha$ of the 
component group $A_\lambda$ such that the center ${}^LZ\subset A_\lambda$ 
acts trivially in this representation (since $u=1$ here). 
This is equivalent to $\alpha$ being a one dimensional representation, and   
we can parameterize such $\alpha$ by a pair of Lusztig-Shoji symbols 
$(\sigma_-,\sigma_+)$ for a pair $(u_-,u_+)$ of distinguished unipotent partitions 
for the parameter $(m_-,m_+)$, such that $|u_-|+|u_+|=2n+1$. 
We denote by $\pi_{\lambda,(\sigma_-,\sigma_+)}^G$ 
the corresponding irreducible discrete series representation of 
$\mc{H}^{\epsilon,\tilde{\mf{s}},e}$ (depending on the chosen isomorphism  
$\mc{H}^{\epsilon,\mf{s}}\simeq\mc{H}^{\epsilon,\tilde{\mf{s}},e}\otimes\mathbb{C}[\Omega^{\mf{s}}_1]$).
According to \cite{BHK}, the formal degree of $\pi$ is equal to the formal degree 
of  $\pi_{\lambda,(\sigma_-,\sigma_+)}^G$. As before, let 
$\textup{fdeg}_\mathbb{Q}(\pi_{\lambda,(\sigma_-,\sigma_+)}^G)$ denote the 
rational factor of $\textup{fdeg}(\pi_{\lambda,(\sigma_-,\sigma_+)}^G)$.
The irreducible discrete series representations of $\mc{H}^{d,a,b}$ 
with central character corresponding to $(u_-,u_+)$ are parameterized \cite{OpdSol2} 
by pairs of Slooten symbols $(\sigma_-,\sigma_+)$ associated to $(u_-,u_+)$ 
at parameter $(m_-,m_+)$. 
The discrete series of $\mc{H}^{d,a,b}$ corresponding to 
$(u_-,u_+),(\sigma_-,\sigma_+)$ 
was denoted by $\pi_{(u_-,u_+),(\sigma_-,\sigma_+)}$. 
By Remark \ref{rem:number} we easily check 
that we have a total of $\binom{l_-+l_+}{(l_-+l_+-1)/2}+\binom{l_-+l_+}{(l_-+l_+-5)/2}+\dots=2^{l_-+l_+-2}$
such discrete series representations, in accordance with the number of one dimensional representations 
of $A_\lambda$
(with is of type $2^{(l_-+l_+-3)+2}$, according to Proposition  \ref{prop:centra}).
By the above, combined with Theorem \ref{thm:ratSPSO}
we see that 
\begin{equation*}
\textup{fdeg}_{\mathbb{Q}}(\pi_{(u_-,u_+),(\sigma_-,\sigma_+)})=
\begin{cases}
   2^{1-\#(u_-\cup u_+)}& \text{ if } a\not= b\text{ or } d>0 \\
   2^{-\#(u_-\cup u_+)}  & \text{else} 
  \end{cases}
\end{equation*}
According to \cite[Paragraph 6.4]{OpdSol2} 
(also see  \cite[Proposition 6.6]{DeOp2}), and using the fact that (see \cite{CK}) the 
Slooten symbols and the Lusztig-Shoji symbols match, we see that upon restriction 
to  $\mc{H}^{d,a,b}$ there are the following possibilities:
\begin{equation*}
\pi_{\lambda,(\sigma_-,\sigma_+)}^G|_{\mc{H}^{d,a,b}}=
\begin{cases}
   \pi_{(u_-,u_+),(\sigma_-,\sigma_+)}& \text{ if } a\not= b\text{ or } d=0 \\
   \pi_{(u_-,u_+),(\sigma_-,\sigma_+)}& \text{ if } a= b,\,d>0 \text{ and } u_-=0\\
   \pi_{(u_-,u_+),(\sigma_-,\sigma_+)}\oplus\pi_{(u_-,u_+),(\sigma_-',\sigma_+)}   & \text{ if } a= b,\,d>0 \text{ and } u_-\not=0\\
  \end{cases}
\end{equation*}
Here $\sigma_-'$ is the symbol obtained from $\sigma_-$ by interchanging the
top and the bottom row. In the second case $d>0 \text{ and } u_-=0$, there 
are two irreducible discrete series representations of $\mc{H}^{IM}(G^\epsilon)$ which 
restrict to the same irreducible $\pi_{(u_-,u_+),(\sigma_-,\sigma_+)}$ (whose 
central characters form one $X^*_{un}(G^\epsilon)$-orbit).
Restriction of the spectral decomposition of $\tau^{\epsilon,\mf{s},e}$
to $\mc{H}^{d,a,b}$ shows  $\textup{fdeg}_{\mathbb{Q}}(\pi_{\lambda,(\sigma_-,\sigma_+)}^G)=
\frac{1}{2}\textup{fdeg}_{\mathbb{Q}}(\pi_{(u_-,u_+),(\sigma_-,\sigma_+)})$ 
in this case, while 
$\textup{fdeg}_{\mathbb{Q}}(\pi_{\lambda,(\sigma_-,\sigma_+)}^G)=
\textup{fdeg}_{\mathbb{Q}}(\pi_{(u_-,u_+),(\sigma_-,\sigma_+)})$ in the other two cases.
Hence we have, for all $d\geq 0$:
\begin{equation*}
\textup{fdeg}_{\mathbb{Q}}(\pi_{\lambda,(\sigma_-,\sigma_+)}^G)=
\begin{cases}
   2^{-\#(u_-\cup u_+)}& \text{ if } u_-=0\\
   2^{1-\#(u_-\cup u_+)}& \text{ if }  u_-\not=0\\
     \end{cases}
\end{equation*}
Hence, using Proposition  \ref{prop:centra}, (\ref{eq:quot})
and (\ref{eq:HIIconv}) we see that  Theorem \ref{thm:HII} follows for 
this case $G=\textup{PCSp}_{2n}$ and $u=\epsilon$, if we show that 
$|C_\lambda^{F_\epsilon}|=2^{\#(u_-\cap u_+)}$. Recall that 
$M_\lambda^0\simeq (\mathbb{C}^\times)^{\#(u_-\cap u_+)}$ 
(cf. \cite[Section 13.1]{C}), on 
which $F_\epsilon$ acts by $\textup{Ad}(s_0)$. Clearly $\textup{ad}(s_0)$ 
must act by $-1$ on $\mf{m}_\lambda=\textup{Lie}(M_\lambda^0)$, 
and so $F_\epsilon$ acts by $F_\epsilon(m)=m^{-1}$ on $M_\lambda^0$. 
The desired result follows for $u=\epsilon$.

Next, we need to check Theorem  \ref{thm:HII} for the contributions 
coming from the nontrivial inner form $G^\eta$ in this case. Now the cuspidal unipotent 
parahoric subgroups $\mathbb{P}_{s,t}^\eta$ are given by $\eta$-invariant subdiagrams of type 
$\textup{B}_{s^2+s}\cup \textup{B}_{s^2+s}\cup A_{\frac{1}{2}(t^2+t)-1}$ such that 
$d+1:=\frac{1}{2}(n-2(s^2+s)-\frac{1}{2}(t^2+t)+2)\in\mathbb{Z}_{>0}$. This corresponds to a type 
$\mathfrak{s}:=\mathfrak{s}^\eta_{d,s,t}$ for $G^\eta$ which is completely  determined by a pair of nonnegative integers 
$(s,t)$ satisfying the above inequality. The corresponding affine Hecke algebra $\mc{H}^{\eta,\mathfrak{s},e}$ is 
of type $C_d(m_-,m_+)[q^2]$, with $m_+=\frac{1}{4}(3+2t+4s)$ and  $m_-=\frac{1}{4}|1-2t+4s|$. We have 
$\Omega^{\mathfrak{s}}_1=C_2$ (always), and hence using Proposition 
\ref{prop:balanced} and \cite[Section 13.7]{C}, the rational factor $\tau^{\eta,\mf{s},e}(1)_{\mathbb{Q}}$ of 
$\tau^{\eta,\mf{s},e}(1)$ equals 
\begin{equation*}
\tau^{\eta,\mf{s},e}(1)_{\mathbb{Q}}=2^{-s-1}=
\begin{cases}
2^{-\frac{1}{2}(m_+ + m_-+1)} & \text{if } \epsilon_-\not=\epsilon_+\\
2^{-\frac{1}{2}(m_+ - m_-+1)} & \text{if } \epsilon_-=\epsilon_+\\
\end{cases}
\end{equation*}
Using Theorem \ref{thm:ratfactextra}, we obtain 
two discrete series representations $\pi_{(\pi_-,\pi_+),extra}^\pm$, with 
$\textup{fdeg}_{\mathbb{Q}}(\pi_{(\pi_-,\pi_+),extra}^\pm)= 2^{\frac{1}{2}(l_-+l_+-1)-\#(u_-\cup u_+)}$ (in all cases).
In view of Proposition \ref{prop:centra}, this is indeed the rational factor of the formal degree of the two 
elements of the Lusztig packet attached to the Langlands parameter $\lambda$ on which 
${}^LZ\subset A_\lambda$ acts by $\eta$ times the identity, as predicted by (\ref{eq:HIIconv}). 

For $G=\textup{P}(\textup{CO}^0_{2n})$ (with $n\geq 4$) we do a similar analysis. 
In this case, $\Omega$ is isomorphic to $C_4$  
if $n$ is odd, and isomorphic to $C_2\times C_2$ if $n$ is even.
Let $\theta$ denote a diagram automorphism of order two of the finite type $\textup{D}_n$ 
sub diagram.  Let us write $\Omega=\{\epsilon,\eta,\rho,\eta\rho\}$, where $\eta$ is $\theta$-invariant, 
and $[\rho,\theta]=[\rho\eta,\theta]=\eta$. 
Let us first consider the split case $G^\epsilon$.
In this case $\mc{H}^{IM}(G)$
is of type $\textup{D}_n[q]$, which was denoted by $\mc{H}(\mc{R}^\textup{D}_{ad},m^D)$ in \cite[(54)]{Opd4}.
Let us denote $\mc{H}(\mc{R}^\textup{D}_{\mathbb{Z}^n},m^D)$ (notation as in  \cite[Paragraph 7.1.4]{Opd4})
by $\tilde{\textup{D}}_n[q]$. Its spectral diagram consists of the Dynkin diagram for $\textup{D}_n^{(1)}$, 
with the action of the automorphism $\eta$ as in \cite[Figure 1]{DeOp2} (we have, 
in the sense of \cite[Definition 2.11]{Opd4}, that $\Omega_Y^\vee=\langle\eta\rangle\simeq C_2$). 
As was 
discussed in \cite[Paragraph 7.1.4]{Opd4}, we have spectral coverings $\textup{D}_n[q]\leadsto\tilde{\textup{D}}_n[q]$
and $\tilde{\textup{D}}_n[q]\leadsto \textup{C}_n(0,0)[q]$, corresponding to strict algebra embeddings
$\tilde{\textup{D}}_n[q]\subset \textup{D}_n[q]$ and  $\tilde{\textup{D}}_n[q]\subset \textup{C}_n(0,0)[q]$, both of index $2$.
We normalize the trace of $\tilde{\textup{D}}_n[q]$ by restriction from $\textup{D}_n[q]$, and 
of $\textup{C}_n(0,0)[q]$ such that its restriction to $\tilde{\textup{D}}_n[q]$ equals the trace we just defined 
on $\tilde{\textup{D}}_n[q]$. The conjugacy classes of parahoric subgroups of $G$ which support a (unique) cuspidal 
unipotent representation correspond to unordered pairs $(a,b)$ with $a,b\in2\mathbb{Z}_{\geq0}$ such that 
$d=n-a^2-b^2\geq 0$. The pair $(a,b)$ corresponds 
to a sub diagram of type $\textup{D}_{a^2}\sqcup \textup{D}_{b^2}$ of the type $\textup{D}_n^{(1)}$ 
diagram of a set of simple affine roots of $G(k)$. We put $m_-=|a-b|$, and $m_+=|a+b|$.
We have 
\begin{equation*}
\mc{H}^{\epsilon,\tilde{\mf{s}},e}\simeq
\begin{cases}
   \textup{C}_d(m_-,m_+)[q]& \text{ if } a\not= b \text{ or } d=0\\
   \textup{B}_d(1,m_+)[q]      & \text{ if } a=b>0 \text{ and } d>0\\
   \textup{D}_n[q] & \text{ if } a=b=0\\
  \end{cases}
\end{equation*}
and  
\begin{equation*}
\Omega^{\mf{s}}_1\simeq\begin{cases}
   C_2& \text{ if } a>0,\,b>0 \text{ and }a\not=b\text{ or } d>0 \\
   C_2\times C_2  &\text{ if } a= b,\,d=0,\text{ and } n\in 2\mathbb{Z}\\
   C_4 &\text{ if } a= b,\,d=0,\text{ and } n\in 2\mathbb{Z}+1\\
   1     & \text{ else} 
  \end{cases}
\end{equation*}
As before we denote by $\mc{H}^{d,a,b}$ the type $\mf{C}_{class}^{\textup{IV}}$-
object $\mc{H}^{d,a,b}\simeq \textup{C}_d(m_-,m_+)[q]$ 
which is covered by $\mc{H}^{\epsilon,\tilde{\mf{s}},e}$. For $m_-=m_+=0$ we also introduce 
$\tilde{\mc{H}}^{n,0,0}\simeq \tilde{\textup{D}}_n[q]$. Then we have 
\begin{equation*}
\begin{cases}
   \mc{H}^{d,a,b}= \mc{H}^{\epsilon,\tilde{\mf{s}},e}& \text{if } a\not= b\text{ or } d=0 \\
   \mc{H}^{d,a,b}\subset \mc{H}^{\epsilon,\tilde{\mf{s}},e}& \text{has index two if }  a=b>0 \text{ and } d>0\\
   \mc{H}^{n,0,0}\supset \tilde{\mc{H}}^{n,0,0}\subset \mc{H}^{\epsilon,\tilde{\mf{s}},e}& \text{ if } a=b=0
   \text{ (both inclusions have index two) }\\
  \end{cases}
\end{equation*}
We have, by definition of our normalizations, and using  
Proposition 
\ref{prop:balanced} and \cite[Section 13.7]{C}, 
\begin{equation*}
\tau^{\epsilon,\mf{s},e}(1)_{\mathbb{Q}}=\tau^{d,a,b}(1)_{\mathbb{Q}}=
\begin{cases}
   2^{-m_+}& \text{ if } a=b\text{ and } d=0,\,\text{ or if } a=b=0\\
   2^{1-m_+}   & \text{ else} 
  \end{cases}
\end{equation*}
where as before, $\tau^{d,a,b}$ denotes the trace of the type $\textup{C}_d(m_-,m_+)[q]$-algebra 
(an object of $\mf{C}_{class}^{\textup{IV}}$) 
which is spectrally covered by $\mc{H}^{\epsilon,\tilde{\mf{s}},e}$. 

The irreducible discrete series representations of $\mc{H}^{d,a,b}$ 
with central character corresponding to $(u_-,u_+)$ are parameterized by 
pairs of Slooten symbols $(\sigma_-,\sigma_+)$, denoted by 
$\pi_{(u_-,u_+),(\sigma_-,\sigma_+)}$. 
By the above, combined with Theorem \ref{thm:ratSPSO}
we see that 
\begin{equation*}
\textup{fdeg}_{\mathbb{Q}}(\pi_{(u_-,u_+),(\sigma_-,\sigma_+)})=
\begin{cases}
   2^{-\#(u_-\cup u_+)}&\text{ if } a=b\text{ and } d=0,\,\text{ or if } a=b=0\\
   2^{1-\#(u_-\cup u_+)}  & \text{ else} \\
  \end{cases}
  \end{equation*}
 As in the previous case $G=\textup{PCSp}_{2n}$,  Proposition  \ref{prop:centra}, (\ref{eq:quot})
and (\ref{eq:HIIconv}) imply that Theorem  \ref{thm:HII} is true in this case iff (here $\lambda$ denotes 
a discrete unramified Langlands parameter for $G$ which gives rise to the pair $(u_-,u_+)$ as in 
Proposition \ref{prop:centra}):
 \begin{equation}\label{eq:PCO1}
\textup{fdeg}_{\mathbb{Q}}(\pi_{\lambda,(\sigma_-,\sigma_+)}^G)=
\begin{cases}
   2^{-\#(u_-\cup u_+)}& \text{ if } u_-=0\\
   2^{1-\#(u_-\cup u_+)}& \text{ if }  u_-\not=0\\
     \end{cases}
\end{equation}
In the case $d=0$ we have $\textup{fdeg}_{\mathbb{Q}}(\pi^G_{\lambda,(\sigma_-,\sigma_+)})
=\textup{fdeg}_{\mathbb{Q}}(\pi_{(u_-,u_+),(\sigma_-,\sigma_+)})$, and since $a=b$ is equivalent 
to $u_-=0$ in this case, we are done if $d=0$. Similarly, if $a\not= b$ (hence $u_-\not=0$) 
there is no branching, 
and we are done. 
So from now on, we may and will assume $d>0$ and $a=b$.
The case $a=b>0$ is completely analogous to what we did in the case $G=\textup{PCSp}_{2n}$.
This leaves the case $a=b=0$. We combine results of \cite[Appendix]{RamRam}, \cite[Lemma 6.10]{DeOp2}
and \cite[Section 8]{OpdSol2}
to derive the branching behavior of the discrete series.
If $u_-=0$ then there exist two distinct discrete series 
representations $\pi^G_{\lambda_+,(0,\sigma_+)}$ and $\pi^G_{\lambda_-,(0,\sigma_+)}$
of $\mc{H}^{\epsilon,\tilde{\mf{s}},e}=\mc{H}^{IM}(G^\epsilon)$ 
whose central characters are distinct (but lie in the same $X^*_{un}(G^F)$-orbit), and which 
restrict to the same irreducible discrete series representation $\tilde{\pi}^G_{\lambda,(0,\sigma_+)}$
of $\tilde{\mc{H}}^{n,0,0}$. On the other hand, there also exist two irreducible discrete series characters 
$\pi_{(0,u_+),(0,\sigma_+)}$ and $\pi_{(0,u_+),(0,\sigma_+')}$ of $\mc{H}^{n,0,0}$ 
which both restrict to  $\tilde{\pi}^G_{\lambda,(0,\sigma_+)}$. It follows easily that 
$\textup{fdeg}_{\mathbb{Q}}(\pi^G_{\lambda_\pm,(0,\sigma_+)})
=\textup{fdeg}_{\mathbb{Q}}(\pi_{(0,u_+),(0,\sigma_+)})=2^{-\#(u_+)}$ as desired.

If $u_-\not=0$, and $\lambda$ is an unramified discrete Langlands parameter 
for $G$ corresponding to $(u_-,u_+)$, then $\pi^G_{\lambda,(\sigma_-,\sigma_+)}$ 
restricts to a direct sum 
$\tilde{\pi}^G_{\lambda,(\sigma_-,\sigma_+, +1)}\oplus\tilde{\pi}^G_{\lambda,(\sigma_-,\sigma_+,-1)}$
of irreducible discrete series representations of $\tilde{\mc{H}}^{n,0,0}$. Indeed, by \cite[A.13]{RamRam}
the restriction is either irreducible or a direct sum of two irreducibles, which are moreover themselves 
discrete series by  \cite[Lemma 6.3]{DeOp2}. Moreover it follows from  \cite[A.13]{RamRam}
that if  there exists a $\pi^G_{\lambda,(\sigma_-,\sigma_+)}$ with $\sigma_-\not=0$ and 
$\sigma_+\not=0$ which restricts to an irreducible in this way, then the number of 
irreducible discrete series representations of $\tilde{\mc{H}}^{n,0,0}$ with $u_-$ and $u_+$ not 
equal to $0$ is strictly less than twice the number of irreducible 
discrete series of the kind described above of $\mc{H}^{IM}(G^\epsilon)$. But this contradicts the 
classification of the discrete series as in \cite[Section 8]{OpdSol2} (this counting argument 
is similar to the proof of \cite[Lemma 6.10]{DeOp2}).
There are four irreducible discrete series characters
$\pi_{(u_-,u_+),(\sigma_-,\sigma_+)}$, $\pi_{(u_-,u_+),(\sigma_-',\sigma_+)}$, $\pi_{(u_-,u_+),(\sigma_-,\sigma_+')}$ and 
$\pi_{(u_-,u_+),(\sigma_-',\sigma_+')}$ of $\mc{H}^{n,0,0}$, and it is easy to see that all of  
these restrict to irreducible discrete series characters of  $\tilde{\mc{H}}^{n,0,0}$: Two of them will 
restrict to  $\tilde{\pi}^G_{\lambda,(\sigma_-,\sigma_+, +1)}$, and the other two will restrict to 
$\tilde{\pi}^G_{\lambda,(\sigma_-,\sigma_+, -1)}$. Altogether it follows that 
$\textup{fdeg}_{\mathbb{Q}}(\pi^G_{\lambda_\pm,(\sigma_-,\sigma_+)})
=2\textup{fdeg}_{\mathbb{Q}}(\pi_{(u_-,u_+),(\sigma_-,\sigma_+)})=2^{1-\#(u_-\cup u_+)}$ in these cases, as desired.
Using Remark \ref{rem:number} again, we see that the total number of this kind of unipotent discrete series representations 
equals 
$2^{l_-+l_+-3}$ if $l_-\not=0$, and $2^{l_+-2}$ otherwise. This should correspond to 
the subset of the Lusztig packet associated to $\lambda$ 
which is parameterized by the set $\textup{Irr}_{\epsilon}(A_\lambda)$ of irreducible characters of $A_\lambda$ on which 
${}^LZ=\Omega^*$ acts trivially. Indeed, this is half the number of one-dimensional irreducibles of $A_\lambda$.

Next, let us take the inner form $G^u$ with $u=\eta$.  The analysis is exactly the same as for $u=\epsilon$, except that 
now $a$ and $b$ are both odd. We again obtain  $2^{l_-+l_+-3}$ (if $l_-\not=0$) or  $2^{l_+-2}$ (otherwise) 
unipotent discrete series representations in the Lusztig packet for $\lambda$, this times the ones 
parameterized by the set of irreducible characters $\textup{Irr}_{\eta}(A_\lambda)$ of $A_\lambda$ on which 
${}^LZ=\Omega^*$ acts as a multiple of $\eta$. The collection 
$\textup{Irr}_{\epsilon}(A_\lambda)\cup\textup{Irr}_{\eta}(A_\lambda)$ coincides with the collection of 
$2^{l_-+l_+-2}$ (if $l_-\not=0$) (or  $2^{l_+-1}$ if $l_-=0$) one-dimensional irreducible representations of $A_\lambda$.  

Finally consider the inner forms with $u=\rho$ or $u=\rho\eta$. These two inner forms are equivalent as rational forms, 
via the outer automorphism corresponding to $\theta$, hence it suffices to consider the case $u=\rho$ only. 
This time the cuspidal unipotent 
parahoric subgroups $\mathbb{P}_{s,t}^\rho$ are given by $\rho$-invariant subdiagrams of type 
\begin{equation*}
\mathbb{P}_{s,t}^\rho\simeq
\begin{cases}
\textup{D}_{s^2}\cup \textup{D}_{s^2}\cup{}^2\textup{A}_{\frac{1}{2}(t^2+t)-1}& \text{ if } n \text{ even } \\
{}^2\textup{D}_{s^2}\cup {}^2\textup{D}_{s^2}\cup {}^2\textup{A}_{\frac{1}{2}(t^2+t)-1} & \text{ if } n \text{ odd } \\
\end{cases}
\end{equation*}
such that $d+1:=\frac{1}{2}(n-2s^2-\frac{1}{2}(t^2+t)+2)\in\mathbb{Z}_{>0}$. This corresponds to a type 
$\mathfrak{s}:=\mathfrak{s}^\rho_{d,s,t}$ for $G^\rho$ which is completely  determined by a pair of nonnegative integers 
$(s,t)$ satisfying the above inequality, and the congruences: $s\equiv n (\textup{mod}\  2)$, 
$t\equiv  0,3(\textup{mod}\  4)$ (if $n$ even), and $t\equiv  1,2(\textup{mod}\ 4)$ (if $n$ odd) . 
The corresponding affine Hecke algebra $\mc{H}^{\rho,\tilde{\mathfrak{s}},e}$ is 
of always type $C_d(m_-,m_+)[q^2]$, with $m_+=\frac{1}{4}(1+2t+4s)$ and  $m_-=\frac{1}{4}|1+2t-4s|$. We have 
\begin{equation}\label{eq:omega}
\Omega^{\mathfrak{s}}_1=
\begin{cases}
\Omega& \text{ if } s>0 \text{ or } d=0 \\
 \langle \rho\rangle\simeq C_2 &\text{ if } s=0 \text{ and } d>0 \\
\end{cases}
\end{equation}
Using \ref{prop:balanced} and \cite[Section 13.7]{C}, the rational factor $\tau^{\rho,\mf{s},e}(1)_{\mathbb{Q}}$ of 
$\tau^{\rho,\mf{s},e}(1)$ equals 
\begin{equation*}
\tau^{\rho,\mf{s},e}(1)_{\mathbb{Q}}=2^{-s-1}=
\begin{cases}
2^{-\frac{1}{2}(m_+ + m_-+2)} & \text{if } \epsilon_-\not=\epsilon_+\\
2^{-\frac{1}{2}(m_+ - m_-+2)} & \text{if } \epsilon_-=\epsilon_+\\
\end{cases}
\end{equation*}
Using Theorem \ref{thm:ratfactextra}, we obtain 
discrete series representations $\pi_{(\pi_-,\pi_+),extra}^\alpha$ with rational parts of formal degrees equals 
$\textup{fdeg}_{\mathbb{Q}}(\pi_{(\pi_-,\pi_+),extra}^\alpha)= 2^{\frac{1}{2}(l_-+l_+-2)-\#(u_-\cup u_+)}$ (in all cases), 
where $\alpha$ denotes an irreducible character of $\Omega^{\mathfrak{s}}_1$. 
In view of Proposition \ref{prop:centra} and (\ref{eq:PCO1}), 
this is indeed the rational factor of the formal degree of the two 
elements of the Lusztig packet attached to the Langlands parameter $\lambda$ on which 
${}^LZ\subset A_\lambda$ acts by $\rho$ times the identity, as predicted by (\ref{eq:HIIconv}). 
As to the numerology of counting the number of such irreducible representations in a Lusztig packet 
attached to a unipotent discrete Langlands parameter $\lambda$ for $G$: Let us write $(u_-,u_+)$ 
for the (ordered) pair of unipotent partitions  attached to $\lambda$ 
(these are partitions with odd, distinct parts such that $|u_-|+|u_+|=2n$).
If $u_-\not=0$ and $u_-\not= u_+$ then we have two such packets (for the pairs  
$(u_-, u_+)$ and $(u_+,u_-)$ which contain discrete series representation with the same  
$q$-rational factor. According to Proposition \ref{prop:centra} both these packets contain
$2$ irreducibles on which ${}^LZ$ acts as a multiple of $\rho$ (and also two where 
${}^LZ$ acts as multiple of $\rho\eta$) (together these are the four irreducibles in each of 
these packets which are not one-dimensional). This matches the ``Hecke side", since we 
have (by (\ref{eq:omega})) that $\mc{H}^{\rho,\mathfrak{s}}$ is either a direct sum of 
four copies of $\mc{H}^{\rho,\tilde{\mathfrak{s}},e}$, each contributing one 
irreducible discrete series with the desired $q$-rational factor in the formal degree 
(if $s\not=0$, or equivalently $m_-\not=m_+$) 
or of two such copies (if $s=0$, or equivalently $m_-=m_+$). But in the latter case,  each of 
these copies of $\mc{H}^{\rho,\tilde{\mathfrak{s}},e}$ contributes two such irreducible discrete series 
(whose central characters are mapped by the STM to $(u_-,u_+)$ and $(u_+,u_-)$ respectively).
If $u_-=u_+$ then necessarily $m_-=m_+$, and the two 
copies of $\mc{H}^{\rho,\tilde{\mathfrak{s}},e}$ contribute each one discrete series to the packet 
associated to $\lambda$, corresponding to the two irreducibles of $A_\lambda$ on which ${}^LZ$
acts as $\rho$. Finally we have the case $u_-=0$. In this case there are four distinct discrete 
Langlands parameters $\lambda_1=\lambda,\,\lambda_2,\,\lambda_3,\,\lambda_4$ which 
share the same $q$-rational factor in the formal degree, and each of the four corresponding Lusztig 
packets should have one member associated to the single irreducible of $A_{\lambda_i}$ on 
which ${}^LZ$ acts by $\rho$ (according to Proposition \ref{prop:centra}). Hence in all cases the 
Hecke algebra side and the L-packet side indeed match. This finishes the case 
$G=\textup{P}(\textup{CO}^0_{2n})$.

The last case to consider is the non-split quasi-split orthogonal group $\textup{P}((\textup{CO}^*)^0_{2n+2})$. 
Now we have 
$u\in \Omega/(1-\theta)\Omega=\Omega/\langle\eta\rangle\simeq\langle\overline{\rho}\rangle\simeq C_2$.
We have $\mc{H}^{IM}(G)=\textup{C}_{n}(1,1)[q]$.
The conjugacy classes of parahoric subgroups $\mathbb{P}^{d,a,b}$ which support a (unique) cuspidal unipotent 
representation are parametrized by ordered pairs  $(a,b)$ with $a,b\in\mathbb{Z}_{\geq0}$, with $a$ 
even and $b$ odd, and such that $d=n+1-a^2-b^2\geq0$. The parahoric 
$\mathbb{P}^{d,a,b}$ is of type $\textup{D}_{a^2}\cup {}^2\textup{D}_{b^2}$.
The corresponding cuspidal unipotent type is denoted by $\mf{s}=\mf{s}^{d,a,b}$. We have 
$\mc{H}^{\epsilon,\tilde{\mf{s}},e}=\textup{C}_d(m_-,m_+)[q]$, with $m_+=a+b$ and $m_-=|a-b|$.
Furthermore, $\Omega_1^{\epsilon,\mf{s},\theta}=C_2$ (if $a>0$ or $d=0$)  
or $=1$ (if $a=0$ and $d>0$), implying that $\tau^{\epsilon,\mf{s},e}(1)_\mathbb{Q}=2^{1-m_+}$
(in all cases). 

Let $\lambda$ be a discrete unramified Langlands parameter for $G$. 
According to \cite{Reed}, in the notation of (\ref{eq:qsLP}), we have 
$C_{G^\vee}(\lambda(\textup{Frob}\times \operatorname{id}))$ is the connected 
cover in $\textup{Spin}_{2n+2}$ of  
$\textup{SO}_{2n_-+1}\times\textup{SO}_{2n_++1}$
(with $n_-+n_+=n$), and the $G^\vee$-orbits of such $\lambda$ correspond bijectively to ordered pairs 
$(u_-,u_+)$ where $u_\pm$ is a distinguished unipotent class in $\textup{SO}_{2n_\pm+1}$.
Note that this means that $u_\pm\vdash 2n_\pm+1$ has odd, distinct parts. 

Let $(\sigma_-,\sigma_+)$ be a Slooten symbol for the parameters $(m_-,m_+)$ 
corresponding to the pair $(\lambda_-,\lambda_+)$, and let $\pi_{(u_-,u_+),(\sigma_-,\sigma_+)}$  
be the correspond discrete series representation of $\mc{H}^{\epsilon,\tilde{\mf{s}},e}$.
Then, Theorem \ref{thm:ratSPSO} 
implies that $\textup{fdeg}_{\mathbb{Q}}(\pi_{(u_-,u_+),(\sigma_-,\sigma_+)})=2^{1-\#(u_-\cup u_+)}$.
It easily follows that this agrees with (\ref{eq:HIIconv}). The number of such irreducible discrete series 
equals $2^{l_-+l_+-2}$, in as expected by Proposition \ref{prop:centra}.

Let us now consider $u=\overline{\rho}$. Now the cuspidal unipotent 
parahoric subgroups $\mathbb{P}_{s,t}^{\overline{\rho}}$ are given by ${\overline{\rho}}$-invariant subdiagrams of type 
\begin{equation*}
\mathbb{P}_{s,t}^\rho\simeq
\begin{cases}
\textup{D}_{s^2}\cup \textup{D}_{s^2}\cup {}^2\textup{A}_{\frac{1}{2}(t^2+t)-1}& \text{ if } n \text{ even } \\
{}^2\textup{D}_{s^2}\cup {}^2\textup{D}_{s^2}\cup {}^2\textup{A}_{\frac{1}{2}(t^2+t)-1} & \text{ if } n \text{ odd } \\
\end{cases}
\end{equation*}
such that $d+1:=\frac{1}{2}(n-2s^2-\frac{1}{2}(t^2+t)+3)\in\mathbb{Z}_{>0}$. This corresponds to a type 
$\mathfrak{s}:=\mathfrak{s}^\rho_{d,s,t}$ for $G^\rho$ which is completely  determined by a pair of nonnegative integers 
$(s,t)$ satisfying the above inequality, and the congruences: $s\equiv n (\textup{mod}\  2)$, 
$t\equiv  1,2(\textup{mod}\  4)$ (if $n$ even), and $t\equiv  0,3(\textup{mod}\ 4)$ (if $n$ odd) . 
The corresponding affine Hecke algebra $\mc{H}^{\rho,\tilde{\mathfrak{s}},e}$ is 
of always type $C_d(m_-,m_+)[q^2]$, with $m_+=\frac{1}{4}(1+2t+4s)$ and  $m_-=\frac{1}{4}|1+2t-4s|$. 
We have 
\begin{equation*}
\Omega^{\mathfrak{s},\theta}_1=
\begin{cases}
\langle\eta\rangle\simeq C_2& \text{ if } s>0 \text{ or } d=0 \\
 1 &\text{ if } s=0 \text{ and } d>0 \\
\end{cases}
\end{equation*}
and we get 
\begin{equation*}
\tau^{\rho,\mf{s},e}(1)_{\mathbb{Q}}=2^{-s}=
\begin{cases}
2^{-\frac{1}{2}(m_+ + m_-)} & \text{if } \epsilon_-\not=\epsilon_+\\
2^{-\frac{1}{2}(m_+ - m_-)} & \text{if } \epsilon_-=\epsilon_+\\
\end{cases}
\end{equation*}
Hence, using Theorem \ref{thm:ratfactextra}, 
the extra special STM $\mc{H}^{\rho,\tilde{\mathfrak{s}},e}\leadsto \mc{H}^{IM}(G)$ 
yields one additional discrete series representation $\pi_{(\pi_-,\pi_+),extra}$
added to the Lusztig packet associated to $\lambda$, whose formal degree satisfies 
$\textup{fdeg}_{\mathbb{Q}}(\pi_{(\pi_-,\pi_+),extra})= 2^{\frac{1}{2}(l_-+l_+)-\#(u_-\cup u_+)}$, 
as desired in view of Proposition \ref{prop:centra}. 
\end{proof}
\subsubsection{Unipotent representations for inner forms of  $\textup{SO}_{2n+1}$}\label{subsub:II}
In these cases, a unipotent affine Hecke algebra is always spectrally isomorphic to a direct sum of finitely many 
copies of objects of $\mf{C}_{class}^{\textup{II}}$.  The treatment of these cases 
is analogous to the symplectic and even orthogonal cases discussed in the previous paragraph, 
but in all aspects much simpler (no branching phenomena, no extraspecial STM's). 
We will content ourselves to give the results only.

We have $\Omega=\{\epsilon,\eta\}\simeq C_2$, and $\mc{H}^{IM}(G^\epsilon)$ is of type $\textup{C}_n(\frac{1}{2},\frac{1}{2})[q]$.
The conjugacy classes of parahoric subgroups 
 $\mathbb{P}^{d,a,b}$ of $G$ supporting a (unique) cuspidal unipotent 
representation are parametrized by ordered pairs  $(a,b)$ with $a,b\in\mathbb{Z}_{\geq0}$, with $a$ 
even, and such that $d=n-a^2-(b^2+b)\geq0$. The parahoric 
$\mathbb{P}^{d,a,b}$ is a type $\textup{D}_{a^2}\cup \textup{B}_{b^2+b}$.
The corresponding cuspidal unipotent type is denoted by $\mf{s}=\mf{s}^{d,a,b}$. We have 
$\mc{H}^{\epsilon,\tilde{\mf{s}},e}=\textup{C}_d(m_-,m_+)[q]$, with $m_+=\frac{1}{2}+a+b$ and $m_-=|\frac{1}{2}-a+b|$.
Furthermore, $\Omega_1^{\epsilon,\mf{s},\theta}=C_2$ (if $a>0$ or $d=0$) or $=1$ (if $a=0$ and $d>0$), 
implying that $\tau^{\epsilon,\mf{s},e}(1)_\mathbb{Q}=2^{\frac{1}{2}-m_+}$ (in all cases). 
For the nontrivial inner form $G^\eta$ of $G$, the formulas are the same except that now $a$ is 
odd, and $\mathbb{P}^{d,a,b}$ has type ${}^2\textup{D}_{a^2}\cup \textup{B}_{b^2+b}$.

Now an orbit of discrete unipotent Langlands parameters $\lambda$ for $G$ corresponds to an ordered pair 
$(u_-,u_+)$ of unipotent partitions with $u_\pm\vdash 2n_\pm$ such that 
$n_-+n_+=n$, where $u_\pm$ consists of distinct, even parts. 

The discrete series representations of $\mc{H}^{\epsilon,\tilde{\mf{s}},e}=\textup{C}_d(m_-,m_+)[q]$ are 
parameterized by a pair of Slooten symbols $(\sigma_-,\sigma_+)$ for such pairs $(u_-,u_+)$, at the parameter pair 
$(m_-,m_+)$. The ordered pair $(\sigma_-,\sigma_+)$ corresponds to an ordered pair of partitions 
$(\pi_-,\pi_+)$ with $\pi_\pm\vdash n_\pm$. Let us denote this discrete series representation 
of $\textup{C}_d(m_-,m_+)[q]$ by $\delta_{(\pi_-,\pi_+)}$. 
By Theorem  \ref{thm:ratSOodd} we arrive at:
\begin{equation}
\textup{fdeg}_{\mathbb{Q}}(\pi_{(u_-,u_+),(\sigma_-,\sigma_+)})=2^{-\#(u_-\cup u_+)}
\end{equation}
It is easy to check that this matches (\ref{eq:HIIconv}) (see e.g. \cite[Corollary 6.1.6]{CM}).
\subsubsection{The example of type ${}^3\textup{D}_4$}\label{ex:ex2}
Let $G$ be the group of type ${}^3\textup{D}_4$ defined over a nonarchimedean local field $k$.
The group $G$ is quasisplit, 
and the dual $L$-group is isomorphic to ${}^LG:=\langle\theta\rangle\ltimes G^\vee$ 
where $G^\vee=\textup{Spin}(8)$ and where 
$\theta$ is an outer automorphism of order $3$. 
Hence ${}^LZ=1$, and $G$ has no nontrival inner forms. 
There are two cuspidal unipotents called ${}^3\textup{D}_4[1]$ and ${}^3\textup{D}_4[-1]$ (cf. \cite{C}, 
section 13.7).

The image $\phi(F)=s\theta$ of the Frobenius 
element under a discrete unramified Langlands parameter $\phi$ 
is an isolated semisimple automorphism. Via the action of 
$\textup{Int}(G^\vee)$ it is conjugated to a semisimple class of the 
from  $\theta s_i$ with $s_i$ a vertex of $C_\theta$ (cf. \cite{GR}). 
In the case at hand, we label the nodes of the twisted affine root diagram 
according to \cite[Section 4.4]{GR}, and we have to consider $\theta s_0$, 
$\theta s_1$ and $\theta s_2$. 

We have $\mc{H}^{IM}(G)=\textup{G}_2(3,1)[q]$, normalized by 
$\tau(1):=[3]_q^{-1}(v-v^{-1})^{-2}$
according to (\ref{eq:deg}). The $W_0$-orbit space of the character torus $T$ of the 
root lattice $X$ of type $G_2$ can be identified \cite{Bo} with the space of 
$\textup{Int}(G^\vee)$-orbits of semisimple classes of ${}^LG$ of the form $\theta g$, 
via the map $T\ni t\to \theta t$. In this way we will identify, as usual,  
the space of central characters of affine Hecke algebra $\mc{H}^{IM}(G)=G_2(3,1)[q]$ and 
the space of semisimple $\textup{Int}(G^\vee)$-orbits of this form of ${}^LG$.
The Hecke algebra $\mc{H}^{IM}(G)$ has two orbits of real residual points 
$W_0r_{0,reg}$ and $W_0r_{0,sub}$, and two nonreal ones $W_0r_1$ and $W_0r_2$ 
(using the same numbering of the nodes of the diagram as before).
At each residual point of $G_2(m_l,m_s)[q]$ at the parameter value $(m_l,m_s)=(3,1)$,  
the number of irreducible 
discrete series characters 
supported at this point is equal to the number of generic residual points which 
specialize at $(3,1)$ to the given residual point. This number is always $1$, 
except for $W_0r_{0,sub}$, where it is equal to two (\cite{OpdSol2}). We can 
and will baptise these orbits of generic residual points $W_0r$, 
using Kazhdan-Lusztig parameters for the discrete series of $G_2(1,1)[q]$,   
by an irreducible representation of $A_\lambda$,  where 
$\lambda$ is the Langlands parameter of the split group of type $G_2$. 
The subregular unipotent orbit of $G_2$ gives rise to a unipotent discrete Langlands 
 parameter $\lambda=\lambda_{sub}$ of ${}^3\textup{D}_4$ with $A_\lambda=S_3$. 
 Its ``weighted Dynkin diagram"  is $r_{0,sub}$. 
 The two orbits of generic residual points of the generic Hecke algebra of type $G_2$ 
 which are confluent at $(1,1)$ are also confluent at $(3,1)$.
 By the above, we call these two orbits of generic residual 
 points $W_0r_{sub,triv}$ and $W_0r_{sub,\sigma}$, where 
$\sigma$ is the two dimensional irreducible character of $S_3$. 
 The orbit of generic points $W_0r_{sub,triv}$
represents a generic discrete series character of degree $3$, which 
has generic formal degree with rational constant factor $\frac{1}{2}$.
The other orbit of generic residual points $W_0r_{sub,\sigma}$ has degree $1$, and 
generic rational constant $1$. At the confluence of these two generic residual points at 
parameter $(1,1)$, we get in the limit an additional constant factor $\frac{3}{1}$  
for $W_0r_{sub,triv}$ leading to the well known equal parameter case of Theorem 
\ref{thm:HII} at the subregular unipotent orbit  
for split $\textup{G}_2$ (cf. \cite{Re0}). At the confluence point for the parameters 
$(3,1)$ the rational constants do not change, however. Thus together with the cuspidal character 
${}^3\textup{D}_4[1]$ we get a packet $\Pi_\lambda$ for $\lambda=\lambda_{sub}$ consisting  
of three representations, naturally parameterized 
by the characters of $A_{\lambda}=S_3$, whose formal degrees have rational constant 
$\frac{1}{2}$ (for the cuspidal ${}^3\textup{D}_4[1]$ corresponding to the ``missing representation"  
$sign$ of $A_\lambda=S_3$, and for the generic discrete series character associated 
to $W_0r_{,sub,triv}$ evaluated at the parameter value $(3,1)$), and  rational constant $1$ 
(for the generic discrete series $W_0r_{sub,\sigma}$ evaluated at $(3,1)$).

For the regular parameter of $G^\vee_{\theta s_1}$ we get two discrete 
series characters, namely the cuspidal one ${}^3\textup{D}_4[-1]$ and the Iwahori 
spherical one. Both have $\frac{1}{2}$ as a rational constant factor.

Finally, at the regular parameter of $G^\vee_{\theta s_2}$, we have 
one Iwahori spherical discrete series representation, with rational constant 
$1$. 

These constants are clearly compatible with Theorem \ref{thm:HII}. 
Namely, consider (\ref{eq:HIIconv}). For a discrete Langlands parameters $\lambda$ 
with $\theta s_i=\lambda(F)$ and such that $u:=\lambda \begin{pmatrix} 1\ 1 \\ 0\ 1\end{pmatrix}$ 
is regular within the connected reductive group $G^\vee_{\theta s_i}$ this follows 
because $C_\lambda^F=1$ in such a case (this is obvious for $i=0$ and $i=2$ 
since then $A_\lambda=1$, and for $i=1$ we see that $u$ is the distinguished element  
$[5,3]$ in $\textup{Spin}(8)$ by the table of \cite[page 397]{C}, whence $M_\lambda^0=1$), and hence 
$A_\lambda\approx (\pi_0(M_\lambda))^F$ is isomorphic to 
$Z(G^\vee_{\theta s_i})/Z({}^LG)=Z(G^\vee_{\theta s_i})$ (by \cite[Section 6]{Re0}). This yields the result for all cases 
except $\lambda_{sub}$. For this case we remark that the image of the subregular unipotent 
of $G_2$ in $\textup{Spin}(8)$ is a unipotent class of $\textup{Spin}(8)$ 
with  elementary divisors 
$[1,1,3,3]$. Hence $M_\lambda^0$ is a two dimensional torus, on which $F$ 
acts as a rotation of order $3$. Thus $C_\lambda^F$ is cyclic of order $3$, 
and $(\pi_0(M_\lambda))^F\approx C_2$. This indeed yields the constants 
we just computed.


\begin{thebibliography}{99}
\bibitem[A]{A}J. Arthur, 
``A Note on L-packets", 
Pure and App. Math Q., {\bf 2} (2006), pp. 199Ð217.

\bibitem[BD]{BD}
Bernstein, J.,  Deligne, P., 
``Le ÓcentreÓ de Bernstein'',
In ÓRepresentations des groups redutifs sur un corps local,
Traveaux en coursÓ (P.Deligne ed.), Hermann, Paris (1984), 1-32.

\bibitem[Bo1]{Bo} Borel, A.,
``Admissible representations of a semisimple group over a local field
with vectors fixed under an Iwahori subgroup'',
\emph{Invent. Math.} {\bf 35} (1976), 233--259.

\bibitem[Bo2]{Bo2} Borel, A.
``Automorphic $L$-functions",
Proceedings of Symposia in Pure Mathematics,
Vol. {\bf{33}} (1979), part 2, pp. 27--61.

\bibitem[BT]{BT} Bruhat, F., Tits, J.,
``Groupes reductifs sur un corps local,''
\emph{Inst. Hautes  \'Etudes Sci. Publ. Math.} {\bf 41} (1972), 5�251; 60 (1984),
pp. 197�376.

\bibitem[BT2]{BT2} Bruhat, F., Tits, J.,
``Groupes r\'eductifs sur un corps local. II. Sch\'emas en groupes. Existence d'une donn\'ee
radicielle valu\'ee'',
\emph{Inst. Hautes Etudes Sci. Publ. Math.}, No. {\bf 60} (1984),
pp. 197--376.

\bibitem[BT3]{BT3} Bruhat, F., Tits, J.,
``Groupes alg\'ebriques sur un corps local Chapitre III.
Compl\'ements et applications `a la cohomologie galoisienne'',
J. Fac. Sci. Univ. Tokyo, 34 (1987), pp. 671-698

\bibitem[BKH]{BHK} Bushnell, C.J., Henniart, G., Kutzko, P.C.,
``Towards an explicit Plancherel formula for $p$-adic reductive
groups'', preprint (2005).

\bibitem[BK]{BK1} Bushnell, C.J., Kutzko, P.C.,
``Types in reductive $p$-adic groups: the Hecke algebra
of a cover'',
\emph{Proc. Amer. Math. Soc.} {\bf 129} (2001),
no. 2, 601--607.

\bibitem[Ca]{C} Carter, R.W.,
``Finite groups of Lie type'',
\emph{Wiley Classics Library}, John Wiley
and sons, Chichester, UK, 1993.

\bibitem[C]{Ch2} Cherednik, I.V.,
``Double affine Hecke algebras and Macdonald's conjectures'',
\emph{Annals of Math.} {\bf 141} (1995), 191--216.


\bibitem[CK]{CK} Ciubotaru, D., Kato, S., 
``Tempered modules in exotic Deligne-Langlands correspondence",
 Adv. Math. 226 (2011), no. 2, 1538Ð1590.
 
 \bibitem[CKK]{CKK}
 Ciubotaru, Dan; Kato, Midori; Kato, Syu, 
``On characters and formal degrees of discrete series of affine 
 Hecke algebras of classical types", 
 Invent. Math. 187 (2012), no. 3, 589Ð635;
 math.RT.1001.4312.
 
 \bibitem[CO1]{CiuOpd1}
 Ciubotaru, D., Opdam, E.M., 
``Formal degrees of unipotent discrete series representations and the exotic 
 Fourier transform", to appear in PLMS, 
math.RT.1310.3546.

\bibitem[COT]{COT}
Ciubotaru, D., Opdam, E.M., Trapa, P.,
``Algebraic and analytic Dirac induction for graded affine Hecke algebras'',
{\it J. Inst. Math. Jussieu} {\bf 13} (2014), no. 3, 447--486.

\bibitem[CO2]{CiuOpd2} Ciubotaru, D., Opdam, E.M., 
``A uniform classification of the discrete series representations 
of affine Hecke algebras", 
\emph{In preparation}.

\bibitem[CM]{CM} Collingwood, D.H., McGovern, W.M., 
``Nilpotent orbits in semisimple Lie algebras",
Van Nostrand Reinhold mathematics series, New York, 1993.
 
\bibitem[CH]{CH} Corwin, L., and Howe, Roger E., 
``Computing characters of tamely ramified $p$-adic 
division algebras", 
\emph{Pacific Journal of Mathematics} {\bf 73}(2) (1977), pp. 461--477

\bibitem[DR]{DeRe} Debacker, S., Reeder, M.,
``Depth zero supercuspidal L-packets and their stability''
 Ann. of Math. (2) {\bf 169} (2009), no. 3, 795Ð901.

\bibitem[DO1]{DeOp1} Delorme, P.,  Opdam, E.M.,
``The Schwartz algebra of an affine Hecke algebra",
 Journal fuer die reine und angewandte Mathematik
                {\bf 625} (2008), 59--114.

\bibitem[DO2]{DeOp2} Delorme, P., E.M. Opdam,
``Analytic R-groups of affine Hecke algebras",
       J. Reine Angew. Math. {\bf 658} (2011), 133-172.

\bibitem[Dix]{Dix} Dixmier, J.,
``Les $C^*$-alg\`ebres et leurs representations'',
\emph{Cahiers Scientifiques} \textbf{29}, Gauthier-Villars \'Editeur,
Paris, France, 1969.

\bibitem[EOS]{EOS} Emsiz, E., Opdam, E.M., Stokman, J.V.,
``Periodic integrable systems with delta-potentials'',
\emph{Comm. Math. Phys.} {\bf 264}(1) (2006), 191-225.

\bibitem[FO]{FO}
Feng, Y., Opdam, E.M., 
``On a uniqueness property of cuspidal unipotent representations",
to appear.


\bibitem[G]{Geck} Geck, M.,
``On the representation theory of Iwahori-Hecke
algebras of extended finite Weyl groups'',
\emph{Representation Theory} \textbf{4} (2000), 370-397.


\bibitem[GR]{GR} Gross, B., Reeder, M., 
``Arithmetic invariants of discrete Langlands parameters",
 \emph{Duke Math. J.} \textbf{154} (2010), no. 3, 431Ð-508.


\bibitem[HR]{HR} Haines, T., Rapoport, M.,
``On parahoric subgroups'',
math.RT.0804.3788 (2008)

\bibitem[HO1]{HO} Heckman, G.J., Opdam, E.M.,
``Yang's system of particles and Hecke algebras'',
\emph{Ann. of Math.} \textbf{145} (1997), 139-173.

\bibitem[HO2]{HOH} Heckman, G.J.,  Opdam, E.M.,
``Harmonic analysis for affine Hecke algebras'',
\emph{Current Developments in Mathematics} (S.-T. Yau, editor),
1996, Intern. Press, Boston.

\bibitem[H1]{H1} Heiermann, V., 
``Op\'erateurs d'entrelacement et alg\`ebres de Hecke avec param\`etres d'un groupe r\'eductif p-adique;  
le cas des groupes classiques'',
\emph{Selecta Mathematica}, New Series, \textbf{17}, Number 3 (2011), 713-756.

\bibitem[H2]{H2} Heiermann, V.,  
``Paramtres de Langlands et Algbres d'entrelacement'', 
\emph{Int. Math. Res. Notices}, Vol. 2010 (9) (2010), 1607Ñ1623.

\bibitem[H3]{H3} Heiermann, V.,
``Local Langlands Correspondence for Classical Groups and Affine Hecke Algebras'', 
arXiv:1502.04357v1, (2015), 26 pages.

\bibitem[HII]{HII} 
Hiraga, K., Ichino, A., Ikeda, T. 
``Formal degrees and adjoint gamma factors (and errata)", 
\emph{J. Amer. Math. Soc.}, \textbf{21} (2008), no. 1, 283-Ð304.

\bibitem[Hu]{Hum1} Humphreys, J.E.,
``Introduction to Lie algebras and Representation Theory'',
\emph{GTM} {\bf 9}, 3rd ed, Springer Verlag, 1980.

\bibitem[IM]{IwMa} Iwahori, N., Matsumoto, H.,
``On some Bruhat decomposition and the structure,
of the Hecke rings of the $p$-adic Chevalley groups'',
\emph{Inst. Hautes \'Etudes Sci. Publ. Math} \textbf{25} (1965), 5-48.

\bibitem[Ka]{Kat1} Kato, S.,
``An exotic Deligne-Langlands correspondence for symplectic groups'',
 Duke Math. J. {\bf 148} (2009), no. 2, 305Ð371.
 
 \bibitem[Ka2]{Kat2} Kato, S.,
 ``A homological study of Green polynomials",
 arXiv math RT.11114640(2013).
 
\bibitem[KL]{KL} Kazhdan, D., Lusztig, G.,
``Proof of the Deligne-Langlands conjecture for affine Hecke
algebras'',
\emph{Invent. Math.} {\bf 87} (1987), pp. 153--215.

\bibitem[Kne]{Kne} Kneser, M.,
``Galois-Kohomologie halbeinfacher algebraischer Gruppen
\"uber p-adischen K\"orpern. II''
\emph{Math. Zeitschr.} {\bf 89}, (1965) pp. 250--272

\bibitem[Kot1]{Ko1} Kottwitz, R.E.,
``Stable trace formula:
cuspidal tempered terms''
Duke Mathematical Journal, {\bf 51}, No. 3 (1984)
pp. 611-650

\bibitem[Kot2]{Ko2} Kottwitz, R.E.,
``Stable Trace Formula:
Elliptic Singular Terms''
Math. Ann. {\bf 275} (1986) pp. 365-399

\bibitem[Kot3]{Ko3} Kottwitz, R.E.,
``Isocrystals with additional structure. II'',
Compositio Mathematica {\bf 109} (1997), 255--339.

\bibitem[Lau]{Lau} Lauter, R.,
``Irreducible representations and the spectrum of
$\Psi^*$-algebras of pseudodifferential operators'',
(unpublished document).

\bibitem[Lu1]{Lu1} Lusztig, G., 
``Unipotent Characters of the Symplectic
and Odd Orthogonal Groups Over a Finite Field"
\emph{Invent. math.} {\bf 64}, (1981), 263--296.

\bibitem[Lu2]{L1}
G.~Lusztig,
``Characters of reductive groups over a finite field'',
Ann.~Math. Studies {\bf 107}, Princeton Univ. Press 1984.

\bibitem[Lu3]{Lu1.5} Lusztig, G., 
``Intersection cohomology complexes on a reductive group"
\emph{Invent. math.} {\bf 75}, (1984), 205--272.

\bibitem[Lu4]{Lu2} Lusztig, G., 
``Unipotent representations of a finite Chevalley group 
of type $\textup{E}_8$", 
\emph{Quart. J. Math.} {\bf 30}, (1979), 315--338.  

\bibitem[Lu5]{Lu3} Lusztig, G., 
``On the Unipotent Characters
of the Exceptional Groups Over Finite Fields", 
\emph{Invent. math.} {\bf 60}, (1980), 173- 192. 

\bibitem[Lu6]{Lu0}Lusztig, G.,
``Some examples of square integrable representations of
semisimple p-adic groups'', T.A.M.S. {\bf 277} (1983) no. 2, pp.\
623--653.

\bibitem[Lu7]{Lu10} G. Lusztig, 
``Leading coefficients of character values of Hecke algebras",
Proc. Symp. Pure Math., {\bf 47}, 
Amer. Math. Soc., Providence, RI (1987), pp. 235Ð262.

\bibitem[Lu8]{Lus2} Lusztig, G.,
``Affine Hecke algebras and their graded version'',
\emph{J. Amer. Math. Soc} \textbf{2} (1989), 599--635.

\bibitem[Lu9]{Lu4} Lusztig, G.,
``Classification of unipotent representations of simple
$p$-adic groups'',
\emph{Internat. Math. Res. Notices} {\bf 11} (1995), 517--589.

\bibitem[Lu10]{Lusz4} Lusztig, G.,
``Cuspidal local systems and graded Hecke algebras IIÓ 
in Representations of Groups, 
ed. by B. Allison and G. Cliff, Conf. Proc. Canad. Math. Soc. {\bf 16}, 
Amer. Math. Soc., Providence, 1995, 217Ð275.

\bibitem[Lu11]{Lus3} Lusztig, G.,
``Cuspidal local systems and graded Hecke algebras, III'',
\emph{Representation Theory}, \textbf{6} (2002), 202--242.

\bibitem[Lu12]{Lu5} Lusztig, G.,
``Hecke algebras with Unequal Parameters'',
\emph{CRM Monograph Series} {\bf 18},
Amer. Math. Soc., Providence RI, 2003.

\bibitem[Lu13]{Lu6} Lusztig, G., 
``Classification of unipotent representations 
of simple p-adic groups. II'', 
\emph{Represent. Theory}, \textbf{6} (2002), 243-289.

\bibitem[Lu14]{L4}
G.~Lusztig,
``Unipotent almost characters of simple $p$-adic groups'',
preprint, \texttt{arXiv:1212.6540}.

\bibitem[Lu15]{L5}
G.~Lusztig,
``Unipotent almost characters of simple $p$-adic groups II'',
preprint, \texttt{arXiv:1303.5026}.

\bibitem[LuSpa]{LuSpa} Lusztig, G., and Spaltenstein, N., 
``On the Generalized Springer Correspondence for classical groupsÓ, 
Advanced Studies in Pure Mathematics {\bf 6}, 1985, 
Algebraic Groups and Related topic, 
289Ð316. 

\bibitem[Ma1]{Ma1}
``Spherical functions on a group of $p$-adic type'',
\emph{Publ. Ramanujan Institute} {\bf 2} (1971).

\bibitem[Ma2]{Ma2} Macdonald, I.G.,
``Affine Hecke algebras and orthogonal polynomials'',
\emph{Cambridge Tracts in Mathematics} {\bf 157},
Cambridge University Press, 2003.

\bibitem[Mat]{Mat} Matsumoto, H.,
``Analyse harmonique dans les systems de Tits bornologiques de
type affine'',
\emph{Springer Lecture Notes} {\bf 590} (1977).

 \bibitem[Moe]{Moe} Moeglin, C.  
``Stabilit\'e pour les repr\'esentations elliptiques de r\'eduction unipotente: le cas des groupes unitaires", 
in Automorphic representations, L-functions and applications: progress and prospects, 
361Ð402, Ohio State Univ. Math. Res. Inst. Publ., 11, de Gruyter, Berlin, 2005. 

\bibitem[MW]{MW} Moeglin, C., Waldspurger, J.-L.,
``Pacquets stables de repr\'esentations temp\'er\'es et de
reduction unipotente pour SO(2n + 1)"
Invent. Math. 152 (2003), 461Ð623.

\bibitem[Mo1]{Mo1} Morris, L.,
``Tamely ramified intertwining algebras'',
\emph{Invent. Math.} {\bf 114} (1993), 233--274.

\bibitem[Mo2]{Mo2} Morris, L.,
``Level zero $G$-types'',
\emph{Compositio Math.} {\bf 118}(2) (1999), 135--157.

\bibitem[O1]{Opd0} Opdam, E.M.,
``A generating function for the trace of the
Iwahori-Hecke algebra'',
in ``Studies in memory of Issai Schur",
Joseph, Melnikov and Rentschler (eds.), Progress in Math.,
Birkhauser (2002), 301--324.

\bibitem[O2]{Opd1} Opdam, E.M.,
``On the spectral decomposition of affine Hecke algebras'',
J. Inst. Math. Jussieu \textbf{3}(4) (2004), 531-648.

\bibitem[O3]{Opd2} Opdam, E.M.,
``Hecke algebras and harmonic analysis'',
pp. 1227-1259 in: \emph{Proceedings of the International Congress
of Mathematicians - Madrid, August 22-30, 2006. Vol. II},
EMS Publ. House (2006).

\bibitem[O4]{Opd3} Opdam, E.M.,
``The central support of the Plancherel measure of an 
affine Hecke algebra'',
\emph{Moscow Mathematical Journal} {\bf 7}(4) (2007), 723--741.

\bibitem[O5]{Opd4} Opdam, E.M.,
``Spectral correspondences for affine Hecke algebras",
arxiv: math.RT.1310.7193 (2013).

\bibitem[OS1]{OpdSol} Opdam E.M., and Solleveld, M.S.,
``Homological algebra for affine Hecke algebras'',
 Advances in Mathematics {\bf 220} (2009), pp. 1549-1601.


\bibitem[OS2]{OpdSol2}
Opdam, E.M., Solleveld, M.S.,
``Discrete series characters for affine Hecke algebras
and their formal degrees'',
 Acta Math. {\bf 205} (2010), no. 1, 105-187.

\bibitem[RR]{RamRam} Ram, A., Ramagge, J.,
``Affine Hecke algebras, cyclotomic Hecke algebras and Clifford
theory'', 
A tribute to C. S. Seshadri (Chennai, 2002), 428Ð466, 
Trends Math., BirkhŠuser, Basel, 2003.

\bibitem[Re1]{Re0}M. Reeder,
``On the Iwahori spherical discrete series of $p$-adic
Chevalley groups; formal degrees and $L$-packets", 
Ann. Sci. Ec. Norm. Sup. {\bf 27} (1994), pp.\ 463--491.

\bibitem[Re2]{Re} Reeder, M.,
``Formal degrees and L-packets of unipotent
discrete series of exceptional $p$-adic groups'',
with an appendix by Frank L\"ubeck,
\emph{J. reine angew. Math.} \textbf{520} (2000), 37--93.

\bibitem[Re3]{Ree} Reeder, M.,
``Euler-Poincar\'e pairings and
elliptic representations of Weyl groups and $p$-adic groups'',
\emph{Compositio Math.} \textbf{129} (2001), 149--181.

\bibitem[Re4]{Reed} Reeder, M.,
``Torsion automorphisms of simple Lie algebras", 
Enseign. Math. (2) {\bf 56} (2010), no. 1-2, 3Ð47.

\bibitem[ScSt]{ScSt} Schneider, P., Stuhler, U.,
``Representation theory and sheaves on the Bruhat-Tits building'',
\emph{Publ. Math. Inst. Hautes \'Etudes Sci.} \textbf{85} (1997), 97-191.

\bibitem[Se1]{Se1} Serre, J.P.,
``Local fields'', Springer-Verlag, 1979.

\bibitem[Se2]{Se2} Serre, J.P.,
Galois cohomology, Springer-Verlag, 2002.

\bibitem[Slo1]{SlootenThesis} Slooten, K.,
``A combinatorial
generalization of the Springer correspondence for classical type'',
PhD Thesis, Universiteit van Amsterdam (2003).

\bibitem[Slo2]{Slooten} Slooten, K.,
``Generalized Springer correspondence and Green functions for type B/C
graded Hecke algebras'',
\emph{Advances in Mathematics} \textbf{203} (2006), 34-108.

\bibitem[Sol]{Sol} Solleveld, M.S.,
``Periodic cyclic homology of affine Hecke algebras'',
PhD Thesis, Universiteit van Amsterdam (2007).

\bibitem[Sol2]{Sol2} Solleveld, M.,  
``Homology of graded Hecke algebras." 
J. Algebra {\bf 323} (2010), no. 6, 1622Ð1648. 

\bibitem[Spa]{Spa} Spaltenstein, N., 
``On the Generalized Springer Correspondence for Exceptional GroupsÓ, 
Advanced Studies in Pure Mathematics {\bf 6}, 1985, 
Algebraic Groups and Related topic, 
317Ð338. 

\bibitem[St]{Stein} Steinberg, R.,
``Lecures on Chevalley groups'',
Yale University, 1967.

\bibitem[Ta]{Tate} Tate, J.,
``Number theoretic background'',
Proceedings of Symposia in Pure Mathematics,
Vol. {\bf{33}} (1979), part 2, pp. 29--69.

\bibitem[T]{Tits} Tits, J.,
``Reductive groups over local fields'',
Proceedings of Symposia in Pure Mathematics,
Vol. {\bf{33}} (1979), part 1, pp. 29--69.

\bibitem[V]{V} Vogan, D.,
``The local Langlands conjecture'',
Representation Theory of groups and algebras,
Contemp. Math. {\bf 145},
pp 305-379, AMS, Providence, RI, 1993.


\bibitem[W]{W} Waldspurger, J.-L.,
``La formule de Plancherel pour les groupes p-adiques'' (d\'\,apr\`es Harish-Chandra),
\emph{J. Inst. Math. Jussieu} {\bf 2} (2003), 235-333.
\end{thebibliography}
\end{document}